\documentclass[final,3p]{elsarticle}
 \usepackage{graphics}
 \usepackage{graphicx}
 \usepackage{epsfig}
\usepackage{amssymb}
 \usepackage{amsthm}
 \usepackage{lineno}
 \usepackage{amsmath}
   \numberwithin{equation}{section}
\usepackage{mathrsfs}
\newtheorem{thm}{Theorem}[section]

\newtheorem{lem}[thm]{Lemma}
\newtheorem{prop}[thm]{Proposition}
\newtheorem{defn}[thm]{Definition}

\biboptions{sort&compress,square}
\allowdisplaybreaks

\begin{document}
\begin{frontmatter}
\author[rvt1,rvt2]{Jin Hong}
\ead{jhong@nenu.edu.cn}
\author[rvt2]{Yong Wang\corref{cor2}}
\ead{wangy581@nenu.edu.cn}
\cortext[cor2]{Corresponding author.}
\address[rvt1]{School of Mathematics and Statistics, Yili Normal University, Yining 835000, China}
\address[rvt2]{School of Mathematics and Statistics, Northeast Normal University, Changchun, 130024, China}

\title{The spectral Einstein functional for the Dirac operator with torsion  }
\begin{abstract}
In this paper, we compute the spectral Einstein functional associated with the Dirac operator with torsion on even-dimensional spin manifolds without boundary.
\end{abstract}
\end{frontmatter}
\section{Introduction}
 Noncommutative residues are an important tool for studying noncommutative geometry. Therefore, much attention has been paid to noncommutative residues.
 In \cite{Co1}, Connes derived the analog of the conformal four-dimensional Polyakov action using noncommutative residues. Connes proved that the noncommutative residue on a compact manifold $M$ coincides with the Dixmier trace on a pseudo-differential operator of order -dim$M$ \cite{Co2}. Furthermore, Connes claimed that the noncommutative residue of the inverse square of the Dirac operator is equivalent to the Einstein-Hilbert action on \cite{Co2, Co3}. Kastler gave a direct proof of this theorem in \cite{Ka}, while Kalau and Walze obtained it simultaneously in \cite{KW} in the normal coordinate system. Based on the theory of noncommutative residues proposed by Wodzicki, Fedosov et al. \cite{FGLS} the noncommutative residues of the classical elemental algebra of Boutet de Monvel's calculus are constructed on compact manifolds of dimension $n>2$. Using elliptic pseudo-differential operators and noncommutative residues is a natural way to study the spectral Einstein functional and the operator-theoretic interpretation of the gravitational action on bounded manifolds. Concerning the Dirac operator and the signature operator, Wang carried out the computation of noncommutative residues and succeeded in proving the Kastler-Kalau-Walze type theorem for manifolds with boundary \cite{Wa1, Wa3, Wa4}.
 
Many noncommutative residues of the Dirac operator are studied \cite{Wa1,Wa3,Wa4,WWJ,WWw}. In \cite{FGV2}, Figueroa et al. introduced a noncommutative integral based on the noncommutative residue \cite{wo2}. In \cite{Ac2}, Ackermann and Tolksdorf proved a generalized version of the well-known Lichnerowicz formula for the square of the most general Dirac operator with torsion $D_T$ on an even-dimensional spin manifold associated with a metric connection with torsion.
Wang et al. \cite{WWJ} gave two kinds of operator-theoretic explanations of the
gravitational action of Dirac operators with torsion in the case of 4-dimensional compact manifolds with flat boundary. In \cite{WWw}, Wang et al. gave some new spectral functionals which is the extension of spectral functionals to the noncommutative realm with torsion, and related them to the noncommutative residue for manifolds with boundary about Dirac operators with torsion. 
In \cite{DL}, Dabrowski et al. defined the spectral Einstein functional for a general spectral triple. For the noncommutative torus, they computed the spectral Einstein functional for the Dirac operator. Pf${\mathrm{\ddot{a}}}$ffle and Stephan \cite{pf1} considered orthogonal connections with arbitrary torsion on
compact Riemannian manifolds and computed the spectral action. In \cite{pf2}, Pf${\mathrm{\ddot{a}}}$ffle and Stephan considered compact Riemannian spin manifolds without boundary equipped with orthogonal connections, and investigated the induced Dirac operators. Iochum et al. \cite{ILV} derived a commutative spectral triple and studied the spectral action for a rather general geometric setting, which includes the (skew-symmetric) torsion and the chiral bag conditions on the boundary. In \cite{pf3},  Pf${\mathrm{\ddot{a}}}$ffle and Stephan gave a Lichnerowicz type formula for the Dirac operator with torsion. Hanisch et al. \cite{Ha} derived a formula for the gravitational part of the spectral action for Dirac operators on 4-dimensional manifolds with totally anti-symmetric torsion. 
In \cite{DL2}, Dabrowski et al. examined the metric and Einstein bilinear functionals of differential forms for the Hodge–Dirac operator on an oriented, closed, even-dimensional Riemannian manifold.
So, it is natural to consider the spectral Einstein functional by the Dirac operator with torsion. The motivation of this paper is to compute the spectral Einstein functional associated with the Dirac operator with torsion on even-dimensional spin manifolds without boundary.

This paper is organized as follows. In Section 2, we give a brief exposition of the Dirac operator with torsion. Using
the results in Sec.2, we obtain spectral Einstein functional associated with the Dirac operator with torsion on even-dimensional spin manifolds without boundary in the next section.

\section{The Dirac operator with torsion }
We give some definitions and basic notions which we will use in this paper.

Let $M$ be a smooth compact oriented Riemannian $n$-dimensional manifolds without boundary and $N$ be a vector bundle on $M$.
We say that $P$ is a differential operator of Laplace type, if it has locally the form
\begin{equation}\label{p}
	P=-(g^{ij}\partial_i\partial_j+A^i\partial_i+B),
\end{equation}
where $\partial_{i}$  is a natural local frame on $TM,$ $(g^{ij})_{1\leq i,j\leq n}$ is the inverse matrix associated to the metric
matrix  $(g_{ij})_{1\leq i,j\leq n}$ on $M,$ $A^{i}$ and $B$ are smooth sections of $\textrm{End}(N)$ on $M$ (endomorphism).
If $P$ satisfies the form \eqref{p}, then there is a unique
connection $\nabla$ on $N$ and a unique endomorphism $E$ such that
\begin{equation}
	P=-[g^{ij}(\nabla_{\partial_{i}}\nabla_{\partial_{j}}- \nabla_{\nabla^{L}_{\partial_{i}}\partial_{j}})+E],\nonumber
\end{equation}
where $\nabla^{L}$ is the Levi-Civita connection on $M$. Moreover
(with local frames of $T^{*}M$ and $N$), $\nabla_{\partial_{i}}=\partial_{i}+\omega_{i} $
and $E$ are related to $g^{ij}$, $A^{i}$ and $B$ through
\begin{eqnarray}
	&&\omega_{i}=\frac{1}{2}g_{ij}\big(A^{i}+g^{kl}\Gamma_{ kl}^{j} id\big),\nonumber\\
	&&E=B-g^{ij}\big(\partial_{i}(\omega_{j})+\omega_{i}\omega_{j}-\omega_{k}\Gamma_{ ij}^{k} \big),\nonumber
\end{eqnarray}
where $\Gamma_{ kl}^{j}$ is the  Christoffel coefficient of $\nabla^{L}$.

Let $\nabla^{T}$ denote the metric connection
\begin{equation}
		\langle \nabla_{X}^{T} Y, Z \rangle=\langle \nabla_{X}^{L} Y, Z \rangle + T(X, Y, Z),\nonumber
\end{equation}
where $T$ is a there form.

Let $M$ be an $n=2m$ dimensional ($n\geq 3$) spin  manifold, we can lift $\nabla^T$ to $\nabla^{S(TM),T}$ on $S(TM)$.
The Dirac operator with torsion $D_T$ is defined as:
\begin{align}
	D_T=&\sum_{j=1}^{n}c(e_{j})\nabla_{e_{j}}^{S(TM),T}\nonumber\\
		=&\sum_{j=1}^{n}c(e_{j})\bigg(e_{j}+\frac{1}{4}\sum_{l, t=1}^{n}\langle \nabla_{e_{j}}^{T}e_{l}, e_{t}\rangle c(e_{l})c(e_{t})\bigg)\nonumber\\
		=&\;D+\frac{1}{4}\sum_{j,l,t=1}^{n}T(e_j, e_l, e_t)c(e_{j})c(e_{l})c(e_{t})\nonumber\\
		=&\;D+\frac{3}{2}\sum_{1\leqslant j< l< t\leqslant n}T(e_j, e_l, e_t)c(e_{j})c(e_{l})c(e_{t}),\nonumber
\end{align}
where $D$ is the Dirac operator induced by the Levi-Civita connection, and $c(e_{j})$ be the Clifford action which satisfies the relation
\begin{align}
	&c(e_{i})c(e_{j})+c(e_{j})c(e_{i})=-2g^{M}(e_{i}, e_{j})=-2\delta_i^j.\nonumber
\end{align}

By lemma 2.1 in \cite{WWw} ,we get
\begin{align} \label{dt2}
D_{T}^{2}=-g^{ab}(\nabla_{{\partial}_a}\nabla_{{\partial}_b}-\nabla_{{\nabla_{{\partial}_a}^{L}}\partial_b})+E,
\end{align}
where
\begin{align} \label{e}
	E=\frac{3}{2} d T+\frac{1}{4} s-\frac{3}{4}\|T\|^{2},
\end{align}
and
\begin{align} \label{a}
	\nabla_{{\partial}_a}={\partial}_a-\widetilde{T}_a={\partial}_a+\frac{1}{4}\sum_{s, t=1}^{n}\langle \nabla_{{\partial}_a}^{L}e_{s}, e_{t}\rangle c(e_{s})c(e_{t})+\frac{3}{2}\sum_{1\leqslant j<l\leqslant n}T({\partial}_a, e_j, e_l) c(e_{j})c(e_{l}).
\end{align}
So we have
\begin{align} \label{ta1}
	\widetilde{T}_a=-\frac{1}{4}\sum_{s, t=1}^{n}\langle \nabla_{{\partial}_a}^{L}e_{s}, e_{t}\rangle c(e_{s})c(e_{t})-\frac{3}{2}\sum_{1\leqslant j<l\leqslant n}T({\partial}_a, e_j, e_l) c(e_{j})c(e_{l}).
\end{align}
In normal coordinates, $\widetilde{T}_a$ is expanded near $x=0$ by Taylor expansion. By ${\partial}_{l}\langle \nabla_{{\partial}_a}^{L}e_{s}, e_{t}\rangle(x_0)=\frac{1}{2}{\rm R}_{lats}(x_0)$ we get
\begin{align} \label{ta2}
	\widetilde{T}_a={T}_a+{T}_{ab}x^b+O(x^2),
\end{align}
where
\begin{align} \label{ta0}
	{T}_a=-\frac{3}{2}\sum_{1\leqslant j<l\leqslant n}T({\partial}_a, e_j, e_l)(x_0) c(e_{j})c(e_{l}),
\end{align}
and
\begin{align} \label{tab}
	{T}_{ab}=-\frac{1}{8}\sum_{s, t=1}^{n}\mathrm{R}_{bats}(x_0) c(e_{s})c(e_{t})-\frac{3}{2}\sum_{1\leqslant j<l\leqslant n}\frac{\partial T({\partial}_a, e_j, e_l)}{{\partial}x_b}(x_0) c(e_{j})c(e_{l}).
\end{align}

 \begin{lem}\label{lem1}{\rm \cite{DL}}
	The following identities hold:
	\begin{align}\label{sigma0} 
		\sigma_{-2 m}(\Delta_{T,E}^{-m})=&\|\xi\|^{-2 m-2}\sum_{a,b,j,k=1}^{2m}\left(\delta_{a b}-\frac{m}{3} {\rm R}_{a j b k} x^{j} x^{k}\right) \xi_{a} \xi_{b}+O\left(\mathbf{x}^{3}\right) ;\\
		\sigma_{-2m-1}(\Delta_{T,E}^{-m})=& \frac{-2 m i}{3}\|\xi\|^{-2 m-2} \sum_{a,k=1}^{2m}\operatorname{Ric}_{a k} x^{k} \xi_{a}-2 m i\|\xi\|^{-2 m-2}\sum_{a,b=1}^{2m}\left(T_{a} \xi_{a}+T_{a b} x^{b} \xi_{a}\right)+O(\mathbf{x^2}) ;\\
		\sigma_{-2m-2}(\Delta_{T,E}^{-m})=& \frac{m(m+1)}{3}\|\xi\|^{-2 m-4} \sum_{a,b=1}^{2m}\operatorname{Ric}_{a b} \xi_{a} \xi_{b} \\
		&-2 m(m+1)\|\xi\|^{-2 m-4}\sum_{a,b=1}^{2m} T_{a} T_{b} \xi_{a} \xi_{b}+m\sum_{a,b=1}^{2m}\left(T_{a} T_{a}-T_{a a}\right)\|\xi\|^{-2 m-2}\nonumber \\
		&+2 m(m+1)\|\xi\|^{-2 m-4} \sum_{a,b=1}^{2m}T_{a b} \xi_{a} \xi_{b}-mE \|\xi\|^{-2m-2}+O(\mathbf{x}) ,\nonumber
	\end{align}
where ${\rm R}_{a j b k}$ and $\operatorname{Ric}_{a k}$ are the components of the Riemann and Ricci tensor.
\end{lem}

By \eqref{e},\eqref{ta0},\eqref{tab} and lemma \ref{lem1}, we get the following lemma.

\begin{lem}\label{lem2}
	General dimensional symbols about Dirac operator with torsion are given,
	\begin{align} 
		\sigma_{-2 m}(\Delta_{T,E}^{-m})=&\|\xi\|^{-2 m-2}\sum_{a,b,j,k=1}^{2m} \left(\delta_{a b}-\frac{m}{3} R_{a j b k} x^{j} x^{k}\right) \xi_{a} \xi_{b}+O\left(\mathbf{x}^{3}\right);\label{lem21} \\
		\sigma_{-2m-1}(\Delta_{T,E}^{-m})=& \frac{-2 m \sqrt{-1}}{3}\|\xi\|^{-2 m-2}\sum_{a,b=1}^{2m}  \operatorname{Ric}_{a b} x^{b} \xi_{a}\label{lem22}\\
		&+3m\sqrt{-1}\|\xi\|^{-2 m-2}\sum_{a,j<l=1}^{2m}T(e_a,e_j,e_l)(x_0)c(e_j)c(e_l)\xi_{a} \nonumber\\
		&+\frac{ m \sqrt{-1}}{4}\|\xi\|^{-2 m-2} \sum_{a,b,t,s=1}^{2m} \operatorname{R}_{b a t s}(x_0) c(e_s) c(e_t) x^{b} \xi_{a}\nonumber\\
		&+3m\sqrt{-1}\|\xi\|^{-2 m-2}\sum_{a,b,j<l=1}^{2m}\frac{\partial{T}(\partial_{a},e_j,e_l)}{\partial{x_b}}(x_0)c(e_j)c(e_l) x^b \xi_a +O\left(\mathbf{x}^{2}\right);\nonumber\\
		\sigma_{-2m-2}(\Delta_{T,E}^{-m})=& \frac{m(m+1)}{3}\|\xi\|^{-2 m-4}\sum_{a,b=1}^{2m} \operatorname{Ric}_{a b} \xi_{a} \xi_{b}\label{lem23} \\
		&-\frac{9m(m+1)}{2}\|\xi\|^{-2 m-4}\sum_{a,b,j<l,\hat{j}<\hat{l}=1}^{2m}T(e_a,e_j,e_l)(x_0)T(e_b, e_{\hat{j}},e_{\hat{l}})(x_0)c(e_j)c(e_l)c(e_{\hat{j}})c(e_{\hat{l}})\xi_{a}\xi_b \nonumber\\
		&+\frac{9m}{4}\|\xi\|^{-2 m-2}\sum_{a,j<l,\hat{j}<\hat{l}=1}^{2m}T(e_a,e_j,e_l)(x_0)T(e_a, e_{\hat{j}},e_{\hat{l}})(x_0)c(e_j)c(e_l)c(e_{\hat{j}})c(e_{\hat{l}}) \nonumber\\
		&+\frac{3m}{2}\|\xi\|^{-2 m-2}\sum_{a,j<l=1}^{2m}\frac{\partial{T}(\partial_{a},e_j,e_l)}{\partial{x_a}}(x_0)c(e_j)c(e_l)\nonumber\\
		&-\frac{ m (m+1)}{4}\|\xi\|^{-2 m-4} \sum_{a,b,t,s=1}^{2m}\operatorname{R}_{b a t s}(x_0) c(e_s) c(e_t) \xi_{a} \xi_{b}\nonumber\\
		&-3m(m+1)\|\xi\|^{-2 m-4}\sum_{a,b,j<l=1}^{2m}\frac{\partial{T}(\partial_{a},e_j,e_l)}{\partial{x_b}}(x_0)c(e_j)c(e_l)\xi_{a}\xi_b\nonumber\\
		&-m\|\xi\|^{-2 m-2}(\frac{1}{4} s-\frac{3}{4}\|T\|^2)\nonumber\\
		&-\frac{3m}{4}\|\xi\|^{-2 m-2}\sum_{ i<j<k<t=1}^{2m}dT(e_i,e_j,e_k,e_t)(x_0)c(e_i)c(e_j)c(e_k)c(e_t)+O\left(\mathbf{x}\right),\nonumber
	\end{align}
where ${\rm R}_{a j b k}$ and $\operatorname{Ric}_{a k}$ are the components of the Riemann and Ricci tensor, s is the scalar curvature.
\end{lem}

 \section{The spectral Einstein functional for the Dirac operator with torsion } 
For a pseudo-differential operator $P$, acting on sections of a spinor bundle over an even $n$-dimensional compact Riemannian spin manifold $M$, the analogue of the volume element in noncommutative geometry is the operator  $D^{-n}=: d s^{n} $. And pertinent operators are realized as pseudodifferential operators on the spaces of sections. Extending previous definitions by Connes \cite{co5}, a noncommutative integral was introduced in \cite{FGV2} based on the noncommutative residue \cite{wo2}, combine (1.4) in \cite{co4} and \cite{Ka}, using the definition of the residue:
\begin{align}\label{wres}
	\int P d s^{n}:=\operatorname{Wres} P D^{-n}:=\int_{S^{*} M} \operatorname{tr}\left[\sigma_{-n}\left(P D^{-n}\right)\right](x, \xi),
\end{align}
where  $\sigma_{-n}\left(P D^{-n}\right) $ denotes the  $(-n)$th order piece of the complete symbols of  $P D^{-n} $,  $\operatorname{tr}$  as shorthand of trace.
 This section is designed to get the metric functional and the spectral Einstein functional for the Dirac operator with torsion defined in \cite{DL}.
 
 \begin{defn}{\rm\cite{DL2} }
  If $(\mathscr{L}, D, \mathscr{H})$ is an $n\mbox{-}$summable spectral triple, let  $\Omega_{D}^{1}$ be the $\mathcal{L} $ bimodule of one-forms generated by  $\mathcal{L}$ and $[D, \mathcal{L}]$. Moreover, assume that there exists a generalised algebra of pseudodifferential operators which contains  $\mathcal{L}, D$, and $|D|^{\ell}$ for $\ell \in \mathbb{Z}$ with a tracial state $\mathrm{Wres}$ over this algebra (called a noncommutative residue), which identically vanishes on $ T|D|^{-k}$ for any $ k>n $ and a zero-order operator $ T $ (an operator in the algebra generated by  $\mathcal{L}$  and  $\Omega^{1}(\mathcal{L})$  ). Then, for  u, w $\in \Omega_{D}^{1}(\mathcal{L})$ , we call
  \begin{align}
 \mathscr{A}(u, w):=\mathrm{Wres}\left(u w|D|^{-n}\right),
  \end{align}
 the metric functional, and
   \begin{align}
 \mathscr{B}(u, w):=\mathrm{Wres}\left(u\{D, w\} D|D|^{-n}\right) .
  \end{align}
 the Einstein functional.
\end{defn}

  \begin{prop}
  Let $M$ be an $n=2m$ dimensional ($n\geq 3$) closed spin manifold and  $D_T$ a Dirac operator with torsion, then $(C^{\infty}(M),D_T,L^2(M,S(TM)))$ is an $n\mbox{-}$summable spectral triple.
 \end{prop}

 \begin{thm}\label{thm}
 	Let $M$ be an $n=2m$ dimensional ($n\geq 3$) spin  manifold, the metric functional $\mathscr{A}_{D_T}$ and the spectral Einstein functional $\mathscr{B}_{D_T}$ are equal to
 \begin{align}
 	\mathscr{A}_{D_T}&=\mathrm{Wres}\bigg(c(v)c(w){D_T^{-n}}\bigg)=-2^{m} \frac{2 \pi^{m}}{\Gamma\left(m\right)}\int_{M} g(v,w)d{\rm Vol}_M;\\
 	\mathscr{B}_{D_T}&=\mathrm{Wres}\bigg(c(v)\big(c(w)D_T+D_Tc(w)\big)D_T^{-n+1}\bigg)\\
 	&=\;2^{m} \frac{2 \pi^{m}}{\Gamma\left(m\right)}\int_{M}\biggl\{-\frac{1}{6}\mathbb{G}(v,w)-\frac{9}{2}\|T\|^2 g(v,w)+\frac{9}{2}\sum_{j,l=1}^{2m} T(v,e_{j},e_{l})T(w,e_{j},e_{l})\nonumber\\
 	&\;\;\;\;+\frac{3}{2}\sum_{a=1}^{2m} \nabla_{e_a}(T)(e_a,v,w)+3\sum_{j=1}^{2m} T(v,\nabla_{e_j}^{L}w,e_j)\biggr\}d{\rm Vol}_M,\nonumber
 \end{align}
 where  $g(v,w)=\sum_{a=1}^{n}v_{a} w_{a} $ and $ \mathbb{G}(v,w)=\operatorname{Ric}(v,w)-\frac{1}{2} s g(v,w)$, $v=\sum_{a=1}^{n} v_{a}e_{a}, w=\sum_{b=1}^{n} w_{b}e_{b}.$
 \end{thm}

\begin{proof}
The proof of the $\mathscr{A}_{D_T}$ formula for the measure function is obvious. Below we prove that the spectral Einstein function $\mathscr{B}_{D_T}$, splitting it into two parts as follows:
\begin{align}
	&\mathscr{B}_{1}=\mathrm{Wres}\bigg(c(v)c(w){D_T^{-n+2}}\bigg);\\
	&\mathscr{B}_{2}=\mathrm{Wres}\bigg(c(v)D_Tc(w) D_T D_T^{-n}\bigg).
\end{align}

{\bf Part I)} $\mathscr{B}_{1}=\mathrm{Wres}\bigg(c(v)c(w){D_T^{-n+2}}\bigg)$, $\mathcal{P}_{1}=c(v)c(w)$.
Let  $n=2 m $, by \eqref{wres}, we need to compute  $\int_{S^{*} M} \operatorname{tr}\left[\sigma_{-2 m}\left(\mathcal{P}_{1} D_T^{-2 m+2}\right)\right](x, \xi) $. Based on the algorithm yielding the principal symbol of a product of pseudo-differential operators in terms of the principal symbols of the factors,  by (2.14) in lemma \ref{lem2} we have

\begin{align}\label{d1}
	\sigma_{-2m}(D_T^{-2 m+2})=&\;\;\frac{m(m-1)}{3}\|\xi\|^{-2 m-2}\sum_{a,b=1}^{2m} \operatorname{Ric}_{a b} \xi_{a} \xi_{b} \\
	&-\frac{9m(m-1)}{2}\|\xi\|^{-2 m-2}\sum_{a,b,j<l,\hat{j}<\hat{l}=1}^{2m}T(e_a,e_j,e_l)(x_0)T(e_b, e_{\hat{j}},e_{\hat{l}})(x_0)c(e_j)c(e_l)c(e_{\hat{j}})c(e_{\hat{l}})\xi_{a}\xi_b \nonumber\\
	&+\frac{9(m-1)}{4}\|\xi\|^{-2 m}\sum_{a,j<l,\hat{j}<\hat{l}=1}^{2m}T(e_a,e_j,e_l)(x_0)T(e_a, e_{\hat{j}},e_{\hat{l}})(x_0)c(e_j)c(e_l)c(e_{\hat{j}})c(e_{\hat{l}}) \nonumber\\
	&+\frac{3(m-1)}{2}\|\xi\|^{-2 m}\sum_{a,j<l=1}^{2m}\frac{\partial{T}(\partial_{a},e_j,e_l)}{\partial{x_a}}(x_0)c(e_j)c(e_l)\nonumber\\
	&-\frac{ m (m-1)}{4}\|\xi\|^{-2 m-2} \sum_{a,b,t,s=1}^{2m}\operatorname{R}_{b a t s}(x_0) c(e_s) c(e_t) \xi_{a} \xi_{b}\nonumber\\
	&-3m(m-1)\|\xi\|^{-2 m-2}\sum_{a,b,j<l=1}^{2m}\frac{\partial{T}(\partial_{a},e_j,e_l)}{\partial{x_b}}(x_0)c(e_j)c(e_l)\xi_{a}\xi_b\nonumber\\
	&-(m-1)\|\xi\|^{-2 m}(\frac{1}{4} \operatorname{R}^{g}-\frac{3}{4}\|T\|^2)\nonumber\\
	&-\frac{3(m-1)}{4}\|\xi\|^{-2 m}\sum_{ i<j<k<t=1}^{2m}dT(e_i,e_j,e_k,e_t)(x_0)c(e_i)c(e_j)c(e_k)c(e_t)+O\left(\mathbf{x}\right).\nonumber
\end{align}
By \eqref{d1} and $\mathcal{P}_{1}=c(v)c(w)$, we get
\begin{align}\label{PD}
	\sigma_{-2 m}\left(\mathcal{P}_{1}  D^{-2 m+2}_T\right) & = \;\;\frac{m(m-1)}{3}\|\xi\|^{-2 m-2}\sum_{a,b=1}^{2m} \operatorname{Ric}_{a b} \xi_{a} \xi_{b}c(v)c(w) \\
	&-\frac{9m(m-1)}{2}\|\xi\|^{-2 m-2}\sum_{a,b,j<l,\hat{j}<\hat{l}=1}^{2m}T(e_b,e_j,e_l)(x_0)T(e_a, e_{\hat{j}},e_{\hat{l}})(x_0)c(v)c(w)c(e_j)c(e_l)c(e_{\hat{j}})c(e_{\hat{l}})\xi_{a}\xi_b \nonumber\\
	&+\frac{9(m-1)}{4}\|\xi\|^{-2 m}\sum_{a,j<l,\hat{j}<\hat{l}=1}^{2m}T(e_a,e_j,e_l)(x_0)T(e_a, e_{\hat{j}},e_{\hat{l}})(x_0)c(v)c(w)c(e_j)c(e_l)c(e_{\hat{j}})c(e_{\hat{l}}) \nonumber\\
	&+\frac{3(m-1)}{2}\|\xi\|^{-2 m}\sum_{a,j<l=1}^{2m}\frac{\partial{T}(\partial_{a},e_j,e_l)}{\partial{x_a}}(x_0)c(v)c(w)c(e_j)c(e_l)\nonumber\\
	&-\frac{ m (m-1)}{4}\|\xi\|^{-2 m-2} \sum_{a,b,t,s=1}^{2m}\operatorname{R}_{b a t s}(x_0)c(v)c(w) c(e_s) c(e_t) \xi_{a} \xi_{b}\nonumber\\
	&-3m(m-1)\|\xi\|^{-2 m-2}\sum_{a,b,j<l=1}^{2m}\frac{\partial{T}(\partial_{a},e_j,e_l)}{\partial{x_b}}(x_0)c(v)c(w)c(e_j)c(e_l)\xi_{a}\xi_b\nonumber\\
	&-(m-1)\|\xi\|^{-2 m}(\frac{1}{4} \operatorname{R}^{g}-\frac{3}{4}\|T\|^2)c(v)c(w)\nonumber\\
	&-\frac{3(m-1)}{4}\|\xi\|^{-2 m}\sum_{ i<j<k<t=1}^{2m}dT(e_i,e_j,e_k,e_t)(x_0)c(v)c(w)c(e_i)c(e_j)c(e_k)c(e_t)+O\left(\mathbf{x}\right).\nonumber
\end{align}

Below, we compute each term of  $\int_{\|\xi\|=1} \operatorname{tr} [\sigma_{-2m}(\mathcal{P}_{1}D^{-2m+2})](x, \xi) \sigma(\xi)$ in turn.  
Based on the relation of the Clifford action $\operatorname{tr}\big(c(\mathcal{X})c(\mathcal{Y})\big)=-g(\mathcal{X},\mathcal{Y})$ and $\int_{\|\xi\|=1} \xi_{a} \xi_{b}\sigma(\xi)=\frac{1}{n}\delta_{a}^{b}{\rm Vol}(S^{n-1})$, we get the following equations.

\noindent{\bf (I-$\mathbb{A}$)}
\begin{align}
	&\int_{\|\xi\|=1}\operatorname{tr} \biggl\{\frac{m(m-1)}{3}\|\xi\|^{-2 m-2}\sum_{a,b=1}^{2m} \operatorname{Ric}_{a b} \xi_{a} \xi_{b}c(v)c(w)\biggr\}(x_0)\sigma(\xi)\\
	&=-\frac{m-1}{6}s g(v,w){\rm tr}[id]{\rm Vol}(S^{n-1}).\nonumber
\end{align}
where $s$ be the scalar curvature.

\noindent{\bf (I-$\mathbb{B}$)}(See Appendix)

Let $v=\sum_{p=1}^{2m}v_p e_p, w=\sum_{q=1}^{2m}w_q e_q$, and based on the relation of the Clifford action and ${\rm tr}\mathcal{XY}={\rm tr}\mathcal{YX}$, we can obtain the equality
\begin{align}\label{chu}
	&\operatorname{tr}\bigg(\sum_{j\neq l,\hat{j}\neq \hat{l}=1}^{2m}c(v)c(w)c(e_j)c(e_l)c(e_{\hat{j}})c(e_{\hat{l}})\bigg)\\
	&=\sum_{j\neq l,\hat{j}\neq \hat{l}=1}^{2m}\bigg[v_{\hat{l}} w_l \delta_{j}^{\hat{j}}-v_{\hat{l}} w_j \delta_{l}^{\hat{j}}-v_{\hat{j}} w_l \delta_{j}^{\hat{l}}+v_{\hat{j}} w_j \delta_{l}^{\hat{l}}-v_l w_{\hat{l}} \delta_{j}^{\hat{j}}+v_l w_{\hat{j}} \delta_{j}^{\hat{l}}+v_j w_{\hat{l}} \delta_{l}^{\hat{j}}-v_j w_{\hat{j}} \delta_{l}^{\hat{l}}\nonumber\\
	&\;\;\;\;\;\;\;\;\;\;\;\;\;\;\;\;\;-\delta_{j}^{\hat{l}}\delta_{l}^{\hat{j}}g(v,w)+\delta_{j}^{\hat{j}}\delta_{l}^{\hat{l}}g(v,w)\bigg]{\rm tr}[id],\nonumber
\end{align}
where $v_l=g(v,e_l), w_j=g(w,e_j)$, then
\begin{align}\label{1b}
	&\int_{\|\xi\|=1} {\rm tr}\biggl\{-\frac{9m(m-1)}{2}\|\xi\|^{-2 m-2}\sum_{a,b,j<l,\hat{j}<\hat{l}=1}^{2m}T(e_a,e_j,e_l)T(e_b, e_{\hat{j}},e_{\hat{l}})c(v)c(w)c(e_j)c(e_l)c(e_{\hat{j}})c(e_{\hat{l}})\xi_{a}\xi_b\biggr\}(x_0)\sigma(\xi)\\
	&=-\frac{9(m-1)}{16}\sum_{a,b,j\neq l,\hat{j}\neq \hat{l}=1}^{2m}T(e_a,e_j,e_l)T(e_a, e_{\hat{j}},e_{\hat{l}}){\rm tr}\bigg(c(v)c(w)c(e_j)c(e_l)c(e_{\hat{j}})c(e_{\hat{l}})\bigg){\rm Vol}(S^{n-1})\nonumber\\
	&=-\frac{9(m-1)}{8}\sum_{a,j\neq l=1}^{2m}T^2(e_a,e_j,e_l)g(v,w){\rm tr}[id]{\rm Vol}(S^{n-1}).\nonumber
\end{align}

\noindent{\bf (I-$\mathbb{C}$)}

Similar to {\bf (I-$\mathbb{B}$)}, one obtains
\begin{align}
	&\int_{\|\xi\|=1} {\rm tr}\biggl\{\frac{9(m-1)}{4}\|\xi\|^{-2 m}\sum_{a,j<l,\hat{j}<\hat{l}=1}^{2m}T(e_a,e_j,e_l)T(e_a, e_{\hat{j}},e_{\hat{l}})c(v)c(w)c(e_j)c(e_l)c(e_{\hat{j}})c(e_{\hat{l}})\biggr\}(x_0)\sigma(\xi)\\
	&=\;\;\frac{9(m-1)}{8}\sum_{a,j\neq l=1}^{2m}T^2(e_a,e_j,e_l)g(v,w){\rm tr}[id]{\rm Vol}(S^{n-1}).\nonumber
\end{align}

\noindent{\bf (I-$\mathbb{D}$)}(See Appendix)

Since
\begin{align}
	&{\rm tr}\bigg(\sum_{j\neq l=1}^{2m}c(v)c(w)c(e_j)c(e_l)\bigg)=\sum_{j\neq l=1}^{2m}\bigg[v_j w_l - v_l w_j\bigg]{\rm tr}[id],
\end{align}
and
\begin{align}\label{pt}
	\frac{\partial{T}(\partial_{a},e_j,e_l)} {\partial x_a}(x_0)=\nabla_{e_a}(T)(e_{a},e_{j},e_{l})(x_0),
\end{align}
then
\begin{align}\label{2d}
	&\int_{\|\xi\|=1} {\rm tr}\biggl\{\frac{3(m-1)}{2}\|\xi\|^{-2 m}\sum_{a,j<l=1}^{2m}\frac{\partial{T}(\partial{a},e_j,e_l)}{\partial{x_a}}c(v)c(w)c(e_j)c(e_l)\biggr\}(x_0)\sigma(\xi)\\
	&=\;\;\frac{3(m-1)}{4}\sum_{a,j\neq l=1}^{2m}\frac{\partial{T}(\partial{a},e_j,e_l)}{\partial{x_a}}(x_0){\rm tr}\bigg(c(v)c(w)c(e_j)c(e_l)\bigg){\rm Vol}(S^{n-1})\nonumber\\
	&=\;\;\frac{3(m-1)}{4}\sum_{a,j\neq l=1}^{2m}\nabla_{e_a}(T)(e_{a},e_{j},e_{l})(x_0)\bigg(v_j w_l - v_l w_j\bigg){\rm tr}[id]{\rm Vol}(S^{n-1})\nonumber\\
	&=\;\;\frac{3(m-1)}{2}\sum_{a=1}^{2m}\nabla_{e_a}(T)(e_{a},v,w){\rm tr}[id]{\rm Vol}(S^{n-1}).\nonumber
\end{align}

\noindent{\bf (I-$\mathbb{E}$)}

Since 
\begin{align}
	&\int_{\|\xi\|=1} \bigg(\|\xi\|^{-2 m-2} \sum_{a,b,t,s=1}^{2m}\operatorname{R}_{b a t s}\xi_{a} \xi_{b}\bigg)(x_0)\sigma(\xi)\\
	&=\frac{ 1}{2m} \sum_{a,t,s=1}^{2m}\operatorname{R}_{a a t s} {\rm Vol}(S^{n-1})\nonumber\\
	&=\;\;0,\nonumber
\end{align}
then
\begin{align}
	&\int_{\|\xi\|=1} {\rm tr}\biggl\{-\frac{ m (m-1)}{4}\|\xi\|^{-2 m-2} \sum_{a,b,t,s=1}^{2m}\operatorname{R}_{b a t s}c(v)c(w) c(e_s) c(e_t) \xi_{a} \xi_{b}\biggr\}(x_0)\sigma(\xi)=\;0.
\end{align}

\noindent{\bf (I-$\mathbb{F}$)}

Similar to \eqref{2d}
\begin{align}
	&\int_{\|\xi\|=1} {\rm tr}\biggl\{-3m(m-1)\|\xi\|^{-2 m-2}\sum_{a,b,j<l=1}^{2m}\frac{\partial{T}(\partial_{a},e_j,e_l)}{\partial{x_b}}c(v)c(w)c(e_j)c(e_l)\xi_{a}\xi_b\biggr\}(x_0)\sigma(\xi)\\
	&=\;\;-\frac{3(m-1)}{2}\sum_{a=1}^{2m}\nabla_{e_a}(T)(e_{a},v,w){\rm tr}[id]{\rm Vol}(S^{n-1}).\nonumber
\end{align}

\noindent{\bf (I-$\mathbb{G}$)}

Similarly to {\bf (I-$\mathbb{A}$)}, we get
\begin{align}
	&\int_{\|\xi\|=1} {\rm tr}\biggl\{-(m-1)\|\xi\|^{-2 m}(\frac{1}{4} s-\frac{3}{4}\|T\|^2)c(v)c(w)\biggr\}(x_0)\sigma(\xi)\\
	&=(m-1)(\frac{1}{4} s-\frac{3}{4}\|T\|^2) g(v,w){\rm tr}[id]{\rm Vol}(S^{n-1}).\nonumber
\end{align}

\noindent{\bf (I-$\mathbb{H}$)}

Let $v=\sum_{p=1}^{2m}v_p e_p, w=\sum_{q=1}^{2m}w_q e_q$, and based on the relation of the Clifford action and ${\rm tr}\mathcal{XY}={\rm tr}\mathcal{YX}$, we can obtain the equality
\begin{align}
	&\operatorname{tr}\bigg(\sum_{i\neq j \neq k \neq t=1}^{2m}c(v)c(w)c(e_i)c(e_j)c(e_k)c(e_t)\bigg)=0,
\end{align}
then
\begin{align}
	&\int_{\|\xi\|=1} {\rm tr}\biggl\{-\frac{3(m-1)}{4}\|\xi\|^{-2 m}\sum_{i<j<k<t=1}^{2m}dT(e_i,e_j,e_k,e_t)(x_0)c(v)c(w)c(e_i)c(e_j)c(e_k)c(e_t)\biggr\}(x_0)\sigma(\xi)\\
	&=\;\;0.\nonumber
\end{align}

Summing from {\bf (I-$\mathbb{A}$)} to {\bf (I-$\mathbb{H}$)} in turn, we get
\begin{align}\label{zpdt}
	&\int_{\|\xi\|=1} {\rm tr}\biggl\{\sigma_{-2 m}\left(c(v)c(w)  D^{-2 m+2}_T\right)\biggr\}(x_0)\sigma(\xi)\\
	&=\frac{m-1}{12}s g(v,w){\rm tr}[id]{\rm Vol}(S^{n-1})-\frac{3(m-1)}{4}\|T\|^2 g(v,w){\rm tr}[id]{\rm Vol}(S^{n-1}).\nonumber
\end{align}

Since ${\rm tr}[id]=2^{m}$ and ${\rm Vol}(S^{n-1})=\frac{2 \pi^{m}}{\Gamma\left(m\right)}$, by \eqref{wres} we obtain
\begin{align}\label{z1}
	\mathscr{B}_{1}=&\mathrm{Wres}\bigg(c(v)c(w){D_T^{-n+2}}\bigg)\\
	=&\;2^{m} \frac{2 \pi^{m}}{\Gamma\left(m\right)}\int_{M}\biggl\{\frac{m-1}{12}s g(v,w)-\frac{3(m-1)}{4}\|T\|^2 g(v,w)\biggr\}d{\rm Vol}_M\nonumber.
\end{align}

{\bf Part II)} $\mathscr{B}_{2}=\mathrm{Wres}\bigg(c(v)D_Tc(w) D_T D_T^{-n}\bigg)$. $\mathcal{P}_{2}=c(v)D_Tc(w)D_T $, and  $c(v)D_T:=\mathcal{A},c(w)D_T:=\mathcal{B} .$
Let  $n=2 m $, by \eqref{wres}, we need to compute $\int_{S^{*} M} \operatorname{tr}\left[\sigma_{-2 m}\left(\mathcal{P}_{1} D_T^{-2 m}\right)\right](x, \xi) $. Based on the algorithm yielding the principal symbol of a product of pseudo-differential operators in terms of the principal symbols of the factors, we have
\begin{align}\label{ABD}
	\sigma_{-2 m}\left(\mathcal{A} \mathcal{B} D^{-2 m}_T\right) & =\left\{\sum_{|\alpha|=0}^{\infty} \frac{(-i)^{|\alpha|}}{\alpha!} \partial_{\xi}^{\alpha}[\sigma(\mathcal{A} \mathcal{B})] \partial_{x}^{\alpha}\left[\sigma\left(D^{-2 m}_T\right)\right]\right\}_{-2 m} \\
	& =\sigma_{0}(\mathcal{A} \mathcal{B}) \sigma_{-2 m}\left(D^{-2 m}_T\right)+\sigma_{1}(\mathcal{A} \mathcal{B}) \sigma_{-2 m-1}\left(D^{-2 m}_T\right)+\sigma_{2}(\mathcal{A} \mathcal{B}) \sigma_{-2 m-2}\left(D^{-2 m}_T\right) \nonumber\\
	& +(-i) \sum_{j=1}^{2m} \partial_{\xi_{j}}\left[\sigma_{2}(\mathcal{A} \mathcal{B})\right] \partial_{x_{j}}\left[\sigma_{-2 m-1}\left(D^{-2 m}_T\right)\right]+(-i) \sum_{j=1}^{2m} \partial_{\xi_{j}}\left[\sigma_{1}(\mathcal{A} \mathcal{B})\right] \partial_{x_{j}}\left[\sigma_{-2 m}\left(D^{-2 m}_T\right)\right] \nonumber\\
	& -\frac{1}{2} \sum_{j ,l=1}^{2m} \partial_{\xi_{j}} \partial_{\xi_{l}}\left[\sigma_{2}(\mathcal{A} \mathcal{B})\right] \partial_{x_{j}} \partial_{x_{l}}\left[\sigma_{-2 m}\left(D^{-2 m}_T\right)\right] .\nonumber
\end{align}

\begin{lem}\label{www}{\rm \cite{WWw}}
	The symbols of $D_T$ are given
	\begin{align}
		&\sigma_{0}(D_T)=-\frac{1}{4} \sum_{p,s,t=1}^{2m} w_{s,t}(e_p)c(e_p)c(e_s)c(e_t)+\frac{3}{2} \sum_{f<\alpha<\beta=1}^{2m}T(e_{f},e_{\alpha},e_{\beta})c(e_{f})c(e_{\alpha})c(e_{\beta});\nonumber\\
		&\sigma_{1}(D_T)=\sqrt{-1}c(\xi). \nonumber
	\end{align}
\end{lem}
By lemma \ref{www} and $\mathcal{A}=c(v)D_T,\mathcal{B}=c(w)D_T$, we obtain the following lemma.
 \begin{lem}
The symbols of  $\mathcal{A}$  and  $\mathcal{B}$  are given
\begin{align}
	&\sigma_{0}(\mathcal{A})=-\frac{1}{4} \sum_{p,s,t=1}^{2m} w_{s,t}(e_p)c(v)c(e_p)c(e_s)c(e_t)+\frac{3}{2} \sum_{f<\alpha<\beta=1}^{2m}T(e_{f},e_{\alpha},e_{\beta})c(v)c(e_{f})c(e_{\alpha})c(e_{\beta});\nonumber\\
	&\sigma_{1}(\mathcal{A})=\sqrt{-1}c(v)c(\xi); \nonumber\\
	&\sigma_{0}(\mathcal{B})=-\frac{1}{4} \sum_{p,s,t=1}^{2m} w_{s,t}(e_p)c(w)c(e_p)c(e_s)c(e_t)+\frac{3}{2} \sum_{f<\alpha<\beta=1}^{2m}T(e_{f},e_{\alpha},e_{\beta})c(w)c(e_{f})c(e_{\alpha})c(e_{\beta});\nonumber\\
	&\sigma_{1}(\mathcal{B})=\sqrt{-1}c(w)c(\xi)\nonumber.
\end{align}
 \end{lem}

Further, by the composition formula of pseudodifferential operators, we get the following lemma.

 \begin{lem}\label{AB}
	The symbols of  $\mathcal{AB}$  are given	
	\begin{align}
		\sigma_{0}(\mathcal{AB})(x_0)=&\sigma_{0}(\mathcal{A}) \sigma_{0}(\mathcal{B})(x_0)+(-i) \partial_{\xi_{j}}\left[\sigma_{1}(\mathcal{A})\right] \partial_{x_{j}}\left[\sigma_{0}(\mathcal{B})\right](x_0)+(-i) \partial_{\xi_{j}}\left[\sigma_{0}(\mathcal{A})\right] \partial_{x_{j}}\left[\sigma_{1}(\mathcal{B})\right](x_0)\label{lemAB1} \\
		=&\;\;\;\frac{1}{16}\sum_{p,s,t,\hat{p},\hat{s},\hat{t}=1}^{2m} w_{s,t}(e_p) w_{\hat{s}, \hat{t}}(e_{\hat{p}})c(v)c(e_p)c(e_s)c(e_t)c(w)c(e_{\hat{p}})c(e_{\hat{s}})c(e_{\hat{t}})\nonumber\\
		&-\frac{3}{8}\sum_{f<\alpha<\beta,\hat{p},\hat{s},\hat{t}=1}^{2m} T(e_{f},e_{\alpha},e_{\beta}) w_{\hat{s},\hat{t}}(e_{\hat{p}})c(v)c(e_{f})c(e_{\alpha})c(e_{\beta})c(w)c(e_{\hat{p}})c(e_{\hat{s}})c(e_{\hat{t}})\nonumber\\
		&-\frac{3}{8}\sum_{p,s,t,\hat{f}<\hat{\alpha}<\hat{\beta}=1}^{2m} w_{s,t}(e_p)T(e_{\hat{f}},e_{\hat{\alpha}},e_{\hat{\beta}}) c(v)c(e_p)c(e_s)c(e_t)c(w)c(e_{\hat{f}})c(e_{\hat{\alpha}})c(e_{\hat{\beta}})\nonumber\\
		&+\frac{9}{4}\sum_{f<\alpha<\beta,\hat{f}<\hat{\alpha}<\hat{\beta}=1}^{2m} T(e_{f},e_{\alpha},e_{\beta})T(e_{\hat{f}},e_{\hat{\alpha}},e_{\hat{\beta}}) c(v)c(e_{f})c(e_{\alpha})c(e_{\beta})c(w)c(e_{\hat{f}})c(e_{\hat{\alpha}})c(e_{\hat{\beta}})\nonumber\\
		&+\frac{1}{8}\sum_{j,p,t,s=1}^{2m} {\operatorname{R}}_{jpts}c(v)c(dx_j)c(w)c(e_p)c(e_s)c(e_t)\nonumber\\
		&-\frac{1}{4}\sum_{p,s,t,j,\gamma=1}^{2m}w_{s,t}(e_p) \partial x_j(w_\gamma)c(v)c(dx_j)c(e_\gamma)c(e_p)c(e_s)c(e_t) \nonumber\\
		&+\frac{3}{2}\sum_{f<\alpha<\beta,j=1}^{2m}\frac{\partial{T}(e_{f},e_{\alpha},e_{\beta})} {\partial x_j}c(v)c(dx_j)c(e_w)c(e_{f})c(e_{\alpha})c(e_{\beta})\nonumber\\
		&+\frac{3}{2}\sum_{f<\alpha<\beta,j,\gamma=1}^{2m}{T}(e_{f},e_{\alpha},e_{\beta})\partial x_j(w_\gamma)c(v)c(dx_j)c(e_\gamma)c(e_{f})c(e_{\alpha})c(e_{\beta});\nonumber\\
		\sigma_{1}(\mathcal{AB})=&\sigma_{1}(\mathcal{A}) \sigma_{0}(\mathcal{B})+\sigma_{0}(\mathcal{A}) \sigma_{1}(\mathcal{B})+(-i) \partial_{\xi_{j}}\left[\sigma_{1}(\mathcal{A})\right] \partial_{x_{j}}\left[\sigma_{1}(\mathcal{B})\right] \label{lemAB2}\\
		=&-\frac{\sqrt{-1}}{4}\sum_{p,s,t=1}^{2m} w_{s,t}(e_p) c(v)c(\xi)c(w)c(e_p)c(e_s)c(e_t)\nonumber\\
		&+\frac{3\sqrt{-1}}{2}\sum_{f<\alpha<\beta=1}^{2m}{T}(e_{f},e_{\alpha},e_{\beta})c(v)c(\xi)c(w)c(e_{f})c(e_{\alpha})c(e_{\beta})\nonumber\\
		&-\frac{\sqrt{-1}}{4}\sum_{p,s,t=1}^{2m} w_{s,t}(e_p) c(v)c(e_p)c(e_s)c(e_t)c(w)c(\xi)\nonumber\\
		&+\frac{3\sqrt{-1}}{2}\sum_{f<\alpha<\beta=1}^{2m}{T}(e_{f},e_{\alpha},e_{\beta})c(v)c(e_{f})c(e_{\alpha})c(e_{\beta})c(w)c(\xi)\nonumber\\
		&+\sqrt{-1}\sum_{j,\gamma=1}^{2m}\partial x_j(w_\gamma)c(v)c(dx_j)c(e_\gamma)c(\xi);\nonumber\\
		\sigma_{2}(\mathcal{AB})=&\sigma_{1}(\mathcal{A})\sigma_{1}(\mathcal{B})=-c(v)c(\xi)c(w)c(\xi).\label{lemAB3}
	\end{align}
\end{lem}

Next, with \eqref{ABD}, we compute each term of  $\int_{\|\xi\|=1} \operatorname{tr} [\sigma_{-2m}(\mathscr{AB}D_T^{-2m+2})](x, \xi) \sigma(\xi)$ in turn.

\noindent {\bf (II-1)} For $\sigma_{0}(\mathcal{AB}) \sigma_{-2 m}\left(D_{T}^{-2 m}\right)$:

According to \eqref{lem21} in lemma \ref{lem2} and \eqref{lemAB1} in lemma \ref{AB} , where $w_{s,t}(e_p)(x_0)=0$, we get 
\begin{align}\label{0-2m}
&\sigma_{0}(\mathcal{AB}) \sigma_{-2 m}\left(D_T^{-2 m}\right)\left(x_{0}\right)\\
&=\frac{9}{4}\|\xi\|^{-2m}\sum_{f<\alpha<\beta,\hat{f}<\hat{\alpha}<\hat{\beta}=1}^{2m}  T(e_{f},e_{\alpha},e_{\beta})T(e_{\hat{f}},e_{\hat{\alpha}},e_{\hat{\beta}}) c(v)c(e_{f})c(e_{\alpha})c(e_{\beta})c(w)c(e_{\hat{f}})c(e_{\hat{\alpha}})c(e_{\hat{\beta}})\nonumber\\
&\;\;\;\;+\frac{1}{8}\|\xi\|^{-2m}\sum_{j,p,t,s=1}^{2m}  {\operatorname{R}}_{jpts}c(v)c(dx_j)c(w)c(e_p)c(e_s)c(e_t)\nonumber\\
&\;\;\;\;+\frac{3}{2}\|\xi\|^{-2m}\sum_{f<\alpha<\beta,j=1}^{2m}  \frac{\partial{T}(e_{f},e_{\alpha},e_{\beta})} {\partial x_j}c(v)c(dx_j)c(e_w)c(e_{f})c(e_{\alpha})c(e_{\beta})\nonumber\\
&\;\;\;\;+\frac{3}{2}\|\xi\|^{-2m}\sum_{f<\alpha<\beta,j,\gamma=1}^{2m} {T}(e_{f},e_{\alpha},e_{\beta})\partial x_j(w_\gamma)c(v)c(dx_j)c(e_\gamma)c(e_{f})c(e_{\alpha})c(e_{\beta}).\nonumber
\end{align}

 \noindent{\bf (II-1-$\mathbb{A}$)}(See Appendix)
 
Let $v=\sum_{p=1}^{2m}v_p e_p, w=\sum_{q=1}^{2m}w_q e_q$, and based on the relation of the Clifford action and ${\rm tr}\mathcal{XY}={\rm tr}\mathcal{YX}$, we can obtain the equality
\begin{align}\label{t1}
	&{\rm tr} \bigg(\sum_{f\neq\alpha\neq\beta,\hat{f}\neq\hat{\alpha}\neq\hat{\beta}=1}^{2m}c(v)c(e_{f})c(e_{\alpha})c(e_{\beta})c(w)c(e_{\hat{f}})c(e_{\hat{\alpha}})c(e_{\hat{\beta}})\bigg)\\
	&=\sum_{f\neq\alpha\neq\beta,\hat{f}\neq\hat{\alpha}\neq\hat{\beta}=1}^{2m}\bigg[v_f \delta_{\alpha}^{\hat{\beta}} \delta_{\hat{\alpha}}^{\beta} w_{\hat{f}}-v_f \delta_{\alpha}^{\hat{\beta}} \delta_{\hat{f}}^{\beta} w_{\hat{\alpha}}-v_f \delta_{\alpha}^{\hat{\alpha}} \delta_{\hat{\beta}}^{\beta} w_{\hat{f}}+v_f \delta_{\alpha}^{\hat{\alpha}} \delta_{\hat{f}}^{\beta} w_{\hat{\beta}}+v_f \delta_{\alpha}^{\hat{f}} \delta_{\hat{\beta}}^{\beta} w_{\hat{\alpha}}-v_f \delta_{\alpha}^{\hat{f}} \delta_{\hat{\alpha}}^{\beta} w_{\hat{\beta}}\nonumber\\
	&\;\;\;\;\;\;\;\;\;\;\;\;\;\;\;\;\;\;\;\;\;\;\;\;\;\;\;\;-v_{\alpha} \delta_{f}^{\hat{\beta}} \delta_{\hat{\alpha}}^{\beta} w_{\hat{f}}+v_{\alpha} \delta_{f}^{\hat{\beta}} \delta_{\hat{f}}^{\beta} w_{\hat{\alpha}}+v_{\alpha} \delta_{f}^{\hat{\alpha}} \delta_{\hat{\beta}}^{\beta} w_{\hat{f}}-v_{\alpha} \delta_{f}^{\hat{\alpha}} \delta_{\hat{f}}^{\beta} w_{\hat{\beta}}-v_{\alpha} \delta_{f}^{\hat{f}} \delta_{\hat{\beta}}^{\beta} w_{\hat{\alpha}}+v_{\alpha} \delta_{f}^{\hat{f}} \delta_{\hat{\alpha}}^{\beta} w_{\hat{\beta}}\nonumber\\
	&\;\;\;\;\;\;\;\;\;\;\;\;\;\;\;\;\;\;\;\;\;\;\;\;\;\;\;\;+v_{\beta} \delta_{f}^{\hat{\beta}} \delta_{\hat{\alpha}}^{\alpha} w_{\hat{f}}-v_{\beta} \delta_{f}^{\hat{\beta}} \delta_{\hat{f}}^{\alpha} w_{\hat{\alpha}}-v_{\beta} \delta_{f}^{\hat{\alpha}} \delta_{\hat{\beta}}^{\alpha} w_{\hat{f}}+v_{\beta} \delta_{f}^{\hat{\alpha}} \delta_{\hat{f}}^{\alpha} w_{\hat{\beta}}+v_{\beta} \delta_{f}^{\hat{f}} \delta_{\hat{\beta}}^{\alpha} w_{\hat{\alpha}}-v_{\beta} \delta_{f}^{\hat{f}} \delta_{\hat{\alpha}}^{\alpha} w_{\hat{\beta}}\nonumber\\
	&\;\;\;\;\;\;\;\;\;\;\;\;\;\;\;\;\;\;\;\;\;\;\;\;\;\;\;\;+v_{\hat{f}} \delta_{f}^{\hat{\beta}} \delta_{\hat{\alpha}}^{\alpha} w_{\beta}-v_{\hat{f}} \delta_{f}^{\hat{\beta}} \delta_{\hat{\alpha}}^{\beta} w_{\alpha}-v_{\hat{f}} \delta_{f}^{\hat{\alpha}} \delta_{\hat{\beta}}^{\alpha} w_{\beta}+v_{\hat{f}} \delta_{f}^{\hat{\alpha}} \delta_{\hat{\beta}}^{\beta} w_{\alpha}+v_{\hat{f}} \delta_{\beta}^{\hat{\alpha}} \delta_{\hat{\beta}}^{\alpha} w_{f}-v_{\hat{f}} \delta_{\beta}^{\hat{\beta}} \delta_{\hat{\alpha}}^{\alpha} w_{f}\nonumber\\
	&\;\;\;\;\;\;\;\;\;\;\;\;\;\;\;\;\;\;\;\;\;\;\;\;\;\;\;\;-v_{\hat{\alpha}} \delta_{f}^{\hat{\beta}} \delta_{\hat{f}}^{\alpha} w_{\beta}+v_{\hat{\alpha}} \delta_{f}^{\hat{\beta}} \delta_{\hat{f}}^{\beta} w_{\alpha}+v_{\hat{\alpha}} \delta_{f}^{\hat{f}} \delta_{\hat{\beta}}^{\alpha} w_{\beta}-v_{\hat{\alpha}} \delta_{f}^{\hat{f}} \delta_{\hat{\beta}}^{\beta} w_{\alpha}-v_{\hat{\alpha}} \delta_{\beta}^{\hat{f}} \delta_{\hat{\beta}}^{\alpha} w_{f}+v_{\hat{\alpha}} \delta_{\beta}^{\hat{\beta}} \delta_{\hat{f}}^{\alpha} w_{f}\nonumber\\
	&\;\;\;\;\;\;\;\;\;\;\;\;\;\;\;\;\;\;\;\;\;\;\;\;\;\;\;\;+v_{\hat{\beta
	}} \delta_{f}^{\hat{\alpha}} \delta_{\hat{f}}^{\alpha} w_{\beta}-v_{\hat{\beta}} \delta_{f}^{\hat{\alpha}} \delta_{\hat{f}}^{\beta} w_{\alpha}-v_{\hat{\beta}} \delta_{f}^{\hat{f}} \delta_{\hat{\alpha}}^{\alpha} w_{\beta}+v_{\hat{\beta}} \delta_{f}^{\hat{f}} \delta_{\hat{\alpha}}^{\beta} w_{\alpha}+v_{\hat{\beta}} \delta_{\beta}^{\hat{f}} \delta_{\hat{\alpha}}^{\alpha} w_{f}-v_{\hat{\beta}} \delta_{\beta}^{\hat{\alpha}} \delta_{\hat{f}}^{\alpha} w_{f}\nonumber\\
	&\;\;\;\;\;\;\;\;\;\;\;- \delta_{f}^{\hat{\beta}} \delta_{\hat{\alpha}}^{\alpha}\delta_{\hat{f}}^{\beta}g(v,w) +\delta_{f}^{\hat{\beta}} \delta_{\hat{f}}^{\alpha}\delta_{\hat{\alpha}}^{\beta} g(v,w)+ \delta_{f}^{\hat{\alpha}} \delta_{\hat{\beta}}^{\alpha}\delta_{\hat{f}}^{\beta} g(v,w)-\delta_{f}^{\hat{\alpha}}\delta_{\hat{f}}^{\alpha} \delta_{\hat{\beta}}^{\beta} g(v,w)-\delta_{f}^{\hat{f}}\delta_{\hat{\beta}}^{\alpha} \delta_{\hat{\alpha}}^{\beta} g(v,w)+\delta_{f}^{\hat{f}} \delta_{\hat{\alpha}}^{\alpha}\delta_{\hat{\beta}}^{\beta} g(v,w)\bigg]{\rm tr}[id],\nonumber
\end{align}
then
\begin{align}
&\int_{\|\xi\|=1}\operatorname{tr} \biggl\{\frac{9}{4}\|\xi\|^{-2m}\sum_{f<\alpha<\beta,\hat{f}<\hat{\alpha}<\hat{\beta}=1}^{2m}  T(e_{f},e_{\alpha},e_{\beta})T(e_{\hat{f}},e_{\hat{\alpha}},e_{\hat{\beta}}) c(v)c(e_{f})c(e_{\alpha})c(e_{\beta})c(w)c(e_{\hat{f}})c(e_{\hat{\alpha}})c(e_{\hat{\beta}})\biggr\}(x_0)\sigma(\xi)\\
&=\int_{\|\xi\|=1}\operatorname{tr} \biggl\{ \frac{9}{4\times36}\|\xi\|^{-2m}\sum_{f\neq\alpha\neq\beta,\hat{f}\neq\hat{\alpha}\neq\hat{\beta}=1}^{2m}  T(e_{f},e_{\alpha},e_{\beta})T(e_{\hat{f}},e_{\hat{\alpha}},e_{\hat{\beta}})c(v)c(e_{f})c(e_{\alpha})c(e_{\beta})c(w)c(e_{\hat{f}})c(e_{\hat{\alpha}})c(e_{\hat{\beta}})\biggr\}(x_0)\sigma(\xi)\nonumber\\
&=\sum_{f\neq\alpha\neq\beta=1}^{2m} \bigg[\frac{3}{8}T^2(e_{f},e_{\alpha},e_{\beta})g(v,w)-\frac{9}{4}T(v,e_{\alpha},e_{\beta})T(w,e_{\alpha},e_{\beta})\bigg]{\rm tr}[id]{\rm Vol}(S^{n-1}).\nonumber
\end{align}

\noindent{\bf (II-1-$\mathbb{B}$)}(See Appendix)

Similar to {\bf (I-$\mathbb{B}$)}, it can be obtained that
\begin{align}\label{t2}
	&{\rm tr}\bigg(\sum_{j\neq p,t\neq s=1}^{2m}c(v)c(e_{j})c(w)c(e_p)c(e_s)c(e_t)\bigg)\\
	&=\sum_{j\neq p,t\neq s=1}^{2m}\bigg[-v_t w_p \delta_{j}^{s}-v_t w_j \delta_{p}^{s}+v_s w_p \delta_{j}^{t}+v_s w_j \delta_{p}^{t}-v_p w_s \delta_{j}^{t}+v_p w_t \delta_{j}^{s}\nonumber\\
	&\;\;\;\;\;\;\;\;\;\;\;\;\;\;\;\;\;-v_j w_t \delta_{p}^{s}+v_j w_s \delta_{p}^{t}+\delta_{j}^{t}\delta_{p}^{s}g(v,w)-\delta_{j}^{s}\delta_{p}^{t}g(v,w)\bigg]{\rm tr}[id],\nonumber
\end{align}
then
\begin{align}\label{t2s}
	&\int_{\|\xi\|=1}\operatorname{tr}\biggl\{ \frac{1}{8}\|\xi\|^{-2m}\sum_{j,p,t,s=1}^{2m}  {\operatorname{R}}_{jpts}c(v)c(e_j)c(w)c(e_p)c(e_s)c(e_t) \biggr\}(x_0)\sigma(\xi)\\
	&=\bigg(\frac{1}{4} s g(v,w)-\frac{1}{2}{\rm Ric}(v,w)\bigg){\rm tr}[id]{\rm Vol}(S^{n-1}).\nonumber
\end{align}

\noindent{\bf (II-1-$\mathbb{C}$)}(See Appendix)

Similarly, we have
\begin{align}\label{t3}
	&{\rm tr}\bigg(\sum_{j,f\neq{\alpha}\neq{\beta}=1}^{2m}c(v)c(e_{j})c(w)c(e_f)c(e_{\alpha})c(e_{\beta})\bigg)\\
	&=\sum_{j,f\neq{\alpha}\neq{\beta}=1}^{2m}\bigg[-v_{\beta} w_f \delta_{j}^{\alpha}+v_{\beta} w_{\alpha} \delta_{j}^{f}+v_{\alpha} w_f \delta_{j}^{\beta}-v_{\alpha} w_{\beta} \delta_{j}^{f}-v_{f} w_{\alpha} \delta_{j}^{\beta}+v_f w_{\beta} \delta_{j}^{\alpha}\bigg]{\rm tr}[id],\nonumber
\end{align}
and
\begin{align}\label{pt2}
	\frac{\partial{T}(e_{f},e_{\alpha},e_{\beta})} {\partial x_j}(x_0)=\nabla_{e_j}(T)(e_{f},e_{\alpha},e_{\beta})(x_0),
\end{align}
then
\begin{align}
	&\int_{\|\xi\|=1} \operatorname{tr}\biggl\{\frac{3}{2}\|\xi\|^{-2m}\sum_{f<\alpha<\beta,j=1}^{2m}  \frac{\partial{T}(e_{f},e_{\alpha},e_{\beta})} {\partial x_j} c(v)c(dx_j)c(e_w)c(e_{f})c(e_{\alpha})c(e_{\beta})\biggr\}(x_0) \sigma(\xi)\\
	&=\int_{\|\xi\|=1} \biggl\{\frac{1}{4}\|\xi\|^{-2m}\sum_{f\neq\alpha\neq\beta,j=1}^{2m}  \frac{\partial{T}(e_{f},e_{\alpha},e_{\beta})} {\partial x_j} \operatorname{tr}[c(v)c(dx_j)c(e_w)c(e_{f})c(e_{\alpha})c(e_{\beta})] \biggr\}(x_0)\sigma(\xi)\nonumber\\
	&=-\frac{3}{2}\sum_{j=1}^{2m} \nabla_{e_j}(T)(v,w,e_j){\rm tr}[id]{\rm Vol}(S^{n-1}).\nonumber
\end{align}

\noindent{\bf (II-1-$\mathbb{D}$)}

Similarly, we have
\begin{align}\label{t4}
	&{\rm tr}\bigg(\sum_{j,\gamma,f\neq{\alpha}\neq{\beta}=1}^{2m}c(v)c(e_{j})c(e_{\gamma})c(e_f)c(e_{\alpha})c(e_{\beta})\bigg)\\
	&=\sum_{j,\gamma,f\neq{\alpha}\neq{\beta}=1}^{2m}\bigg[-v_{\beta} \delta_{\gamma}^{f} \delta_{j}^{\alpha}+v_{\beta}\delta_{\gamma}^{\alpha} \delta_{j}^{f}+v_{\alpha} \delta_{\gamma}^{f} \delta_{j}^{\beta}-v_{\alpha} \delta_{\gamma}^{\beta} \delta_{j}^{f}-v_{f} \delta_{\gamma}^{\alpha} \delta_{j}^{\beta}+v_f \delta_{\gamma}^{\beta} \delta_{j}^{\alpha}\bigg]{\rm tr}[id],\nonumber
\end{align}
then
\begin{align}
	&\int_{\|\xi\|=1} \biggl\{\frac{3}{2}\|\xi\|^{-2m}\sum_{f<\alpha<\beta,j,\gamma=1}^{2m} {T}(e_{f},e_{\alpha},e_{\beta})\partial x_j(w_\gamma) c(v)c(dx_j)c(e_\gamma)c(e_{f})c(e_{\alpha})c(e_{\beta})\biggr\}(x_0) \sigma(\xi)\\
	&=\int_{\|\xi\|=1} \operatorname{tr}\biggl\{\frac{1}{4}\|\xi\|^{-2m}\sum_{f\neq\alpha\neq\beta,j,\gamma=1}^{2m} {T}(e_{f},e_{\alpha},e_{\beta})\partial x_j(w_\gamma) \operatorname{tr}[c(v)c(dx_j)c(e_\gamma)c(e_{f})c(e_{\alpha})c(e_{\beta})]\biggr\}(x_0) \sigma(\xi)\nonumber\\
	&=-\frac{3}{2}\sum_{j=1}^{2m} T(v,\nabla_{e_j}^{L}w,e_j){\rm tr}[id]{\rm Vol}(S^{n-1}).\nonumber
\end{align}

Summing sequentially from {\bf (II-1-$\mathbb{A}$)} to {\bf (II-1-$\mathbb{D}$)} we get
\begin{align}
	&\int_{\|\xi\|=1} \operatorname{tr}\bigg[\sigma_{0}(\mathcal{AB}) \sigma_{-2 m}\left(D^{-2 m}\right)\left(x_{0}\right)\bigg] \sigma(\xi)\\
	&=\bigg[\frac{1}{4} s g(v,w)-\frac{1}{2}{\rm Ric}(v,w)\bigg]{\rm tr}[id]{\rm Vol}(S^{n-1})\nonumber\\
	&\;\;\;\;+\sum_{f\neq\alpha\neq\beta=1}^{2m} \bigg[\frac{3}{8}T^2(e_{f},e_{\alpha},e_{\beta})g(v,w)-\frac{9}{4}T(v,e_{\alpha},e_{\beta})T(w,e_{\alpha},e_{\beta})\bigg]{\rm tr}[id]{\rm Vol}(S^{n-1})\nonumber\\
	&\;\;\;\;-\frac{3}{2}\sum_{j=1}^{2m} \bigg[\nabla_{e_j}(T)(v,w,e_j)+ T(v,\nabla_{e_j}^{L}w,e_j)\bigg]{\rm tr}[id]{\rm Vol}(S^{n-1}).\nonumber
\end{align}

\noindent{\bf (II-2)} For $\sigma_{1}(\mathcal{AB}) \sigma_{-2 m-1}\left(D_{T}^{-2 m}\right)$:

According to \eqref{lem22} in lemma \ref{lem2} and \eqref{lemAB2} in lemma \ref{AB} , where $w_{s,t}(e_p)(x_0)=0$, we get
\begin{align}\label{1-2m-1}
	&\sigma_{1}(\mathcal{AB}) \sigma_{-2 m-1}\left(D_{T}^{-2 m}\right)\left(x_{0}\right)\\
	&=-\frac{9m}{2}\|\xi\|^{-2m-2}\sum_{f<\alpha<\beta,k<l=1}^{2m}  T(e_{f},e_{\alpha},e_{\beta})T(e_{a},e_{k},e_{l})\nonumber\\
	&\;\;\;\; \times\bigg[c(v)c(\xi)c(w)c(e_f)c(e_{\alpha})c(e_{\beta})c(e_k)c(e_l)+c(v)c(e_f)c(e_{\alpha})c(e_{\beta})c(w)c(\xi)c(e_k)c(e_l)\bigg]\xi_{a}\nonumber\\
	&\;\;\;\;-3m\|\xi\|^{-2m-2}\sum_{j,\gamma,a,k<l=1}^{2m} \partial x_j(w_\gamma){T}(e_a,e_k,e_l)c(v)c(dx_j)c(e_\gamma)c(\xi)c(e_{k})c(e_{l})\xi_{a}.\nonumber
\end{align}

\noindent{\bf (II-2-$\mathbb{A}$)}(See Appendix)

Based on the relation of the Clifford action and ${\rm tr}\mathcal{XY}={\rm tr}\mathcal{YX}$, we can obtain the equality
\begin{align}\label{t5}
	&{\rm tr}\bigg(\sum_{f\neq{\alpha}\neq{\beta},k\neq l \neq \eta=1}^{2m}c(v)c(e_{\eta})c(w)c(e_f)c(e_{\alpha})c(e_{\beta})c(e_k)c(e_l)+c(v)c(e_f)c(e_{\alpha})c(e_{\beta})c(w)c(e_{\eta})c(e_k)c(e_l)\bigg)\\
	&=\sum_{f\neq{\alpha}\neq{\beta},k\neq l \neq \eta=1}^{2m}\biggl\{v_{\beta}\delta_{f}^{l} \delta_{k}^{\alpha} w_{\eta}-v_{\beta}\delta_{f}^{l} \delta_{\eta}^{\alpha} w_{k}-v_{\beta}\delta_{f}^{k} \delta_{l}^{\alpha} w_{\eta}+v_{\beta}\delta_{f}^{k} \delta_{\eta}^{\alpha} w_{l}+v_{\beta}\delta_{f}^{\eta} \delta_{l}^{\alpha} w_{k}-v_{\beta}\delta_{f}^{\eta} \delta_{k}^{\alpha} w_{l}\nonumber\\
	&\;\;\;\;\;\;\;\;\;\;\;\;\;\;\;\;\;\;\;\;\;\;\;-v_{\alpha}\delta_{f}^{l} \delta_{k}^{\beta} w_{\eta}+v_{\alpha}\delta_{f}^{l} \delta_{\eta}^{\beta} w_{k}+v_{\alpha}\delta_{f}^{k} \delta_{l}^{\beta} w_{\eta}-v_{\alpha}\delta_{f}^{k} \delta_{\eta}^{\beta} w_{l}-v_{\alpha}\delta_{f}^{\eta} \delta_{l}^{\beta} w_{k}+v_{\alpha}\delta_{f}^{\eta} \delta_{k}^{\beta} w_{l}\nonumber\\
	&\;\;\;\;\;\;\;\;\;\;\;\;\;\;\;\;\;\;\;\;\;\;\;+v_{f}\delta_{\alpha}^{l} \delta_{k}^{\beta} w_{\eta}-v_{f}\delta_{\alpha}^{l} \delta_{\eta}^{\beta} w_{k}-v_{f}\delta_{\alpha}^{k} \delta_{l}^{\beta} w_{\eta}+v_{f}\delta_{\alpha}^{k} \delta_{\eta}^{\beta} w_{l}+v_{f}\delta_{\alpha}^{\eta} \delta_{l}^{\beta} w_{k}-v_{f}\delta_{\alpha}^{\eta} \delta_{k}^{\beta} w_{l}\nonumber\\
	&\;\;\;\;\;\;\;\;\;\;\;\;\;\;\;\;\;\;\;\;\;\;\;+v_{\eta}\delta_{\alpha}^{l} \delta_{k}^{\beta} w_{f}-v_{\eta}\delta_{\alpha}^{l} \delta_{f}^{k} w_{\beta}-v_{\eta}\delta_{\alpha}^{k} \delta_{l}^{\beta} w_{f}+v_{\eta}\delta_{\alpha}^{k} \delta_{f}^{l} w_{\beta}+v_{\eta} \delta_{l}^{\beta} \delta_{f}^{k}w_{\alpha}-v_{\eta} \delta_{k}^{\beta} \delta_{f}^{l}w_{\alpha}\nonumber\\
	&\;\;\;\;\;\;\;\;\;\;\;\;\;\;\;\;\;\;\;\;\;\;\;\bigg[-\delta_{\alpha}^{l} \delta_{k}^{\beta} \delta_{f}^{\eta}+\delta_{\alpha}^{l} \delta_{\eta}^{\beta} \delta_{f}^{k}+\delta_{\alpha}^{k} \delta_{l}^{\beta} \delta_{f}^{\eta}-\delta_{\alpha}^{k} \delta_{\eta}^{\beta} \delta_{f}^{l}-\delta_{\alpha}^{\eta} \delta_{l}^{\beta} \delta_{f}^{k}+\delta_{\alpha}^{\eta} \delta_{k}^{\beta} \delta_{f}^{l}\bigg]g(v,w)\biggr\}{\rm tr}[id],\nonumber
\end{align}
then
\begin{align}
	&\int_{\|\xi\|=1}\operatorname{tr} \biggl\{-\frac{9m}{2}\|\xi\|^{-2m-2}\sum_{a,f<\alpha<\beta,k<l=1}^{2m}  T(e_{f},e_{\alpha},e_{\beta})T(e_{a},e_{k},e_{l})\\
	&\;\;\;\; \times c(v)c(\xi)c(w)c(e_f)c(e_{\alpha})c(e_{\beta})c(e_k)c(e_l)+c(v)c(e_f)c(e_{\alpha})c(e_{\beta})c(w)c(\xi)c(e_k)c(e_l)\xi_{a} \biggr\}(x_0)\sigma(\xi)\nonumber\\
	&=\int_{\|\xi\|=1}\operatorname{tr} \biggl\{-\frac{3m}{8}\|\xi\|^{-2m-2}\sum_{a,f\neq\alpha\neq\beta,k\neq l,\eta=1}^{2m}  T(e_{f},e_{\alpha},e_{\beta})T(e_{a},e_{k},e_{l})\nonumber\\
	&\;\;\;\; \times c(v)c(e_\eta)c(w)c(e_f)c(e_{\alpha})c(e_{\beta})c(e_k)c(e_l)+c(v)c(e_f)c(e_{\alpha})c(e_{\beta})c(w)c(e_\eta)c(e_k)c(e_l)\xi_{a}\xi_{\eta} \biggr\}(x_0)\sigma(\xi)\nonumber\\
	&=\sum_{f\neq\alpha\neq\beta=1}^{2m} \bigg[-\frac{9}{4}T^2(e_{f},e_{\alpha},e_{\beta})g(v,w)+9T(v,e_{\alpha},e_{\beta})T(w,e_{\alpha},e_{\beta})\bigg]{\rm tr}[id]{\rm Vol}(S^{n-1}).\nonumber
\end{align}

\noindent{\bf (II-2-$\mathbb{B}$)}

Similarly to {\bf (I-$\mathbb{B}$)}, it follows that
\begin{align}\label{t6}
	&{\rm tr}\bigg(\sum_{j,\gamma,\eta \neq  k\neq l=1}^{2m}c(v)c(dx_j)c(e_\gamma)c(e_\eta)c(e_{k})c(e_{l})\bigg)\\
	&=\sum_{j,\gamma,\eta \neq k\neq l=1}^{2m}\bigg[-v_l \delta_{j}^{k}\delta_{\gamma}^{\eta}+v_l \delta_{j}^{\eta}\delta_{\gamma}^{k}+v_k \delta_{j}^{l}\delta_{\gamma}^{\eta}-v_k \delta_{j}^{\eta}\delta_{\gamma}^{l}-v_{\eta} \delta_{j}^{l}\delta_{\gamma}^{k}+v_{\eta} \delta_{j}^{k}\delta_{\gamma}^{l} \bigg]{\rm tr}[id],\nonumber
\end{align}
then
\begin{align}
	&\int_{\|\xi\|=1}\operatorname{tr} \biggl\{-3m\|\xi\|^{-2m-2}\sum_{j,\gamma,a,k<l=1}^{2m} \partial x_j(w_\gamma){T}(e_a,e_k,e_l)\c(v)c(dx_j)c(e_\gamma)c(\xi)c(e_{k})c(e_{l})\xi_{a}\biggr\}(x_0)\sigma(\xi)\\
	&=\int_{\|\xi\|=1}\operatorname{tr} \biggl\{\frac{-3m}{2}\|\xi\|^{-2m-2}\sum_{\eta,j,\gamma,a,k\neq l=1}^{2m} \partial x_j(w_\gamma){T}(e_a,e_k,e_l)c(v)c(dx_j)c(e_\gamma)c(e_\eta)c(e_{k})c(e_{l})\xi_{a}\xi_{\eta}\biggr\}(x_0)\sigma(\xi)\nonumber\\
	&=\;\;\frac{9}{2}\sum_{j=1}^{2m} T(v,\nabla_{e_j}^{L}w,e_j){\rm tr}[id]{\rm Vol}(S^{n-1}).\nonumber
\end{align}

By summing {\bf (II-2-$\mathbb{A}$)} and {\bf (II-2-$\mathbb{B}$)}, we get
\begin{align}
	&\int_{\|\xi\|=1} \operatorname{tr}\bigg[\sigma_{1}(\mathcal{AB}) \sigma_{-2 m-1}\left(D_T^{-2 m}\right)\left(x_{0}\right)\bigg] \sigma(\xi)\\
	&=\sum_{f\neq\alpha\neq\beta=1}^{2m} \bigg[-\frac{9}{4}T^2(e_{f},e_{\alpha},e_{\beta})g(v,w)+9T(v,e_{\alpha},e_{\beta})T(w,e_{\alpha},e_{\beta})\bigg]{\rm tr}[id]{\rm Vol}(S^{n-1})\nonumber\\
	&\;\;\;\;+\frac{9}{2}\sum_{j=1}^{2m} T(v,\nabla_{e_j}^{L}w,e_j){\rm tr}[id]{\rm Vol}(S^{n-1}).\nonumber
\end{align}

\noindent{\bf (II-3)} For $\sigma_{2}(\mathcal{AB}) \sigma_{-2 m-2}\left(D_T^{-2 m}\right)$:

According to \eqref{lem23} in lemma \ref{lem2} and \eqref{lemAB3} in lemma \ref{AB}, we get
\begin{align}\label{2-2m-2}
	&\sigma_{2}(\mathcal{AB}) \sigma_{-2 m-2}\left(D_T^{-2 m}\right)\left(x_{0}\right)\\
	&=-\frac{m(m+1)}{3}\|\xi\|^{-2 m-4}\sum_{a,b=1}^{2m} \operatorname{Ric}_{a b} \xi_{a} \xi_{b}c(v)c(\xi)c(w)c(\xi) \nonumber\\
	&\;\;\;\;+\frac{9m(m+1)}{2}\|\xi\|^{-2 m-4}\sum_{a,b,j<l,\hat{j}<\hat{l}=1}^{2m}T(e_a,e_j,e_l)T(e_b, e_{\hat{j}},e_{\hat{l}})c(v)c(\xi)c(w)c(\xi)c(e_j)c(e_l)c(e_{\hat{j}})c(e_{\hat{l}})\xi_{a}\xi_b \nonumber\\
	&\;\;\;\;-\frac{9m}{4}\|\xi\|^{-2 m-2}\sum_{a,j<l,\hat{j}<\hat{l}=1}^{2m}T(e_a,e_j,e_l)T(e_a, e_{\hat{j}},e_{\hat{l}})c(v)c(\xi)c(w)c(\xi)c(e_j)c(e_l)c(e_{\hat{j}})c(e_{\hat{l}}) \nonumber\\
	&\;\;\;\;-\frac{3m}{2}\|\xi\|^{-2 m-2}\sum_{j<l=1}^{2m}\frac{\partial{T}(\partial_{a},e_j,e_l)}{\partial{x_a}}c(v)c(\xi)c(w)c(\xi)c(e_j)c(e_l)\nonumber\\
	&\;\;\;\;+\frac{ m (m+1)}{4}\|\xi\|^{-2 m-4} \sum_{a,b,t,s=1}^{2m}\operatorname{R}_{b a t s} c(v)c(\xi)c(w)c(\xi)c(e_s) c(e_t) \xi_{a} \xi_{b}\nonumber\\
	&\;\;\;\;+3m(m+1)\|\xi\|^{-2 m-4}\sum_{a,b,j<l=1}^{2m}\frac{\partial{T}(\partial_{a},e_j,e_l)}{\partial{x_b}}c(v)c(\xi)c(w)c(\xi)c(e_j)c(e_l)\xi_{a}\xi_b\nonumber\\
	&\;\;\;\;+m\|\xi\|^{-2 m-2}(\frac{1}{4} \operatorname{R}^{g}-\frac{3}{4}\|T\|^2)c(v)c(\xi)c(w)c(\xi)\nonumber\\
	&\;\;\;\;+\frac{3m}{4}\|\xi\|^{-2 m-2}\sum_{i<j<k<t=1}^{2m}dT(e_i,e_j,e_k,e_t)c(v)c(\xi)c(w)c(\xi)c(e_i)c(e_j)c(e_k)c(e_t).\nonumber
\end{align}

\noindent{\bf (II-3-$\mathbb{A}$)}

Based on the relation of the Clifford action and ${\rm tr}\mathcal{XY}={\rm tr}\mathcal{YX}$, we get
\begin{align}\label{t7}
	{\rm tr}\bigg(\sum_{f,g=1}^{2m} c(v)c(e_f)c(w)c(e_g)\bigg)=\sum_{f,g=1}^{2m}\bigg[v_g w_f -\delta_{f}^{g}g(v,w)+v_g w_f \bigg]{\rm tr}[id],
\end{align}
then
\begin{align}
	&\int_{\|\xi\|=1} \operatorname{tr}\biggl\{-\frac{m(m+1)}{3}\|\xi\|^{-2 m-4}\sum_{a,b=1}^{2m} \operatorname{Ric}_{a b} \xi_{a} \xi_{b} c(v)c(\xi)c(w)c(\xi)\biggr\}(x_0)\sigma(\xi)\nonumber\\
	&=\int_{\|\xi\|=1} \operatorname{tr}\biggl\{-\frac{m(m+1)}{3}\|\xi\|^{-2 m-4}\sum_{a,b,f,g=1}^{2m} \operatorname{Ric}_{a b} \xi_{a} \xi_{b} \xi_{f} \xi_{g}c(v)c(e_f)c(w)c(e_g)\biggr\}(x_0)\sigma(\xi)\nonumber\\
	&=\;\;\bigg(\frac{m}{6} s g(v,w)-\frac{1}{3}{\rm Ric}(v,w)\bigg){\rm tr}[id]{\rm Vol}(S^{n-1}).\nonumber
\end{align}

\noindent{\bf (II-3-$\mathbb{B}$)}(See Appendix)

Based on the relation of the Clifford action and ${\rm tr}\mathcal{XY}={\rm tr}\mathcal{YX}$, we get
\begin{align}\label{t70}
	&{\rm tr}\bigg(\sum_{f,j\neq l,\hat{j}\neq\hat{l}=1}^{2m}c(v)c(e_f)c(w)c(e_f)c(e_j)c(e_l)c(e_{\hat{j}})c(e_{\hat{l}})\bigg)\\
	&=2\sum_{f,j\neq l,\hat{j}\neq\hat{l}=1}^{2m}\biggl\{m\bigg[v_{\hat{l}}w_{l}\delta^{\hat{j}}_{j}-v_{\hat{l}}w_{j}\delta^{\hat{j}}_{l}-v_{\hat{j}}w_{l}\delta^{\hat{l}}_{j}+v_{\hat{j}}w_{j}\delta^{\hat{l}}_{j}-v_{l}w_{\hat{l}}\delta^{\hat{j}}_{j}+v_{l}w_{\hat{j}}\delta^{\hat{l}}_{j}+v_{j}w_{\hat{l}}\delta^{\hat{j}}_{l}-v_{j}w_{\hat{j}}\delta^{\hat{l}}_{l}\nonumber\\
	&\;\;\;\;\;\;\;\;\;\;\;\;\;\;\;\;\;\;-\delta^{\hat{l}}_{j}\delta^{\hat{j}}_{l}g(v,w)+\delta^{\hat{j}}_{j}\delta^{\hat{l}}_{l}g(v,w)\bigg]-w_{f}v_{\hat{l}}\delta^{f}_{l}\delta^{\hat{j}}_{j}+w_{f}v_{\hat{l}}\delta^{f}_{j}\delta^{\hat{j}}_{l}+w_{f}v_{\hat{j}}\delta^{f}_{l}\delta^{\hat{l}}_{j}-w_{f}v_{\hat{j}}\delta^{f}_{j}\delta^{\hat{l}}_{j}\nonumber\\
	&\;\;\;\;\;\;\;\;\;\;\;\;\;\;\;\;\;\;+w_{f}v_{l}\delta^{f}_{\hat{l}}\delta^{\hat{j}}_{j}-w_{f}v_{l}\delta^{f}_{\hat{j}}\delta^{\hat{l}}_{j}-w_{f}v_{j}\delta^{f}_{\hat{l}}\delta^{\hat{j}}_{l}+w_{f}v_{j}\delta^{f}_{\hat{j}}\delta^{\hat{l}}_{l}+w_{f}v_{f}\delta^{\hat{l}}_{j}\delta^{\hat{j}}_{l}-w_{f}v_{f}\delta^{\hat{j}}_{j}\delta^{\hat{l}}_{l}\biggr\}{\rm tr}[id],\nonumber
\end{align}
and
\begin{align}\label{t8}
	&{\rm tr}\sum_{a,b,j\neq l,\hat{j}\neq\hat{l}=1}^{2m}\bigg(c(v)c(e_a)c(w)c(e_b)c(e_j)c(e_l)c(e_{\hat{j}})c(e_{\hat{l}})+c(v)c(e_b)c(w)c(e_a)c(e_j)c(e_l)c(e_{\hat{j}})c(e_{\hat{l}})\bigg)\\
    &=2\sum_{a,b,j\neq l,\hat{j}\neq\hat{l}=1}^{2m}\bigg[-w_{a}v_{\hat{l}}\delta^{b}_{l}\delta^{\hat{j}}_{j}+w_{a}v_{\hat{l}}\delta^{b}_{j}\delta^{\hat{j}}_{l}+w_{a}v_{\hat{j}}\delta^{b}_{l}\delta^{\hat{l}}_{j}-w_{a}v_{\hat{j}}\delta^{b}_{j}\delta^{\hat{l}}_{j}+w_{a}v_{l}\delta^{b}_{\hat{l}}\delta^{\hat{j}}_{j}-w_{a}v_{l}\delta^{b}_{\hat{j}}\delta^{\hat{l}}_{j}-w_{a}v_{j}\delta^{b}_{\hat{l}}\delta^{\hat{j}}_{l}\nonumber\\
    &\;\;\;\;\;\;\;\;\;\;\;\;\;\;\;\;\;\;\;+w_{a}v_{j}\delta^{b}_{\hat{j}}\delta^{\hat{l}}_{l}+w_{a}v_{b}\delta^{\hat{l}}_{j}\delta^{\hat{j}}_{l}-w_{a}v_{b}\delta^{\hat{j}}_{j}\delta^{\hat{l}}_{l}+v_{\hat{l}}w_{l}\delta^{\hat{j}}_{j}\delta^{a}_{b}-v_{\hat{l}}w_{j}\delta^{\hat{j}}_{l}\delta^{a}_{b}-v_{\hat{j}}w_{l}\delta^{\hat{l}}_{j}\delta^{a}_{b}+v_{\hat{j}}w_{j}\delta^{\hat{l}}_{j}\delta^{a}_{b}\nonumber\\
    &\;\;\;\;\;\;\;\;\;\;\;\;\;\;\;\;\;\;\;-v_{l}w_{\hat{l}}\delta^{\hat{j}}_{j}\delta^{a}_{b}+v_{l}w_{\hat{j}}\delta^{\hat{l}}_{j}\delta^{a}_{b}+v_{j}w_{\hat{l}}\delta^{\hat{j}}_{l}-v_{j}w_{\hat{j}}\delta^{\hat{l}}_{l}\delta^{a}_{b}-\delta^{\hat{l}}_{j}\delta^{\hat{j}}_{l}\delta^{a}_{b}g(v,w)+\delta^{\hat{j}}_{j}\delta^{\hat{l}}_{l}\delta^{a}_{b}g(v,w)\nonumber\\
    &\;\;\;\;\;\;\;\;\;\;\;\;\;\;\;\;\;\;\;-w_{b}v_{\hat{l}}\delta^{a}_{l}\delta^{\hat{j}}_{j}+w_{b}v_{\hat{l}}\delta^{a}_{j}\delta^{\hat{j}}_{l}+w_{b}v_{\hat{j}}\delta^{a}_{l}\delta^{\hat{l}}_{j}-w_{b}v_{\hat{j}}\delta^{a}_{j}\delta^{\hat{l}}_{j}+w_{b}v_{l}\delta^{a}_{\hat{l}}\delta^{\hat{j}}_{j}-w_{b}v_{l}\delta^{a}_{\hat{j}}\delta^{\hat{l}}_{j}-w_{b}v_{j}\delta^{a}_{\hat{l}}\delta^{\hat{j}}_{l}\nonumber\\
    &\;\;\;\;\;\;\;\;\;\;\;\;\;\;\;\;\;\;+w_{b}v_{j}\delta^{a}_{\hat{j}}\delta^{\hat{l}}_{l}+w_{b}v_{a}\delta^{\hat{l}}_{j}\delta^{\hat{j}}_{l}-w_{b}v_{a}\delta^{\hat{j}}_{j}\delta^{\hat{l}}_{l}\bigg]{\rm tr}[id].\nonumber
\end{align}
By \eqref{t70}, \eqref{t8} and $\int_{\|\xi\|=1}\xi_{a}\xi_b\xi_{f}\xi_g\sigma(\xi)=\frac{1}{n(n+2)}\big(\delta_{a}^{b}\delta_{f}^{g}+\delta_{a}^{f}\delta_{b}^{g}+\delta_{a}^{g}\delta_{b}^{f}\big){\rm Vol}(S^{n-1})$, we get
\begin{align}
	&\int_{\|\xi\|=1}\operatorname{tr} \biggl\{\frac{9m(m+1)}{2}\|\xi\|^{-2 m-4}\sum_{a,b,j<l,\hat{j}<\hat{l}=1}^{2m}T(e_a,e_j,e_l)T(e_b, e_{\hat{j}},e_{\hat{l}})c(v)c(\xi)c(w)c(\xi)c(e_j)c(e_l)c(e_{\hat{j}})c(e_{\hat{l}})\xi_{a}\xi_b\biggr\}(x_0)\sigma(\xi)\\
	&=\int_{\|\xi\|=1}\operatorname{tr} \biggl\{\frac{9m(m+1)}{8}\sum_{a,b,f,g,j\neq l,\hat{j}\neq\hat{l}=1}^{2m}T(e_a,e_j,e_l)T(e_b, e_{\hat{j}},e_{\hat{l}})c(v)c(e_f)c(w)c(e_g)c(e_j)c(e_l)c(e_{\hat{j}})c(e_{\hat{l}})\xi_{a}\xi_b\xi_{f}\xi_g\biggr\}(x_0)\sigma(\xi)\nonumber\\
	&=\;\;\frac{9}{32}\sum_{a,f,j\neq l,\hat{j}\neq\hat{l}=1}^{2m}T(e_a,e_j,e_l)T(e_a, e_{\hat{j}},e_{\hat{l}}){\rm tr}[c(v)c(e_f)c(w)c(e_f)c(e_j)c(e_l)c(e_{\hat{j}})c(e_{\hat{l}})]{\rm Vol}(S^{n-1})\nonumber\\
	&\;\;\;\;+\frac{9}{32}\sum_{a,b,j\neq l,\hat{j}\neq\hat{l}=1}^{2m}T(e_a,e_j,e_l)T(e_b, e_{\hat{j}},e_{\hat{l}}){\rm tr}[c(v)c(e_a)c(w)c(e_b)c(e_j)c(e_l)c(e_{\hat{j}})c(e_{\hat{l}})]{\rm Vol}(S^{n-1})\nonumber\\
	&\;\;\;\;+\frac{9}{32}\sum_{a,b,j\neq l,\hat{j}\neq\hat{l}=1}^{2m}T(e_a,e_j,e_l)T(e_b, e_{\hat{j}},e_{\hat{l}}){\rm tr}[c(v)c(e_b)c(w)c(e_a)c(e_j)c(e_l)c(e_{\hat{j}})c(e_{\hat{l}})]{\rm Vol}(S^{n-1})\nonumber\\
	&=\sum_{a,j\neq l=1}^{2m} \bigg[\frac{9}{8}(2-m)T^2(e_{a},e_{j},e_{l})g(v,w)-\frac{9}{4}T(v,e_{j},e_{l})T(w,e_{j},e_{l})\bigg]{\rm tr}[id]{\rm Vol}(S^{n-1}).\nonumber
\end{align}

\noindent{\bf (II-3-$\mathbb{C}$)}

By \eqref{t70} and $\int_{\|\xi\|=1}\xi_{f}\xi_g\sigma(\xi)=\frac{1}{n}\delta_{f}^{g}{\rm Vol}(S^{n-1})$, we get
\begin{align}
	&\int_{\|\xi\|=1}\operatorname{tr} \biggl\{-\frac{9m}{4}\|\xi\|^{-2 m-2}\sum_{a,j<l,\hat{j}<\hat{l}=1}^{2m}T(e_a,e_j,e_l)T(e_a, e_{\hat{j}},e_{\hat{l}})c(v)c(\xi)c(w)c(\xi)c(e_j)c(e_l)c(e_{\hat{j}})c(e_{\hat{l}})\biggr\}(x_0)\sigma(\xi)\\
	&=\int_{\|\xi\|=1}\operatorname{tr} \biggl\{\frac{9m}{8}\|\xi\|^{-2 m-2}\sum_{a,f,g,j\neq l,\hat{j}\neq\hat{l}=1}^{2m}T(e_a,e_j,e_l)T(e_a, e_{\hat{j}},e_{\hat{l}})c(v)c(e_f)c(w)c(e_g)c(e_j)c(e_l)c(e_{\hat{j}})c(e_{\hat{l}})\xi_{f}\xi_g\biggr\}(x_0)\sigma(\xi)\nonumber\\
	&=\frac{9}{16}\sum_{a,f,g,j\neq l,\hat{j}\neq\hat{l}=1}^{2m}T(e_a,e_j,e_l)T(e_a, e_{\hat{j}},e_{\hat{l}}){\rm tr}[c(v)c(e_f)c(w)c(e_f)c(e_j)c(e_l)c(e_{\hat{j}})c(e_{\hat{l}})]{\rm Vol}(S^{n-1})\nonumber\\
	&=\sum_{a,j\neq l=1}^{2m} \frac{9}{8}(m-1)T^2(e_{a},e_{j},e_{l})g(v,w){\rm tr}[id]{\rm Vol}(S^{n-1}).\nonumber
\end{align}

\noindent{\bf (II-3-$\mathbb{D}$)}

Similarly to {\bf (I-$\mathbb{B}$)}, we get
\begin{align}\label{t9}
	&{\rm tr}\bigg(\sum_{f,j\neq l=1}^{2m}c(v)c(e_f)c(w)c(e_f)c(e_j)c(e_l)\bigg)\\
	&=\sum_{f,j\neq l=1}^{2m}\bigg[2m v_l w_j-2m v_j w_l -2w_f v_l \delta_{f}^{j}+2w_f v_j \delta_{f}^{l}\bigg]{\rm tr}[id],\nonumber
\end{align}

then
\begin{align}
	&\int_{\|\xi\|=1}\operatorname{tr} \biggl\{-\frac{3m}{2}\|\xi\|^{-2 m-2}\sum_{j<l=1}^{2m}\frac{\partial{T}(\partial_{a},e_j,e_l)}{\partial{x_a}}c(v)c(\xi)c(w)c(\xi)c(e_j)c(e_l)\biggr\}(x_0)\sigma(\xi)\\
	&=\int_{\|\xi\|=1}\operatorname{tr} \biggl\{-\frac{3m}{4}\|\xi\|^{-2 m-2}\sum_{j\neq l,f,g=1}^{2m}\frac{\partial{T}(\partial_{a},e_j,e_l)}{\partial{x_a}}c(v)c(e_f)c(w)c(e_g)c(e_j)c(e_l)\xi_{f}\xi_{g}\biggr\}(x_0)\sigma(\xi)\nonumber\\
	&=-\frac{3}{8}\|\xi\|^{-2 m-2}\sum_{j\neq l,f,g=1}^{2m}\frac{\partial{T}(\partial_{a},e_j,e_l)}{\partial{x_a}}{\rm tr}\bigg(c(v)c(e_f)c(w)c(e_f)c(e_j)c(e_l)\bigg){\rm Vol}(S^{n-1})\nonumber\\
	&=\;\;\frac{3}{2}(m-1)\sum_{a=1}^{2m}\nabla_{e_a}(T)(e_a,v,w){\rm tr}[id]{\rm Vol}(S^{n-1}).\nonumber
\end{align}

\noindent{\bf (II-3-$\mathbb{E}$)}

Due to $\int_{\|\xi\|=1}\xi_{a}\xi_b\xi_{f}\xi_g\sigma(\xi)=\frac{1}{n(n+2)}\big(\delta_{a}^{b}\delta_{f}^{g}+\delta_{a}^{f}\delta_{b}^{g}+\delta_{a}^{g}\delta_{b}^{f}\big){\rm Vol}(S^{n-1})$ and $\operatorname{R}_{a a t s}=0$, we get
\begin{align}
	&\int_{\|\xi\|=1} \bigg(\|\xi\|^{-2 m-4} \sum_{a,b,f,g,t,s=1}^{2m}\operatorname{R}_{b a t s}  \xi_{a} \xi_{b}\xi_{f} \xi_{g}\bigg)\sigma(\xi)\\
    &=\;\;\frac{1}{4m(m+1)}\sum_{a,b,f,g,t,s=1}^{2m}\operatorname{R}_{b a t s}\big(\delta_{a}^{f}\delta_{b}^{g}+\delta_{a}^{g}\delta_{b}^{f}\big){\rm Vol}(S^{n-1}),\nonumber
\end{align}
then
\begin{align}
	&\int_{\|\xi\|=1}\operatorname{tr} \biggl\{\frac{ m (m+1)}{4}\|\xi\|^{-2 m-4} \sum_{a,b,t,s=1}^{2m}\operatorname{R}_{b a t s} c(v)c(\xi)c(w)c(\xi)c(e_s) c(e_t) \xi_{a} \xi_{b}\biggr\}(x_0)\sigma(\xi)\\
	&=\int_{\|\xi\|=1}\operatorname{tr} \biggl\{\frac{ m (m+1)}{4}\|\xi\|^{-2 m-4} \sum_{a,b,f,g,t,s=1}^{2m}\operatorname{R}_{b a t s} c(v)c(e_f)c(w)c(e_g)c(e_s) c(e_t) \xi_{a} \xi_{b}\xi_{f} \xi_{g}\biggr\}(x_0)\sigma(\xi)\nonumber\\
	&=\;\;\frac{ 1}{16} \sum_{a,b,t,s=1}^{2m}\operatorname{R}_{b a t s} {\rm tr}\bigg(c(v)c(e_a)c(w)c(e_b)c(e_s) c(e_t)+c(v)c(e_a)c(w)c(e_b)c(e_s) c(e_t)\bigg){\rm Vol}(S^{n-1})\nonumber\\
	&=\;\;\frac{1}{8} \sum_{a,b,t,s=1}^{2m}\operatorname{R}_{b a t s} \bigg(-w_a v_t \delta_{b}^{s}+w_a v_s \delta_{b}^{t}+w_b v_t \delta_{a}^{s}-w_b v_s \delta_{a}^{t}\bigg){\rm Vol}(S^{n-1})\nonumber\\
	&=\;\;0.\nonumber
\end{align}

\noindent{\bf (II-3-$\mathbb{F}$)}

By \eqref{t2} and \eqref{pt2}, we obtain
\begin{align}
&\int_{\|\xi\|=1}\operatorname{tr} \biggl\{3m(m+1)\|\xi\|^{-2 m-4}\sum_{a,b,j<l=1}^{2m}\frac{\partial{T}(\partial_{a},e_j,e_l)}{\partial{x_b}}c(v)c(\xi)c(w)c(\xi)c(e_j)c(e_l)\xi_{a}\xi_b\biggr\}(x_0)\sigma(\xi)\\
&=\int_{\|\xi\|=1}\operatorname{tr} \biggl\{\frac{3m(m+1)}{2}\|\xi\|^{-2 m-4}\sum_{a,b,f,g,j\neq l=1}^{2m}\frac{\partial{T}(\partial_{a},e_j,e_l)}{\partial{x_b}}c(v)c(e_f)c(w)c(e_g)c(e_j)c(e_l)\xi_{a}\xi_b\xi_{f}\xi_g\biggr\}(x_0)\sigma(\xi)\nonumber\\
&=\;\;\frac{3}{8}\sum_{a,f,j\neq l=1}^{2m}\nabla_{e_a}(T)(e_a,e_j,e_l){\rm tr}\bigg(c(v)c(e_f)c(w)c(e_f)c(e_j)c(e_l)\bigg){\rm Vol}(S^{n-1})\nonumber\\
&\;\;\;\;+\frac{3}{8}\sum_{a,b,j\neq l=1}^{2m}\nabla_{e_b}(T)(e_a,e_j,e_l){\rm tr}\bigg(c(v)c(e_a)c(w)c(e_b)c(e_j)c(e_l)+c(v)c(e_b)c(w)c(e_a)c(e_j)c(e_l)\bigg){\rm Vol}(S^{n-1})\nonumber\\
&=-\frac{3}{2}(m+1)\nabla_{e_a}(T)(e_a,v,w){\rm tr}[id]{\rm Vol}(S^{n-1}).\nonumber
\end{align}

\noindent{\bf (II-3-$\mathbb{G}$)}

By \eqref{t7}, we obtain
\begin{align}
	&\int_{\|\xi\|=1}\operatorname{tr} \biggl\{m\|\xi\|^{-2 m-2}(\frac{1}{4} \operatorname{R}^{g}-\frac{3}{4}\|T\|^2)c(v)c(\xi)c(w)c(\xi)\biggr\}(x_0)\sigma(\xi)\\
	&=\int_{\|\xi\|=1} \operatorname{tr}\biggl\{m\|\xi\|^{-2 m-2}(\frac{1}{4} \operatorname{R}^{g}-\frac{3}{4}\|T\|^2)c(v)c(e_f)c(w)c(e_g)\xi_{f}\xi_{g}\biggr\}(x_0)\sigma(\xi)\nonumber\\
	&=\;\;\frac{1}{2}(\frac{1}{4} \operatorname{R}^{g}-\frac{3}{4}\|T\|^2)[2g(v,w)-2mg(v,w)]{\rm tr}[id]{\rm Vol}(S^{n-1})\nonumber\\
	&=\;\;(1-m)(\frac{1}{4} s-\frac{3}{4}\|T\|^2)g(v,w){\rm tr}[id]{\rm Vol}(S^{n-1}).\nonumber
\end{align}

\noindent{\bf (II-3-$\mathbb{H}$)}

\begin{align}
\sum_{i<j<k<t=1}^{2m}{\rm tr}\bigg(c(v)c(e_f)c(w)c(e_f)c(e_i)c(e_j)c(e_k)c(e_t)\bigg)=0,
\end{align}
then
\begin{align}
	&\int_{\|\xi\|=1}\operatorname{tr} \biggl\{\frac{3m}{4}\|\xi\|^{-2 m-2}\sum_{i<j<k<t=1}^{2m}dT(e_i,e_j,e_k,e_t)c(v)c(\xi)c(w)c(\xi)c(e_i)c(e_j)c(e_k)c(e_t)\biggr\}(x_0)\sigma(\xi)\\
	&=\int_{\|\xi\|=1}\operatorname{tr} \biggl\{\frac{3m}{4}\|\xi\|^{-2 m-2}\sum_{i<j<k<t=1}^{2m}dT(e_i,e_j,e_k,e_t)c(v)c(e_f)c(w)c(e_g)c(e_i)c(e_j)c(e_k)c(e_t)\xi_{f}\xi_{g}\biggr\}(x_0)\sigma(\xi)\nonumber\\
	&=\;\;\frac{3}{8}\sum_{i<j<k<t=1}^{2m}dT(e_i,e_j,e_k,e_t){\rm tr}\bigg(c(v)c(e_f)c(w)c(e_f)c(e_i)c(e_j)c(e_k)c(e_t)\bigg){\rm Vol}(S^{n-1})\nonumber\\
	&=\;\;0.\nonumber
\end{align}

Summing from {\bf (II-3-$\mathbb{A}$)} to {\bf (II-3-$\mathbb{H}$)}, we get
\begin{align}
	&\int_{\|\xi\|=1} \operatorname{tr}\bigg[\sigma_{2}(\mathcal{AB}) \sigma_{-2 m-2}\left(D^{-2 m}\right)(x_0)\bigg] \sigma(\xi)\\
	&=\;\;\bigg[\frac{m}{6} s g(v,w)-\frac{1}{3}{\rm Ric}(v,w)\bigg]{\rm tr}[id]{\rm Vol}(S^{n-1})\nonumber\\
	&\;\;\;\;+\sum_{a\neq j\neq l=1}^{2m} \bigg[\frac{9}{8}T^2(e_{a},e_{j},e_{l})g(v,w)-\frac{9}{4}T(v,e_{j},e_{l})T(w,e_{j},e_{l})\bigg]{\rm tr}[id]{\rm Vol}(S^{n-1})\nonumber\\
	&\;\;\;\;-3\sum_{a=1}^{2m} \nabla_{e_a}(T)(e_a,v,w){\rm tr}[id]{\rm Vol}(S^{n-1})+(1-m)(\frac{1}{4} s-\frac{3}{4}\|T\|^2)g(v,w){\rm tr}[id]{\rm Vol}(S^{n-1}).\nonumber
\end{align}

\noindent{\bf (II-4)} For $(-i) \sum_{j=1}^{2m} \partial_{\xi_{j}}\left[\sigma_{2}(\mathcal{A} \mathcal{B})\right] \partial_{x_{j}}\left[\sigma_{-2 m-1}\left(D^{-2 m}_T\right)\right]$:

According to \eqref{lem22} in lemma \ref{lem2} and \eqref{lemAB3} in lemma \ref{AB}, we get
\begin{align}\label{2-2m-1}
	&(-i) \sum_{j=1}^{2m} \partial_{\xi_{j}}\left[\sigma_{2}(\mathcal{A} \mathcal{B})\right] \partial_{x_{j}}\left[\sigma_{-2 m-1}\left(D^{-2 m}_T\right)\right](x_0)\\
	&=\;\;\frac{2 m }{3}\|\xi\|^{-2 m-2}\sum_{a,b,j=1}^{2m}  \operatorname{Ric}_{a b}  \xi_{a}\delta^{b}_{j} \bigg(c(v)c(dx_j)c(w)c(\xi)+c(v)c(\xi)c(w)c(dx_j)\bigg)\nonumber\\
	&\;\;\;\;-\frac{ m }{4}\|\xi\|^{-2 m-2} \sum_{a,b,j,t,s=1}^{2m} \operatorname{R}_{b a t s}\xi_{a}\delta^{b}_{j}\bigg(c(v)c(dx_j)c(w)c(\xi) c(e_s) c(e_t)+c(v)c(\xi)c(w)c(dx_j) c(e_s) c(e_t)\bigg)  \nonumber\\
	&\;\;\;\;-3m\|\xi\|^{-2 m-2}\sum_{a,b,j,\hat{j}<\hat{l}=1}^{2m}\frac{\partial{T}(\partial_{a},e_{\hat{j}},e_{\hat{l}})}{\partial{x_b}}\xi_{a}\delta^{b}_{j}\bigg(c(v)c(dx_j)c(w)c(\xi)c(e_{\hat{j}})c(e_{\hat{l}})+c(v)c(\xi)c(w)c(dx_j)c(e_{\hat{j}})c(e_{\hat{l}})\bigg). \nonumber
\end{align}

\noindent{\bf (II-4-$\mathbb{A}$)}

Based on the relation of the Clifford action and ${\rm tr}\mathcal{XY}={\rm tr}\mathcal{YX}$, we get
\begin{align}
	&\sum_{a,b=1}^{2m}{\rm tr}\bigg(c(v)c(e_b)c(w)c(e_a)+c(v)c(e_a)c(w)c(e_b)\bigg)=\sum_{a,b=1}^{2m}\bigg(2v_a w_b -2\delta_{a}^{b}g(v,w)+2v_b w_a\bigg){\rm tr}[id],
\end{align}
then
\begin{align}
	&\int_{\|\xi\|=1} {\rm tr}\biggl\{\frac{2 m }{3}\|\xi\|^{-2 m-2}\sum_{a,b,j=1}^{2m}  \operatorname{Ric}_{a b}  \xi_{a}\delta^{b}_{j}c(v)c(dx_j)c(w)c(\xi)+c(v)c(\xi)c(w)c(dx_j)\biggr\}(x_0)\sigma(\xi)\\
	&=\int_{\|\xi\|=1}{\rm tr} \biggl\{\frac{2 m }{3}\|\xi\|^{-2 m-2}\sum_{a,b,f=1}^{2m}  \operatorname{Ric}_{a b}  \xi_{a}\xi_{f}c(v)c(e_b)c(w)c(e_f)+c(v)c(e_f)c(w)c(e_b)\biggr\}(x_0)\sigma(\xi)\nonumber\\
	&=\;\;\frac{1}{3}\sum_{a,b=1}^{2m}  \operatorname{Ric}_{a b}  {\rm tr}\bigg(c(v)c(e_b)c(w)c(e_a)+c(v)c(e_a)c(w)c(e_b)\bigg)(x_0){\rm Vol}(S^{n-1})\nonumber\\
	&=\;\;\frac{2}{3}\bigg(2\operatorname{Ric}(v,w)-s g(v,w)\bigg){\rm tr}[id] {\rm Vol}(S^{n-1}).\nonumber
\end{align}

\noindent{\bf (II-4-$\mathbb{B}$)}

For the same reason, we get 
\begin{align}
	&\sum_{a,b,s,t=1}^{2m}\operatorname{R}_{bats}{\rm tr}\bigg(c(v)c(e_b)c(w)c(e_a) c(e_s) c(e_t)+c(v)c(e_a)c(w)c(e_b) c(e_s) c(e_t)\bigg)=0,
\end{align}
then
\begin{align}
	&\int_{\|\xi\|=1}\operatorname{tr} \biggl\{-\frac{ m }{4}\|\xi\|^{-2 m-2} \sum_{a,b,j,t,s=1}^{2m} \operatorname{R}_{b a t s}\xi_{a}\delta^{b}_{j}\bigg(c(v)c(dx_j)c(w)c(\xi) c(e_s) c(e_t)+c(v)c(\xi)c(w)c(dx_j) c(e_s) c(e_t)\bigg)\biggr\}(x_0)\sigma(\xi)\\
	&=\int_{\|\xi\|=1}\operatorname{tr} \biggl\{-\frac{ m }{4}\|\xi\|^{-2 m-2} \sum_{a,b,t,s,f=1}^{2m} \operatorname{R}_{b a t s}\xi_{a}\xi_{f}\bigg(c(v)c(e_b)c(w)c(e_f) c(e_s) c(e_t)+c(v)c(e_f)c(w)c(e_b) c(e_s) c(e_t)\bigg)\biggr\}(x_0)\sigma(\xi)\nonumber\\
	&=-\frac{1 }{8}\sum_{a,b,t,s=1}^{2m} \operatorname{R}_{b a t s}\operatorname{tr}\bigg(c(v)c(e_b)c(w)c(e_a) c(e_s) c(e_t)+c(v)c(e_a)c(w)c(e_b) c(e_s) c(e_t)\bigg)(x_0){\rm Vol}(S^{n-1})\nonumber\\
	&=\;\;0.\nonumber
\end{align}

\noindent{\bf (II-4-$\mathbb{C}$)}

By \eqref{pt},we get
\begin{align}
	&\sum_{a,b,\hat{j}\neq\hat{l}=1}^{2m}\nabla_{e_b}(T)(e_a,e_{\hat{j}},e_{\hat{l}}){\rm tr}\bigg(c(v)c(e_b)c(w)c(e_a)c(e_{\hat{j}})c(e_{\hat{l}})+c(v)c(e_a)c(w)c(e_b)c(e_{\hat{j}})c(e_{\hat{l}})\bigg)\\
	&=-8\sum_{a=1}^{2m}\nabla_{e_a}(T)(e_a,v,w){\rm tr}[id],\nonumber
\end{align}
then
\begin{align}
	&\int_{\|\xi\|=1}\operatorname{tr} \biggl\{-3m\|\xi\|^{-2 m-2}\sum_{a,b,j,\hat{j}<\hat{l}=1}^{2m}\frac{\partial{T}(\partial_{a},e_{\hat{j}},e_{\hat{l}})}{\partial{x_b}}\xi_{a}\delta^{b}_{j}\times\\
	&\;\;\;\;\;\;\;\;\;\;\;\;\;\;\;\;\times \bigg(c(v)c(dx_j)c(w)c(\xi)c(e_{\hat{j}})c(e_{\hat{l}})+c(v)c(\xi)c(w)c(dx_j)c(e_{\hat{j}})c(e_{\hat{l}})\bigg)\biggr\}(x_0)\sigma(\xi)\nonumber\\
	&=\int_{\|\xi\|=1}\operatorname{tr} \biggl\{-\frac{3m}{2}\sum_{a,b,f,\hat{j}\neq\hat{l}=1}^{2m}\frac{\partial{T}(\partial_{a},e_{\hat{j}},e_{\hat{l}})}{\partial{x_b}}\xi_{a}\xi_{f}\bigg(c(v)c(e_b)c(w)c(e_f)c(e_{\hat{j}})c(e_{\hat{l}})+c(v)c(e_f)c(w)c(e_b)c(e_{\hat{j}})c(e_{\hat{l}})\bigg)\biggr\}(x_0)\sigma(\xi)\nonumber\\
	&=-\frac{3}{4}\sum_{a,b,\hat{j}\neq\hat{l}=1}^{2m}\nabla_{e_b}(T)(e_a,e_{\hat{j}},e_{\hat{l}}){\rm tr}\bigg(c(v)c(e_b)c(w)c(e_a)c(e_{\hat{j}})c(e_{\hat{l}})+c(v)c(e_a)c(w)c(e_b)c(e_{\hat{j}})c(e_{\hat{l}})\bigg)(x_0){\rm Vol}(S^{n-1})\nonumber\\
	&=\;\;6\sum_{a=1}^{2m}\nabla_{e_a}(T)(e_a,v,w){\rm tr}[id] {\rm Vol}(S^{n-1}).\nonumber
\end{align}

Summing from {\bf (II-4-$\mathbb{A}$)} to {\bf (II-4-$\mathbb{C}$)}, we get
\begin{align}
	&\int_{\|\xi\|=1} \operatorname{tr}\bigg[(-i) \sum_{j=1}^{2m} \partial_{\xi_{j}}\left[\sigma_{2}(\mathcal{A} \mathcal{B})\right] \partial_{x_{j}}\left[\sigma_{-2 m-1}\left(D^{-2 m}_T\right)\right]\bigg] (x_0)\sigma(\xi)\\
	&=\;\;\frac{2}{3}\bigg(2\operatorname{Ric}(v,w)-s g(v,w)\bigg){\rm tr}[id] {\rm Vol}(S^{n-1})+6\sum_{a=1}^{2m}\nabla_{e_a}(T)(e_a,v,w){\rm tr}[id] {\rm Vol}(S^{n-1}).\nonumber
\end{align}

\noindent{\bf (II-5)} For $(-i) \sum_{j=1}^{2m} \partial_{\xi_{j}}\left[\sigma_{1}(\mathcal{A} \mathcal{B})\right] \partial_{x_{j}}\left[\sigma_{-2 m}\left(D^{-2 m}_T\right)\right]$:

According to \eqref{lem21} in lemma \ref{lem2}, we get
\begin{align}\label{1-2m}
	&\partial_{x_{j}}\left[\sigma_{-2 m}\left(D^{-2 m}_T\right)\right](x_0)\\
	&=\partial_{x_{j}}\bigg[\|\xi\|^{-2 m-2}\sum_{a,b,l,k=1}^{2m} \left(\delta_{a b}-\frac{m}{3} R_{a l b k} x^{l} x^{k}\right) \xi_{a} \xi_{b}\bigg](x_0)\nonumber\\
	&=\;\;0,\nonumber
\end{align}
then
\begin{align}
	&\int_{\|\xi\|=1} \operatorname{tr}\bigg[(-i) \sum_{j=1}^{2m} \partial_{\xi_{j}}\left[\sigma_{1}(\mathcal{A} \mathcal{B})\right] \partial_{x_{j}}\left[\sigma_{-2 m}\left(D^{-2 m}_T\right)\right]\bigg] (x_0)\sigma(\xi)=0.
\end{align}

\noindent{\bf (II-6)} For $-\frac{1}{2} \sum_{j ,l=1}^{2m} \partial_{\xi_{j}} \partial_{\xi_{l}}\left[\sigma_{2}(\mathcal{A} \mathcal{B})\right] \partial_{x_{j}} \partial_{x_{l}}\left[\sigma_{-2 m}\left(D^{-2 m}_T\right)\right]$:

According to \eqref{lem21} in lemma \ref{lem2} and \eqref{lemAB3} in lemma \ref{AB}, we get
\begin{align}\label{1-2m}
	&-\frac{1}{2} \sum_{j, l=1} ^{2m}\partial_{\xi_{j}} \partial_{\xi_{l}}\left[\sigma_{2}(\mathcal{A} \mathcal{B})\right] \partial_{x_{j}} \partial_{x_{l}}\left[\sigma_{-2 m}\left(D^{-2 m}_T\right)\right](x_0)\\
	&=-\frac{m}{6}\|\xi\|^{-2 m-2}\sum_{a,b,j,l=1}^{2m}\bigg({\rm R}_{albj}+{\rm R}_{ajbl} \bigg)\bigg[c(v)c(dx_l)c(w)c(dx_j)+c(v)c(dx_j)c(w)c(dx_l)\bigg]\xi_{a}\xi_{b}.\nonumber
\end{align}

Since
\begin{align}
	&{\rm tr}\sum_{a,j,l=1}^{2m}{\rm R}_{alaj}\bigg[c(v)c(dx_l)c(w)c(dx_j)+c(v)c(dx_j)c(w)c(dx_l)\bigg]\nonumber\\
	&=\bigg(4{\rm Ric}(v,w)-2s g(v,w)\bigg){\rm tr}[id],\nonumber
\end{align}
then
\begin{align}
	&\int_{\|\xi\|=1} {\rm tr}\biggl\{-\frac{1}{2} \sum_{j, l=1}^{2m} \partial_{\xi_{j}} \partial_{\xi_{l}}\left[\sigma_{2}(\mathcal{A} \mathcal{B})\right] \partial_{x_{j}} \partial_{x_{l}}\left[\sigma_{-2 m}\left(D^{-2 m}_T\right)\right]\biggr\}(x_0)\sigma(\xi)\\
	&=\int_{\|\xi\|=1} \biggl\{-\frac{m}{6}\|\xi\|^{-2 m-2}\sum_{a,b,j,l=1}^{2m}\bigg({\rm R}_{albj}+{\rm R}_{ajbl} \bigg){\rm tr}\bigg[c(v)c(dx_l)c(w)c(dx_j)+c(v)c(dx_j)c(w)c(dx_l)\bigg]\xi_{a}\xi_{b}\biggr\}(x_0)\sigma(\xi)\nonumber\\
	&=-\frac{1}{6}\sum_{a,j,l=1}^{2m}{\rm R}_{alaj}{\rm tr}\bigg[c(v)c(dx_l)c(w)c(dx_j)+c(v)c(dx_j)c(w)c(dx_l)\bigg](x_0){\rm Vol}(S^{n-1})\nonumber\\
	&=-\frac{1}{3}\bigg(2{\rm Ric}(v,w)-s g(v,w)\bigg){\rm tr}[id] {\rm Vol}(S^{n-1}).\nonumber
\end{align}

Summing from {\bf (II-1)} to {\bf (II-6)}, we get
\begin{align}\label{zabdt}
	&\int_{\|\xi\|=1} {\rm tr}\biggl\{	\sigma_{-2 m}\left(\mathcal{A} \mathcal{B} D^{-2 m}_T\right)\biggr\}(x_0)\sigma(\xi)\\
	&=-\frac{1}{6}\bigg({\rm Ric}(v,w)-\frac{1}{2} s g(v,w)\bigg){\rm tr}[id] {\rm Vol}(S^{n-1})-\frac{m-1}{12} s g(v,w){\rm tr}[id]{\rm Vol}(S^{n-1})\nonumber\\
	&\;\;\;\;-\frac{3(7-m)}{4}\|T\|^2 g(v,w){\rm tr}[id]{\rm Vol}(S^{n-1})+\frac{9}{2}\sum_{j,l=1}^{2m} T(v,e_{j},e_{l})T(w,e_{j},e_{l}){\rm tr}[id]{\rm Vol}(S^{n-1})\nonumber\\
	&\;\;\;\;+\frac{3}{2}\sum_{a=1}^{2m} \nabla_{e_a}(T)(e_a,v,w){\rm tr}[id]{\rm Vol}(S^{n-1})+3\sum_{j=1}^{2m} T(v,\nabla_{e_j}^{L}w,e_j){\rm tr}[id]{\rm Vol}(S^{n-1}).\nonumber
\end{align}
Since ${\rm tr}[id]=2^{m}$ and ${\rm Vol}(S^{n-1})=\frac{2 \pi^{m}}{\Gamma\left(m\right)}$, we obtain
\begin{align}\label{z2}
	\mathscr{B}_{2}=&\mathrm{Wres}\bigg(c(v)D_Tc(w) D_T D_T^{-n}\bigg)\\
	=&\;2^{m} \frac{2 \pi^{m}}{\Gamma\left(m\right)}\int_{M}\biggl\{-\frac{1}{6}\bigg({\rm Ric}(v,w)-\frac{1}{2} s g(v,w)\bigg)-\frac{m-1}{12} s g(v,w)-\frac{3(7-m)}{4}\|T\|^2 g(v,w)\nonumber\\
	&+\frac{9}{2}\sum_{j,l=1}^{2m} T(v,e_{j},e_{l})T(w,e_{j},e_{l})+\frac{3}{2}\sum_{a=1}^{2m} \nabla_{e_a}(T)(e_a,v,w)+3\sum_{j=1}^{2m} T(v,\nabla_{e_j}^{L}w,e_j)\biggr\}d{\rm Vol}_M.\nonumber
\end{align}
Hence Theorem \ref{thm} holds.
\end{proof}

As an application of Theorem \ref{thm}, we give the following definition.
\begin{defn}
	For a closed oriented manifold $M$, the spectral Einstein functional with torsion on $M$ is defined by 
	\begin{align}
		\mathscr{B}(v,w):= \;&2^{m} \frac{2 \pi^{m}}{\Gamma\left(m\right)}\int_{M}\biggl\{-\frac{1}{6}\mathbb{G}(v,w)-\frac{9}{2}\|T\|^2 g(v,w)+\frac{9}{2}\sum_{j,l=1}^{2m} T(v,e_{j},e_{l})T(w,e_{j},e_{l})\\
		&\;+\frac{3}{2}\sum_{a=1}^{2m} \nabla_{e_a}(T)(e_a,v,w)+3\sum_{j=1}^{2m} T(v,\nabla_{e_j}^{L}w,e_j)\biggr\}d{\rm Vol}_M,\nonumber
	\end{align}
 where  $g(v,w)=\sum_{a=1}^{n}v_{a} w_{a} $ and $ \mathbb{G}(v,w)=\operatorname{Ric}(v,w)-\frac{1}{2} s g(v,w) $  , $v=\sum_{a=1}^{n} v_{a}e_{a}, w=\sum_{b=1}^{n} w_{b}e_{b}.$
\end{defn}

\appendix
\section*{Appendix}
\setcounter{equation}{0}
\renewcommand\theequation{A.\arabic{equation}}
\renewcommand{\thethm}{A.\arabic{thm}}  
\setcounter{thm}{0} %

In this appendix, we will prove some facts used in the computation of the trace.
\begin{lem}\label{A}
    \begin{align}
    	&{\bf (1)}\;{\rm tr}\bigg[c(X_1)c( X_2)\bigg]=-g(X_1,  X_2){\rm tr}[id];\label{trace2}\\
    	&{\bf (2)}\;{\rm tr}\bigg[c(X_1)c( X_2)c( X_3)c( X_4)\bigg]\label{trace4}\\
    	&=\bigg[g(X_1,  X_4)g( X_2,  X_3)-g(X_1,  X_3)g( X_2,  X_4)+g(X_1,  X_2)g( X_3,  X_4)\bigg]{\rm tr}[id];\nonumber\\
    	&{\bf (3)}\;{\rm tr}\bigg[c(X_1)c( X_2)c( X_3)c( X_4)c( X_5)c( X_6)\bigg]\label{trace6}\\
    	&=\bigg[-g(X_1,  X_6)g ( X_2,  X_5)g ( X_3,  X_4)+g(X_1,  X_6)g ( X_2,  X_4)g ( X_3,  X_5)-g(X_1,  X_6)g ( X_2,  X_3)g ( X_4,  X_5)\nonumber\\
    	&+g(X_1, X_5)g ( X_2,  X_6)g ( X_3, X_4)-g(X_1, X_5)g ( X_2,  X_4)g ( X_3, X_6)+g(X_1, X_5)g ( X_2, X_3)g ( X_4, X_6)\nonumber\\
    	&-g(X_1, X_4)g ( X_2, X_6)g ( X_3, X_5)+g(X_1, X_4)g ( X_2, X_5)g ( X_3, X_6)-g(X_1, X_4)g ( X_2, X_3)g ( X_5, X_6)\nonumber\\
    	&+g(X_1, X_3)g ( X_2, X_6)g ( X_4, X_5)-g(X_1, X_3)g ( X_2, X_5)g ( X_4, X_6)+g(X_1, X_3)g ( X_2, X_4)g ( X_5, X_6)\nonumber\\
    	&-g(X_1, X_2)g ( X_3, X_6)g ( X_4, X_5)+g(X_1, X_2)g ( X_3, X_5)g ( X_4, X_6)-g(X_1, X_2)g ( X_3, X_4)g ( X_5, X_6)\bigg]{\rm tr}[id];\nonumber\\
    	&{\bf (4)}\;{\rm tr}\bigg[c(X_1)c( X_2)c( X_3)c( X_4)c( X_5)c( X_6)c( X_7)c( X_8)\bigg]\label{trace8}\\
    	&=-g(X_1,  X_8){\rm tr}\bigg[c( X_2)c( X_3)c( X_4)c( X_5)c( X_6)c( X_7)\bigg]+g(X_1,  X_7){\rm tr}\bigg[c( X_2)c( X_3)c( X_4)c( X_5)c( X_6)c( X_8)\bigg]\nonumber\\
    	&-g(X_1,  X_6){\rm tr}\bigg[c( X_2)c( X_3)c( X_4)c( X_5)c( X_7)c( X_8)\bigg]+g(X_1,  X_5){\rm tr}\bigg[c( X_2)c( X_3)c( X_4)c( X_6)c( X_7)c( X_8)\bigg]\nonumber\\
    	&-g(X_1,  X_4){\rm tr}\bigg[c( X_2)c( X_3)c( X_5)c( X_6)c( X_7)c( X_8)\bigg]+g(X_1,  X_3){\rm tr}\bigg[c( X_2)c( X_4)c( X_5)c( X_6)c( X_7)c( X_8)\bigg]\nonumber\\
    	&-g(X_1,  X_2){\rm tr}\bigg[c( X_3)c( X_4)c( X_5)c( X_6)c( X_7)c( X_8)\bigg],\nonumber
    \end{align}
where $X_1, X_2, X_3, X_4, X_5, X_6, X_7, X_8\in \Gamma(TM).$
\end{lem}

\begin{proof}
{\bf (1)}
According to $c(e_{i})c(e_{j})+c(e_{j})c(e_{i})=-2g(e_{i}, e_{j})=-2\delta_i^j$ and ${\rm tr}\mathcal{XY}={\rm tr}\mathcal{YX}$, we get
\begin{align}
	{\rm tr}\bigg[c(X_1)c( X_2)\bigg]&=\frac{1}{2}{\rm tr}\bigg[c(X_1)c( X_2)+c(X_2)c( X_1)\bigg]\\
	&=\frac{1}{2}{\rm tr}\bigg[-2g(X_1,X_2)\bigg]\nonumber\\
	&=-g(X_1,  X_2){\rm tr}[id]\nonumber.
\end{align}
{\bf (2)}
According to $c(e_{i})c(e_{j})+c(e_{j})c(e_{i})=-2g(e_{i}, e_{j})=-2\delta_i^j$, interchange $c(X_1)$ with $c(X_2)$, $c(X_3)$ and $c(X_4)$ in that order, we get
\begin{align}\label{rr3}
	&{\rm tr}\bigg[c(X_1)c( X_2)c( X_3)c( X_4)\bigg]\\
	&=-{\rm tr}\bigg[c( X_2)c(X_1)c( X_3)c( X_4)\bigg]-2g(X_1,  X_2){\rm tr}\bigg[c( X_3)c( X_4)\bigg]\nonumber\\
	&={\rm tr}\bigg[c( X_2)c( X_3)c(X_1)c( X_4)\bigg]+2g(X_1,  X_3){\rm tr}\bigg[c( X_2)c( X_4)\bigg]-2g(X_1,  X_2){\rm tr}\bigg[c( X_3)c( X_4)\bigg]\nonumber\\
	&=-{\rm tr}\bigg[c( X_2)c( X_3)c( X_4)c(X_1)\bigg]-2g(X_1,  X_4){\rm tr}\bigg[c( X_2)c( X_3)\bigg]\nonumber\\
	&\;\;\;\;+2g(X_1,  X_3){\rm tr}\bigg[c( X_2)c( X_4)\bigg]-2g(X_1,  X_2){\rm tr}\bigg[c( X_3)c( X_4)\bigg]\nonumber,
\end{align}
 since ${\rm tr}\mathcal{XY}={\rm tr}\mathcal{YX}$, then ${\rm tr}\bigg[c( X_2)c( X_3)c( X_4)c(X_1)\bigg]={\rm tr}\bigg[c(X_1)c( X_2)c( X_3)c( X_4)\bigg]$, hence
\begin{align}\label{rrr3}
	&{\rm tr}\bigg[c(X_1)c( X_2)c( X_3)c( X_4)\bigg]\\
	&=-g(X_1,  X_4){\rm tr}\bigg[c( X_2)c( X_3)\bigg]+g(X_1,  X_3){\rm tr}\bigg[c( X_2)c( X_4)\bigg]-g(X_1,  X_2){\rm tr}\bigg[c( X_3)c( X_4)\bigg]\nonumber\\
	&=\bigg[g(X_1,  X_4)g( X_2,  X_3)-g(X_1,  X_3)g( X_2,  X_4)+g(X_1,  X_2)g( X_3,  X_4)\bigg]{\rm tr}[id]\nonumber.
\end{align}

{\bf (3)}Similarly, use $c(e_{i})c(e_{j})+c(e_{j})c(e_{i})=-2g(e_{i}, e_{j})=-2\delta_i^j$, interchange $c(X_1)$ with $c(X_2)$ ... $c(X_6)$ in that order, we get 
\begin{align}\label{rr4}
	&{\rm tr}\bigg[c(X_1)c( X_2)c( X_3)c( X_4)c( X_5)c( X_6)\bigg]\\
	=&-{\rm tr}\bigg[c( X_2)c(X_1)c( X_3)c( X_4)c( X_5)c( X_6)\bigg]-2g(X_1,  X_2){\rm tr}\bigg[c( X_3)c( X_4)c( X_5)c( X_6)\bigg]\nonumber\\
	=&\;\;{\rm tr}\bigg[c( X_2)c( X_3)c(X_1)c( X_4)c( X_5)c( X_6)\bigg]+2g(X_1,  X_3){\rm tr}\bigg[c( X_2)c( X_4)c( X_5)c( X_6)\bigg]\nonumber\\
	&\;\;-2g(X_1,  X_2){\rm tr}\bigg[c( X_3)c( X_4)c( X_5)c( X_6)\bigg]\nonumber\\
	=&\;\;...\nonumber\\
	=&-{\rm tr}\bigg[c( X_2)c( X_3)c( X_4)c( X_5)c( X_6)c(X_1)\bigg]-2g(X_1,  X_6){\rm tr}\bigg[c( X_2)c( X_3)c( X_4)c( X_5)\bigg]\nonumber\\
	&\;\;+2g(X_1,  X_5){\rm tr}\bigg[c( X_2)c( X_3)c( X_4)c( X_6)\bigg]-2g(X_1,  X_4){\rm tr}\bigg[c( X_2)c( X_3)c( X_5)c( X_6)\bigg]\nonumber\\
	&\;\;+2g(X_1,  X_3){\rm tr}\bigg[c( X_2)c( X_4)c( X_5)c( X_6)\bigg]-2g(X_1,  X_2){\rm tr}\bigg[c( X_3)c( X_4)c( X_5)c( X_6)\bigg],\nonumber
\end{align}
 since ${\rm tr}\mathcal{XY}={\rm tr}\mathcal{YX}$, then ${\rm tr}\bigg[c( X_2)c( X_3)c( X_4)c( X_5)c( X_6)c(X_1)\bigg]={\rm tr}\bigg[c(X_1)c( X_2)c( X_3)c( X_4)c( X_5)c( X_6)\bigg]$, we get
\begin{align}\label{trace41}
	&{\rm tr}\bigg[c(X_1)c( X_2)c( X_3)c( X_4)c( X_5)c( X_6)\bigg]\\
	=&-g(X_1,  X_6){\rm tr}\bigg[c( X_2)c( X_3)c( X_4)c( X_5)\bigg]+g(X_1,  X_5){\rm tr}\bigg[c( X_2)c( X_3)c( X_4)c( X_6)\bigg]\nonumber\\
	&-g(X_1,  X_4){\rm tr}\bigg[c( X_2)c( X_3)c( X_5)c( X_6)\bigg]+g(X_1,  X_3){\rm tr}\bigg[c( X_2)c( X_4)c( X_5)c( X_6)\bigg]\nonumber\\
	&-g(X_1,  X_2){\rm tr}\bigg[c( X_3)c( X_4)c( X_5)c( X_6)\bigg]\nonumber\\
	=&\bigg[-g(X_1,  X_6)g ( X_2,  X_5)g ( X_3,  X_4)+g(X_1,  X_6)g ( X_2,  X_4)g ( X_3,  X_5)-g(X_1,  X_6)g ( X_2,  X_3)g ( X_4,  X_5)\nonumber\\
	&+g(X_1, X_5)g ( X_2,  X_6)g ( X_3, X_4)-g(X_1, X_5)g ( X_2,  X_4)g ( X_3, X_6)+g(X_1, X_5)g ( X_2, X_3)g ( X_4, X_6)\nonumber\\
	&-g(X_1, X_4)g ( X_2, X_6)g ( X_3, X_5)+g(X_1, X_4)g ( X_2, X_5)g ( X_3, X_6)-g(X_1, X_4)g ( X_2, X_3)g ( X_5, X_6)\nonumber\\
	&+g(X_1, X_3)g ( X_2, X_6)g ( X_4, X_5)-g(X_1, X_3)g ( X_2, X_5)g ( X_4, X_6)+g(X_1, X_3)g ( X_2, X_4)g ( X_5, X_6)\nonumber\\
	&-g(X_1, X_2)g ( X_3, X_6)g ( X_4, X_5)+g(X_1, X_2)g ( X_3, X_5)g ( X_4, X_6)-g(X_1, X_2)g ( X_3, X_4)g ( X_5, X_6)\bigg]{\rm tr}[id].\nonumber
\end{align}

{\bf (4)} Interchange $c(X_1)$ with $c(X_2)$ ... $c(X_8)$ in that order, we get 
\begin{align}\label{trace81}
	&{\rm tr}\bigg[c(X_1)c( X_2)c( X_3)c( X_4)c( X_5)c( X_6)c( X_7)c( X_8)\bigg]\\
	=&-{\rm tr}\bigg[c( X_2)c(X_1)c( X_3)c( X_4)c( X_5)c( X_6)c( X_7)c( X_8)\bigg]-2g(X_1,  X_2){\rm tr}\bigg[c( X_3)c( X_4)c( X_5)c( X_6)c( X_7)c( X_8)\bigg]\nonumber\\
	=&\;\;{\rm tr}\bigg[c( X_2)c( X_3)c(X_1)c( X_4)c( X_5)c( X_6)c( X_7)c( X_8)\bigg]+2g(X_1,  X_3){\rm tr}\bigg[c( X_2)c( X_4)c( X_5)c( X_6)c( X_7)c( X_8)\bigg]\nonumber\\
	&\;\;-2g(X_1,  X_2){\rm tr}\bigg[c( X_3)c( X_4)c( X_5)c( X_6)c( X_7)c( X_8)\bigg]\nonumber\\
	=&\;\;...\nonumber\\
	=&-{\rm tr}\bigg[c( X_2)c( X_3)c( X_4)c( X_5)c( X_6)c( X_7)c( X_8)c(X_1)\bigg]-2g(X_1,  X_8){\rm tr}\bigg[c( X_2)c( X_3)c( X_4)c( X_5)c( X_6)c( X_7)\bigg]\nonumber\\
	&+2g(X_1,  X_7){\rm tr}\bigg[c( X_2)c( X_3)c( X_4)c( X_5)c( X_6)c( X_8)\bigg]-2g(X_1,  X_6){\rm tr}\bigg[c( X_2)c( X_3)c( X_4)c( X_5)c( X_7)c( X_8)\bigg]\nonumber\\
	&\;\;+2g(X_1,  X_5){\rm tr}\bigg[c( X_2)c( X_3)c( X_4)c( X_6)c( X_7)c( X_8)\bigg]-2g(X_1,  X_4){\rm tr}\bigg[c( X_2)c( X_3)c( X_5)c( X_6)c( X_7)c( X_8)\bigg]\nonumber\\
	&\;\;+2g(X_1,  X_3){\rm tr}\bigg[c( X_2)c( X_4)c( X_5)c( X_6)c( X_7)c( X_8)\bigg]-2g(X_1,  X_2){\rm tr}\bigg[c( X_3)c( X_4)c( X_5)c( X_6)c( X_7)c( X_8)\bigg],\nonumber
\end{align}
since ${\rm tr}\mathcal{XY}={\rm tr}\mathcal{YX}$, \\
then ${\rm tr}\bigg[c( X_2)c( X_3)c( X_4)c( X_5)c( X_6)c( X_7)c( X_8)c(X_1)\bigg]={\rm tr}\bigg[c(X_1)c( X_2)c( X_3)c( X_4)c( X_5)c( X_6)c( X_7)c( X_8)\bigg]$, thus
\begin{align}\label{trace82}
	&{\rm tr}\bigg[c(X_1)c( X_2)c( X_3)c( X_4)c( X_5)c( X_6)c( X_7)c( X_8)\bigg]\\
	=&-g(X_1,  X_8){\rm tr}\bigg[c( X_2)c( X_3)c( X_4)c( X_5)c( X_6)c( X_7)\bigg]+g(X_1,  X_7){\rm tr}\bigg[c( X_2)c( X_3)c( X_4)c( X_5)c( X_6)c( X_8)\bigg]\nonumber\\
	&-g(X_1,  X_6){\rm tr}\bigg[c( X_2)c( X_3)c( X_4)c( X_5)c( X_7)c( X_8)\bigg]+g(X_1,  X_5){\rm tr}\bigg[c( X_2)c( X_3)c( X_4)c( X_6)c( X_7)c( X_8)\bigg]\nonumber\\
	&-g(X_1,  X_4){\rm tr}\bigg[c( X_2)c( X_3)c( X_5)c( X_6)c( X_7)c( X_8)\bigg]+g(X_1,  X_3){\rm tr}\bigg[c( X_2)c( X_4)c( X_5)c( X_6)c( X_7)c( X_8)\bigg]\nonumber\\
	&-g(X_1,  X_2){\rm tr}\bigg[c( X_3)c( X_4)c( X_5)c( X_6)c( X_7)c( X_8)\bigg].\nonumber
\end{align}

\end{proof}
\noindent{\bf (I-$\mathbb{B}$)}\eqref{chu}:

\begin{align}
	&\operatorname{tr}\bigg(\sum_{j\neq l,\hat{j}\neq \hat{l}=1}^{2m}c(v)c(w)c(e_j)c(e_l)c(e_{\hat{j}})c(e_{\hat{l}})\bigg)\nonumber\\
	&=\sum_{j\neq l,\hat{j}\neq \hat{l}=1}^{2m}\bigg[v_{\hat{l}} w_l \delta_{j}^{\hat{j}}-v_{\hat{l}} w_j \delta_{l}^{\hat{j}}-v_{\hat{j}} w_l \delta_{j}^{\hat{l}}+v_{\hat{j}} w_j \delta_{l}^{\hat{l}}-v_l w_{\hat{l}} \delta_{j}^{\hat{j}}+v_l w_{\hat{j}} \delta_{j}^{\hat{l}}+v_j w_{\hat{l}} \delta_{l}^{\hat{j}}-v_j w_{\hat{j}} \delta_{l}^{\hat{l}}\nonumber\\
	&\;\;\;\;\;\;\;\;\;\;\;\;\;\;\;\;\;-\delta_{j}^{\hat{l}}\delta_{l}^{\hat{j}}g(v,w)+\delta_{j}^{\hat{j}}\delta_{l}^{\hat{l}}g(v,w)\bigg]{\rm tr}[id].\nonumber
\end{align}

\begin{proof}
According to \eqref{trace6} in lemma \ref{A}, $g(e_{i}, e_{j})=\delta_i^j$, $g(v,e_i)=v_i$ and $ g(w,e_j)=w_j$,  we get
\begin{align}\label{Achu}
	&\operatorname{tr}\bigg(\sum_{j\neq l,\hat{j}\neq \hat{l}=1}^{2m}c(v)c(w)c(e_j)c(e_l)c(e_{\hat{j}})c(e_{\hat{l}})\bigg)\\
	&=\sum_{j\neq l,\hat{j}\neq \hat{l}=1}^{2m}\bigg[-g(v,  e_{\hat{l}})g ( w,  e_{\hat{j}})g ( e_j,  e_l)+g(v,  e_{\hat{l}})g ( w,  e_l)g ( e_j,  e_{\hat{j}})-g(v,  e_{\hat{l}})g ( w,  e_j)g ( e_l,  e_{\hat{j}})\nonumber\\
	&\;\;\;\;\;\;\;\;\;\;\;\;\;\;\;\;\;+g(v, e_{\hat{j}})g ( w,  e_{\hat{l}})g ( e_j, e_l)-g(v, e_{\hat{j}})g ( w,  e_l)g ( e_j, e_{\hat{l}})+g(v, e_{\hat{j}})g ( w, e_j)g ( e_l, e_{\hat{l}})\nonumber\\
	&\;\;\;\;\;\;\;\;\;\;\;\;\;\;\;\;\;-g(v, e_l)g ( w, e_{\hat{l}})g ( e_j, e_{\hat{j}})+g(v, e_l)g ( w, e_{\hat{j}})g ( e_j, e_{\hat{l}})-g(v, e_l)g ( w, e_j)g ( e_{\hat{j}}, e_{\hat{l}})\nonumber\\
	&\;\;\;\;\;\;\;\;\;\;\;\;\;\;\;\;\;+g(v, e_j)g ( w, e_{\hat{l}})g ( e_l, e_{\hat{j}})-g(v, e_j)g ( w, e_{\hat{j}})g ( e_l, e_{\hat{l}})+g(v, e_j)g ( w, e_l)g ( e_{\hat{j}}, e_{\hat{l}})\nonumber\\
	&\;\;\;\;\;\;\;\;\;\;\;\;\;\;\;\;\;-g(v, w)g ( e_j, e_{\hat{l}})g ( e_l, e_{\hat{j}})+g(v, w)g ( e_j, e_{\hat{j}})g ( e_l, e_{\hat{l}})\bigg]{\rm tr}[id]\nonumber \\
	&=\sum_{j\neq l,\hat{j}\neq \hat{l}=1}^{2m}\bigg[v_{\hat{l}} w_l \delta_{j}^{\hat{j}}-v_{\hat{l}} w_j \delta_{l}^{\hat{j}}-v_{\hat{j}} w_l \delta_{j}^{\hat{l}}+v_{\hat{j}} w_j \delta_{l}^{\hat{l}}-v_l w_{\hat{l}} \delta_{j}^{\hat{j}}+v_l w_{\hat{j}} \delta_{j}^{\hat{l}}+v_j w_{\hat{l}} \delta_{l}^{\hat{j}}-v_j w_{\hat{j}} \delta_{l}^{\hat{l}}\nonumber\\
	&\;\;\;\;\;\;\;\;\;\;\;\;\;\;\;\;\;-\delta_{j}^{\hat{l}}\delta_{l}^{\hat{j}}g(v,w)+\delta_{j}^{\hat{j}}\delta_{l}^{\hat{l}}g(v,w)\bigg]{\rm tr}[id].\nonumber
\end{align}

Therefore, \eqref{chu} holds.
\end{proof}

\noindent{\bf (I-$\mathbb{B}$)}\eqref{1b}:

\begin{align}
	&\int_{\|\xi\|=1} {\rm tr}\biggl\{-\frac{9m(m-1)}{2}\|\xi\|^{-2 m-2}\sum_{a,b,j<l,\hat{j}<\hat{l}=1}^{2m}T(e_a,e_j,e_l)T(e_b, e_{\hat{j}},e_{\hat{l}})c(v)c(w)c(e_j)c(e_l)c(e_{\hat{j}})c(e_{\hat{l}})\xi_{a}\xi_b\biggr\}(x_0)\sigma(\xi)\nonumber\\
	&=-\frac{9(m-1)}{16}\sum_{a,b,j\neq l,\hat{j}\neq \hat{l}=1}^{2m}T(e_a,e_j,e_l)T(e_a, e_{\hat{j}},e_{\hat{l}}){\rm tr}\bigg(c(v)c(w)c(e_j)c(e_l)c(e_{\hat{j}})c(e_{\hat{l}})\bigg){\rm Vol}(S^{n-1})\nonumber\\
	&=-\frac{9(m-1)}{8}\sum_{a,j\neq l=1}^{2m}T^2(e_a,e_j,e_l)g(v,w){\rm tr}[id]{\rm Vol}(S^{n-1}).\nonumber
\end{align}

\begin{proof}
Based on $\int_{\|\xi\|=1} \xi_{a} \xi_{b}\sigma(\xi)=\frac{1}{n}\delta_{a}^{b}{\rm Vol}(S^{n-1})$ and $\sum_{j<l,a=1}^{2m}T(e_a,e_j,e_l)=\frac{1}{2} \sum_{j,l,a=1}^{2m}T(e_a,e_j,e_l)$, where $n=2m$, then

\begin{align}\label{a1b}
	&\int_{\|\xi\|=1} {\rm tr}\biggl\{-\frac{9m(m-1)}{2}\|\xi\|^{-2 m-2}\sum_{a,b,j<l,\hat{j}<\hat{l}=1}^{2m}T(e_a,e_j,e_l)T(e_b, e_{\hat{j}},e_{\hat{l}})c(v)c(w)c(e_j)c(e_l)c(e_{\hat{j}})c(e_{\hat{l}})\xi_{a}\xi_b\biggr\}(x_0)\sigma(\xi)\\
	&=-\frac{9(m-1)}{16}\sum_{a,b,j\neq l,\hat{j}\neq \hat{l}=1}^{2m}T(e_a,e_j,e_l)T(e_a, e_{\hat{j}},e_{\hat{l}}){\rm tr}\bigg(c(v)c(w)c(e_j)c(e_l)c(e_{\hat{j}})c(e_{\hat{l}})\bigg){\rm Vol}(S^{n-1})\nonumber\\
	&=-\frac{9(m-1)}{16}\sum_{a,b,j\neq l,\hat{j}\neq \hat{l}=1}^{2m}T(e_a,e_j,e_l)T(e_a, e_{\hat{j}},e_{\hat{l}})\bigg[v_{\hat{l}} w_l \delta_{j}^{\hat{j}}-v_{\hat{l}} w_j \delta_{l}^{\hat{j}}-v_{\hat{j}} w_l \delta_{j}^{\hat{l}}+v_{\hat{j}} w_j \delta_{l}^{\hat{l}}-v_l w_{\hat{l}} \delta_{j}^{\hat{j}}\nonumber\\
	&\;\;\;\;\;\;\;\;\;\;\;\;\;\;\;\;\;+v_l w_{\hat{j}} \delta_{j}^{\hat{l}}+v_j w_{\hat{l}} \delta_{l}^{\hat{j}}-v_j w_{\hat{j}} \delta_{l}^{\hat{l}}-\delta_{j}^{\hat{l}}\delta_{l}^{\hat{j}}g(v,w)+\delta_{j}^{\hat{j}}\delta_{l}^{\hat{l}}g(v,w)\bigg]{\rm tr}[id]{\rm Vol}(S^{n-1})\nonumber\\
	&=-\frac{9(m-1)}{16}\sum_{a,j\neq l=1}^{2m}\bigg[T(e_a,e_j,w)T(e_a,e_j,v)-T(e_a,w,e_l)T(e_a,e_l,v)-T(e_a,e_j,w)T(e_a,v,e_j)\nonumber\\
	&+T(e_a,w,e_l)T(e_a,v,e_l)-T(e_a,e_j,v)T(e_a,e_j,w)+T(e_a,e_j,v)T(e_a,w,e_j)+T(e_a,v,e_l)T(e_a,e_l,w)\nonumber\\
	&+T(e_a,v,e_l)T(e_a,w,e_l)-T(e_a,e_j,e_l)T(e_a,e_l,e_j)g(v,w)+T(e_a,e_j,e_l)T(e_a,e_j,e_l)g(v,w)\bigg]{\rm tr}[id]{\rm Vol}(S^{n-1})\nonumber\\
	&=-\frac{9(m-1)}{8}\sum_{a,j\neq l=1}^{2m}T^2(e_a,e_j,e_l)g(v,w){\rm tr}[id]{\rm Vol}(S^{n-1}).\nonumber
\end{align}
Therefore, \eqref{1b} holds.
\end{proof}

\noindent{\bf (I-$\mathbb{D}$)}\eqref{pt} and {\bf (II-1-$\mathbb{C}$)}\eqref{pt2}:

\begin{align}
	\frac{\partial{T}(\partial_{a},e_j,e_l)} {\partial x_a}(x_0)=\nabla_{e_a}(T)(e_{a},e_{j},e_{l})(x_0),\nonumber
\end{align}
and
\begin{align}
	\frac{\partial{T}(e_{f},e_{\alpha},e_{\beta})} {\partial x_j}(x_0)=\nabla_{e_j}(T)(e_{f},e_{\alpha},e_{\beta})(x_0).\nonumber
\end{align}

\begin{proof}
	
\begin{align}
	\nabla_{e_j}(T)(e_{\alpha},e_{\beta},e_{\gamma})&={e_j}\bigg(T(e_{\alpha},e_{\beta},e_{\gamma})\bigg)-T(\nabla_{e_j}e_{\alpha},e_{\beta},e_{\gamma})\\
	&-T(e_{\alpha},\nabla_{e_j}e_{\beta},e_{\gamma})-T(e_{\alpha},e_{\beta},\nabla_{e_j}e_{\gamma}),\nonumber
\end{align}
since in the normal coordinates, $\nabla_{e_j}e_{\alpha}(x_0)=\nabla_{e_j}e_{\beta}(x_0)=\nabla_{e_j}e_{\gamma}(x_0)=0$, we have
\begin{align}
	\nabla_{e_j}(T)(e_{\alpha},e_{\beta},e_{\gamma})(x_0)&={e_j}\bigg(T(e_{\alpha},e_{\beta},e_{\gamma})\bigg)(x_0)\\
	&=\frac{\partial T(e_{\alpha},e_{\beta},e_{\gamma})}{\partial x_j}(x_0).\nonumber
\end{align}
Therefore, \eqref{pt} and \eqref{pt2} holds.
\end{proof}

\noindent{\bf (II-1-$\mathbb{A}$)}\eqref{t1}:

\begin{align}
	&{\rm tr} \bigg(\sum_{f\neq\alpha\neq\beta,\hat{f}\neq\hat{\alpha}\neq\hat{\beta}=1}^{2m}c(v)c(e_{f})c(e_{\alpha})c(e_{\beta})c(w)c(e_{\hat{f}})c(e_{\hat{\alpha}})c(e_{\hat{\beta}})\bigg)\nonumber\\
	&=\sum_{f\neq\alpha\neq\beta,\hat{f}\neq\hat{\alpha}\neq\hat{\beta}=1}^{2m}\bigg[v_f \delta_{\alpha}^{\hat{\beta}} \delta_{\hat{\alpha}}^{\beta} w_{\hat{f}}-v_f \delta_{\alpha}^{\hat{\beta}} \delta_{\hat{f}}^{\beta} w_{\hat{\alpha}}-v_f \delta_{\alpha}^{\hat{\alpha}} \delta_{\hat{\beta}}^{\beta} w_{\hat{f}}+v_f \delta_{\alpha}^{\hat{\alpha}} \delta_{\hat{f}}^{\beta} w_{\hat{\beta}}+v_f \delta_{\alpha}^{\hat{f}} \delta_{\hat{\beta}}^{\beta} w_{\hat{\alpha}}-v_f \delta_{\alpha}^{\hat{f}} \delta_{\hat{\alpha}}^{\beta} w_{\hat{\beta}}\nonumber\\
	&\;\;\;\;\;\;\;\;\;\;\;\;\;\;\;\;\;\;\;\;\;\;\;\;\;\;\;\;-v_{\alpha} \delta_{f}^{\hat{\beta}} \delta_{\hat{\alpha}}^{\beta} w_{\hat{f}}+v_{\alpha} \delta_{f}^{\hat{\beta}} \delta_{\hat{f}}^{\beta} w_{\hat{\alpha}}+v_{\alpha} \delta_{f}^{\hat{\alpha}} \delta_{\hat{\beta}}^{\beta} w_{\hat{f}}-v_{\alpha} \delta_{f}^{\hat{\alpha}} \delta_{\hat{f}}^{\beta} w_{\hat{\beta}}-v_{\alpha} \delta_{f}^{\hat{f}} \delta_{\hat{\beta}}^{\beta} w_{\hat{\alpha}}+v_{\alpha} \delta_{f}^{\hat{f}} \delta_{\hat{\alpha}}^{\beta} w_{\hat{\beta}}\nonumber\\
	&\;\;\;\;\;\;\;\;\;\;\;\;\;\;\;\;\;\;\;\;\;\;\;\;\;\;\;\;+v_{\beta} \delta_{f}^{\hat{\beta}} \delta_{\hat{\alpha}}^{\alpha} w_{\hat{f}}-v_{\beta} \delta_{f}^{\hat{\beta}} \delta_{\hat{f}}^{\alpha} w_{\hat{\alpha}}-v_{\beta} \delta_{f}^{\hat{\alpha}} \delta_{\hat{\beta}}^{\alpha} w_{\hat{f}}+v_{\beta} \delta_{f}^{\hat{\alpha}} \delta_{\hat{f}}^{\alpha} w_{\hat{\beta}}+v_{\beta} \delta_{f}^{\hat{f}} \delta_{\hat{\beta}}^{\alpha} w_{\hat{\alpha}}-v_{\beta} \delta_{f}^{\hat{f}} \delta_{\hat{\alpha}}^{\alpha} w_{\hat{\beta}}\nonumber\\
	&\;\;\;\;\;\;\;\;\;\;\;\;\;\;\;\;\;\;\;\;\;\;\;\;\;\;\;\;+v_{\hat{f}} \delta_{f}^{\hat{\beta}} \delta_{\hat{\alpha}}^{\alpha} w_{\beta}-v_{\hat{f}} \delta_{f}^{\hat{\beta}} \delta_{\hat{\alpha}}^{\beta} w_{\alpha}-v_{\hat{f}} \delta_{f}^{\hat{\alpha}} \delta_{\hat{\beta}}^{\alpha} w_{\beta}+v_{\hat{f}} \delta_{f}^{\hat{\alpha}} \delta_{\hat{\beta}}^{\beta} w_{\alpha}+v_{\hat{f}} \delta_{\beta}^{\hat{\alpha}} \delta_{\hat{\beta}}^{\alpha} w_{f}-v_{\hat{f}} \delta_{\beta}^{\hat{\beta}} \delta_{\hat{\alpha}}^{\alpha} w_{f}\nonumber\\
	&\;\;\;\;\;\;\;\;\;\;\;\;\;\;\;\;\;\;\;\;\;\;\;\;\;\;\;\;-v_{\hat{\alpha}} \delta_{f}^{\hat{\beta}} \delta_{\hat{f}}^{\alpha} w_{\beta}+v_{\hat{\alpha}} \delta_{f}^{\hat{\beta}} \delta_{\hat{f}}^{\beta} w_{\alpha}+v_{\hat{\alpha}} \delta_{f}^{\hat{f}} \delta_{\hat{\beta}}^{\alpha} w_{\beta}-v_{\hat{\alpha}} \delta_{f}^{\hat{f}} \delta_{\hat{\beta}}^{\beta} w_{\alpha}-v_{\hat{\alpha}} \delta_{\beta}^{\hat{f}} \delta_{\hat{\beta}}^{\alpha} w_{f}+v_{\hat{\alpha}} \delta_{\beta}^{\hat{\beta}} \delta_{\hat{f}}^{\alpha} w_{f}\nonumber\\
	&\;\;\;\;\;\;\;\;\;\;\;\;\;\;\;\;\;\;\;\;\;\;\;\;\;\;\;\;+v_{\hat{\beta
	}} \delta_{f}^{\hat{\alpha}} \delta_{\hat{f}}^{\alpha} w_{\beta}-v_{\hat{\beta}} \delta_{f}^{\hat{\alpha}} \delta_{\hat{f}}^{\beta} w_{\alpha}-v_{\hat{\beta}} \delta_{f}^{\hat{f}} \delta_{\hat{\alpha}}^{\alpha} w_{\beta}+v_{\hat{\beta}} \delta_{f}^{\hat{f}} \delta_{\hat{\alpha}}^{\beta} w_{\alpha}+v_{\hat{\beta}} \delta_{\beta}^{\hat{f}} \delta_{\hat{\alpha}}^{\alpha} w_{f}-v_{\hat{\beta}} \delta_{\beta}^{\hat{\alpha}} \delta_{\hat{f}}^{\alpha} w_{f}\nonumber\\
	&\;\;\;\;\;\;\;\;\;\;\;- \delta_{f}^{\hat{\beta}} \delta_{\hat{\alpha}}^{\alpha}\delta_{\hat{f}}^{\beta}g(v,w) +\delta_{f}^{\hat{\beta}} \delta_{\hat{f}}^{\alpha}\delta_{\hat{\alpha}}^{\beta} g(v,w)+ \delta_{f}^{\hat{\alpha}} \delta_{\hat{\beta}}^{\alpha}\delta_{\hat{f}}^{\beta} g(v,w)-\delta_{f}^{\hat{\alpha}}\delta_{\hat{f}}^{\alpha} \delta_{\hat{\beta}}^{\beta} g(v,w)-\delta_{f}^{\hat{f}}\delta_{\hat{\beta}}^{\alpha} \delta_{\hat{\alpha}}^{\beta} g(v,w)+\delta_{f}^{\hat{f}} \delta_{\hat{\alpha}}^{\alpha}\delta_{\hat{\beta}}^{\beta} g(v,w)\bigg]{\rm tr}[id].\nonumber
\end{align}

\begin{proof}

According to \eqref{trace8} in lemma \ref{A}, $g(e_{i}, e_{j})=\delta_i^j$, $g(v,e_i)=v_i$ and $ g(w,e_j)=w_j$,  we get
\begin{align}\label{At1}
	&{\rm tr} \bigg(\sum_{f\neq\alpha\neq\beta,\hat{f}\neq\hat{\alpha}\neq\hat{\beta}=1}^{2m}c(v)c(e_{f})c(e_{\alpha})c(e_{\beta})c(w)c(e_{\hat{f}})c(e_{\hat{\alpha}})c(e_{\hat{\beta}})\bigg)\\
	=&\sum_{f\neq\alpha\neq\beta,\hat{f}\neq\hat{\alpha}\neq\hat{\beta}=1}^{2m}\biggl\{-g(v,  e_{\hat{\beta}}){\rm tr}\bigg[c( e_f)c( e_\alpha)c( e_\beta)c( w)c( e_{\hat{f}})c( e_{\hat{\alpha}})\bigg]+g(v,  e_{\hat{\alpha}}){\rm tr}\bigg[c( e_f)c( e_\alpha)c( e_\beta)c( w)c( e_{\hat{f}})c( e_{\hat{\beta}})\bigg]\nonumber\\
	&\;\;\;\;\;\;\;\;\;\;\;\;\;\;\;\;\;\;\;\;\;\;\;\;\;\;-g(v,  e_{\hat{f}}){\rm tr}\bigg[c( e_f)c( e_\alpha)c( e_\beta)c( w)c( e_{\hat{\alpha}})c( e_{\hat{\beta}})\bigg]+g(v,  w){\rm tr}\bigg[c( e_f)c( e_\alpha)c( e_\beta)c( e_{\hat{f}})c( e_{\hat{\alpha}})c( e_{\hat{\beta}})\bigg]\nonumber\\
	&\;\;\;\;\;\;\;\;\;\;\;\;\;\;\;\;\;\;\;\;\;\;\;\;\;\;-g(v,  e_\beta){\rm tr}\bigg[c( e_f)c( e_\alpha)c( w)c( e_{\hat{f}})c( e_{\hat{\alpha}})c( e_{\hat{\beta}})\bigg]+g(v,  e_\alpha){\rm tr}\bigg[c( e_f)c( e_\beta)c( w)c( e_{\hat{f}})c( e_{\hat{\alpha}})c( e_{\hat{\beta}})\bigg]\nonumber\\
	&\;\;\;\;\;\;\;\;\;\;\;\;\;\;\;\;\;\;\;\;\;\;\;\;\;\;-g(v,  e_f){\rm tr}\bigg[c( e_\alpha)c( e_\beta)c( w)c( e_{\hat{f}})c( e_{\hat{\alpha}})c( e_{\hat{\beta}})\bigg]\biggr\}\nonumber\\
	&:=\sum_{f\neq\alpha\neq\beta,\hat{f}\neq\hat{\alpha}\neq\hat{\beta}=1}^{2m}\biggl\{-g(v,  e_{\hat{\beta}}){\mathcal{H}_1}+g(v,  e_{\hat{\alpha}}){\mathcal{H}_2}-g(v,  e_{\hat{f}}){\mathcal{H}_3}+g(v,  w){\mathcal{H}_4}\nonumber\\
	&\;\;\;\;\;\;\;\;\;\;\;\;\;\;\;\;\;\;\;\;\;\;\;\;\;\;-g(v,  e_\beta){\mathcal{H}_5}+g(v,  e_\alpha){\mathcal{H}_6}-g(v,  e_f){\mathcal{H}_7}\biggr\}.\nonumber
\end{align}
According to \eqref{trace6} in lemma \ref{A}, $g(e_{i}, e_{j})=\delta_i^j$, $g(v,e_i)=v_i$ and $ g(w,e_j)=w_j$,  we get
\begin{align}
{\mathcal{H}_1}=&\sum_{f\neq\alpha\neq\beta,\hat{f}\neq\hat{\alpha}=1}^{2m}{\rm tr}\bigg[c( e_f)c( e_\alpha)c( e_\beta)c( w)c( e_{\hat{f}})c( e_{\hat{\alpha}})\bigg]\label{h1}\\
&=\sum_{f\neq\alpha\neq\beta,\hat{f}\neq\hat{\alpha}=1}^{2m}\bigg[-g(e_f,  e_{\hat{\alpha}})g ( e_\alpha,  e_{\hat{f}})g ( e_\beta,  w)+g(e_f,  e_{\hat{\alpha}})g ( e_\alpha,  w)g ( e_\beta,  e_{\hat{f}})-g(e_f,  e_{\hat{\alpha}})g ( e_\alpha,  e_\beta)g ( w,  e_{\hat{f}})\nonumber\\
&\;\;\;\;\;\;\;\;\;\;\;\;\;\;\;\;\;\;\;\;\;\;\;\;\;\;+g(e_f, e_{\hat{f}})g ( e_\alpha,  e_{\hat{\alpha}})g ( e_\beta, w)-g(e_f, e_{\hat{f}})g ( e_\alpha,  w)g ( e_\beta, e_{\hat{\alpha}})+g(e_f, e_{\hat{f}})g ( e_\alpha, e_\beta)g ( w, e_{\hat{\alpha}})\nonumber\\
&\;\;\;\;\;\;\;\;\;\;\;\;\;\;\;\;\;\;\;\;\;\;\;\;\;\;-g(e_f, w)g ( e_\alpha, e_{\hat{\alpha}})g ( e_\beta, e_{\hat{f}})+g(e_f, w)g ( e_\alpha, e_{\hat{f}})g ( e_\beta, e_{\hat{\alpha}})-g(e_f, w)g ( e_\alpha, e_\beta)g ( e_{\hat{f}}, e_{\hat{\alpha}})\nonumber\\
&\;\;\;\;\;\;\;\;\;\;\;\;\;\;\;\;\;\;\;\;\;\;\;\;\;\;+g(e_f, e_\beta)g ( e_\alpha, e_{\hat{\alpha}})g ( w, e_{\hat{f}})-g(e_f, e_\beta)g ( e_\alpha, e_{\hat{f}})g ( w, e_{\hat{\alpha}})+g(e_f, e_\beta)g ( e_\alpha, w)g ( e_{\hat{f}}, e_{\hat{\alpha}})\nonumber\\
&\;\;\;\;\;\;\;\;\;\;\;\;\;\;\;\;\;\;\;\;\;-g(e_f, e_\alpha)g ( e_\beta, e_{\hat{\alpha}})g ( w, e_{\hat{f}})+g(e_f, e_\alpha)g ( e_\beta, e_{\hat{f}})g ( w, e_{\hat{\alpha}})-g(e_f, e_\alpha)g ( e_\beta, w)g ( e_{\hat{f}}, e_{\hat{\alpha}})\bigg]{\rm tr}[id]\nonumber\\
&=\sum_{f\neq\alpha\neq\beta,\hat{f}\neq\hat{\alpha}=1}^{2m}\bigg[-g(e_f,  e_{\hat{\alpha}})g ( e_\alpha,  e_{\hat{f}})g ( e_\beta,  w)+g(e_f,  e_{\hat{\alpha}})g ( e_\alpha,  w)g ( e_\beta,  e_{\hat{f}})\nonumber\\
&\;\;\;\;\;\;\;\;\;\;\;\;\;\;\;\;\;\;\;\;\;\;+g(e_f, e_{\hat{f}})g ( e_\alpha,  e_{\hat{\alpha}})g ( e_\beta, w)-g(e_f, e_{\hat{f}})g ( e_\alpha,  w)g ( e_\beta, e_{\hat{\alpha}})\nonumber\\
&\;\;\;\;\;\;\;\;\;\;\;\;\;\;\;\;\;\;\;\;\;\;-g(e_f, w)g ( e_\alpha, e_{\hat{\alpha}})g ( e_\beta, e_{\hat{f}})+g(e_f, w)g ( e_\alpha, e_{\hat{f}})g ( e_\beta, e_{\hat{\alpha}})\bigg]{\rm tr}[id]\nonumber\\
&=\sum_{f\neq\alpha\neq\beta,\hat{f}\neq\hat{\alpha}=1}^{2m}\bigg[- \delta_{f}^{\hat{\alpha}} \delta_{\hat{f}}^{\alpha} w_{\beta}+ \delta_{f}^{\hat{\alpha}} \delta_{\hat{f}}^{\beta} w_{\alpha}+ \delta_{f}^{\hat{f}} \delta_{\hat{\alpha}}^{\alpha} w_{\beta}- \delta_{f}^{\hat{f}} \delta_{\hat{\alpha}}^{\beta} w_{\alpha}- \delta_{\beta}^{\hat{f}} \delta_{\hat{\alpha}}^{\alpha} w_{f}+ \delta_{\beta}^{\hat{\alpha}} \delta_{\hat{f}}^{\alpha} w_{f}\bigg]{\rm tr}[id],\nonumber\\
	{\mathcal{H}_2}=&\sum_{f\neq\alpha\neq\beta,\hat{f}\neq\hat{\beta}=1}^{2m}{\rm tr}\bigg[c( e_f)c( e_\alpha)c( e_\beta)c( w)c( e_{\hat{f}})c( e_{\hat{\beta}})\bigg]\label{h2}\\
	&=\sum_{f\neq\alpha\neq\beta,\hat{f}\neq\hat{\beta}=1}^{2m}\bigg[ -\delta_{f}^{\hat{\beta}} \delta_{\hat{f}}^{\alpha} w_{\beta}+ \delta_{f}^{\hat{\beta}} \delta_{\hat{f}}^{\beta} w_{\alpha}+ \delta_{f}^{\hat{f}} \delta_{\hat{\beta}}^{\alpha} w_{\beta}- \delta_{f}^{\hat{f}} \delta_{\hat{\beta}}^{\beta} w_{\alpha}- \delta_{\beta}^{\hat{f}} \delta_{\hat{\beta}}^{\alpha} w_{f}+ \delta_{\beta}^{\hat{\beta}} \delta_{\hat{f}}^{\alpha} w_{f}\bigg]{\rm tr}[id],\nonumber\\
	{\mathcal{H}_3}=&\sum_{f\neq\alpha\neq\beta,\hat{\alpha}\neq\hat{\beta}=1}^{2m}{\rm tr}\bigg[c( e_f)c( e_\alpha)c( e_\beta)c( w)c( e_{\hat{\alpha}})c( e_{\hat{\beta}})\bigg]\label{h3}\\
	&=\sum_{f\neq\alpha\neq\beta,\hat{\alpha}\neq\hat{\beta}=1}^{2m}\bigg[ -\delta_{f}^{\hat{\beta}} \delta_{\hat{\alpha}}^{\alpha} w_{\beta}+ \delta_{f}^{\hat{\beta}} \delta_{\hat{\alpha}}^{\beta} w_{\alpha}+ \delta_{f}^{\hat{\alpha}} \delta_{\hat{\beta}}^{\alpha} w_{\beta}- \delta_{f}^{\hat{\alpha}} \delta_{\hat{\beta}}^{\beta} w_{\alpha}- \delta_{\beta}^{\hat{\alpha}} \delta_{\hat{\beta}}^{\alpha} w_{f}+ \delta_{\beta}^{\hat{\beta}} \delta_{\hat{\alpha}}^{\alpha} w_{f}\bigg]{\rm tr}[id],\nonumber
\end{align}

\begin{align}\label{h4}
	{\mathcal{H}_4}=&\sum_{f\neq\alpha\neq\beta,\hat{f}\neq\hat{\alpha}\neq\hat{\beta}=1}^{2m}{\rm tr}\bigg[c( e_f)c( e_\alpha)c( e_\beta)c( e_{\hat{f}})c( e_{\hat{\alpha}})c( e_{\hat{\beta}})\bigg]\\
	&=\sum_{f\neq\alpha\neq\beta,\hat{\alpha}\neq\hat{\beta}=1}^{2m}\bigg[-g(e_f,  e_{\hat{\beta}})g ( e_\alpha,  e_{\hat{\alpha}})g ( e_\beta,  e_{\hat{f}})+g(e_f,  e_{\hat{\beta}})g ( e_\alpha,  e_{\hat{f}})g ( e_\beta,  e_{\hat{\alpha}})-g(e_f,  e_{\hat{\beta}})g ( e_\alpha,  e_\beta)g ( e_{\hat{f}},  e_{\hat{\alpha}})\nonumber\\
	&\;\;\;\;\;\;\;\;\;\;\;\;\;\;\;\;\;\;\;\;\;\;\;\;\;\;+g(e_f, e_{\hat{\alpha}})g ( e_\alpha,  e_{\hat{\beta}})g ( e_\beta, e_{\hat{f}})-g(e_f, e_{\hat{\alpha}})g ( e_\alpha,  e_{\hat{f}})g ( e_\beta, e_{\hat{\beta}})+g(e_f, e_{\hat{\alpha}})g ( e_\alpha, e_\beta)g ( e_{\hat{f}}, e_{\hat{\beta}})\nonumber\\
	&\;\;\;\;\;\;\;\;\;\;\;\;\;\;\;\;\;\;\;\;\;\;\;\;\;\;-g(e_f, e_{\hat{f}})g ( e_\alpha, e_{\hat{\beta}})g ( e_\beta, e_{\hat{\alpha}})+g(e_f, e_{\hat{f}})g ( e_\alpha, e_{\hat{\alpha}})g ( e_\beta, e_{\hat{\beta}})-g(e_f, e_{\hat{f}})g ( e_\alpha, e_\beta)g ( e_{\hat{\alpha}}, e_{\hat{\beta}})\nonumber\\
	&\;\;\;\;\;\;\;\;\;\;\;\;\;\;\;\;\;\;\;\;\;\;\;\;\;\;+g(e_f, e_\beta)g ( e_\alpha, e_{\hat{\beta}})g ( e_{\hat{f}}, e_{\hat{\alpha}})-g(e_f, e_\beta)g ( e_\alpha, e_{\hat{\alpha}})g ( e_{\hat{f}}, e_{\hat{\beta}})+g(e_f, e_\beta)g ( e_\alpha, e_{\hat{f}})g ( e_{\hat{\alpha}}, e_{\hat{\beta}})\nonumber\\
	&\;\;\;\;\;\;\;\;\;\;\;\;\;\;\;\;\;\;\;\;\;-g(e_f, e_\alpha)g ( e_\beta, e_{\hat{\beta}})g ( e_{\hat{f}}, e_{\hat{\alpha}})+g(e_f, e_\alpha)g ( e_\beta, e_{\hat{\alpha}})g ( e_{\hat{f}}, e_{\hat{\beta}})-g(e_f, e_\alpha)g ( e_\beta, e_{\hat{f}})g ( e_{\hat{\alpha}}, e_{\hat{\beta}})\bigg]{\rm tr}[id]\nonumber\\
	&=\sum_{f\neq\alpha\neq\beta,\hat{\alpha}\neq\hat{\beta}=1}^{2m}\bigg[-g(e_f,  e_{\hat{\beta}})g ( e_\alpha,  e_{\hat{\alpha}})g ( e_\beta,  e_{\hat{f}})+g(e_f,  e_{\hat{\beta}})g ( e_\alpha,  e_{\hat{f}})g ( e_\beta,  e_{\hat{\alpha}})\nonumber\\
	&\;\;\;\;\;\;\;\;\;\;\;\;\;\;\;\;\;\;\;\;\;\;+g(e_f, e_{\hat{\alpha}})g ( e_\alpha,  e_{\hat{\beta}})g ( e_\beta, e_{\hat{f}})-g(e_f, e_{\hat{\alpha}})g ( e_\alpha,  e_{\hat{f}})g ( e_\beta, e_{\hat{\beta}})\nonumber\\
	&\;\;\;\;\;\;\;\;\;\;\;\;\;\;\;\;\;\;\;\;\;\;-g(e_f, e_{\hat{f}})g ( e_\alpha, e_{\hat{\beta}})g ( e_\beta, e_{\hat{\alpha}})+g(e_f, e_{\hat{f}})g ( e_\alpha, e_{\hat{\alpha}})g ( e_\beta, e_{\hat{\beta}})\bigg]{\rm tr}[id]\nonumber\\
	&=\sum_{f\neq\alpha\neq\beta,\hat{\alpha}\neq\hat{\beta}=1}^{2m}\bigg[ -\delta_{f}^{\hat{\beta}} \delta_{\hat{\alpha}}^{\alpha} \delta^{\hat{f}}_{\beta}+ \delta_{f}^{\hat{\beta}} \delta_{\hat{\alpha}}^{\beta} \delta^{\hat{f}}_{\alpha}+ \delta_{f}^{\hat{\alpha}} \delta_{\hat{\beta}}^{\alpha} \delta^{\hat{f}}_{\beta}- \delta_{f}^{\hat{\alpha}} \delta_{\hat{\beta}}^{\beta} \delta^{\hat{f}}_{\alpha}- \delta_{\beta}^{\hat{\alpha}} \delta_{\hat{\beta}}^{\alpha} \delta^{\hat{f}}_{f}+ \delta_{\beta}^{\hat{\beta}} \delta_{\hat{\alpha}}^{\alpha} \delta^{\hat{f}}_{f}\bigg]{\rm tr}[id],\nonumber
\end{align}

\begin{align}
	{\mathcal{H}_5}=&\sum_{f\neq\alpha,\hat{f}\neq\hat{\alpha}\neq\hat{\beta}=1}^{2m}{\rm tr}\bigg[c( e_f)c( e_\alpha)c( w)c( e_{\hat{f}})c( e_{\hat{\alpha}})c( e_{\hat{\beta}})\bigg]\label{h5}\\
	&=\sum_{f\neq\alpha,\hat{f}\neq\hat{\alpha}\neq\hat{\beta}=1}^{2m}\bigg[-g(e_f,  e_{\hat{\beta}})g ( e_\alpha,  e_{\hat{\alpha}})g ( w,  e_{\hat{f}})+g(e_f,  e_{\hat{\beta}})g ( e_\alpha,  e_{\hat{f}})g ( w,  e_{\hat{\alpha}})-g(e_f,  e_{\hat{\beta}})g ( e_\alpha,  w)g ( e_{\hat{f}},  e_{\hat{\alpha}})\nonumber\\
	&\;\;\;\;\;\;\;\;\;\;\;\;\;\;\;\;\;\;\;\;\;\;\;\;\;\;+g(e_f, e_{\hat{\alpha}})g ( e_\alpha,  e_{\hat{\beta}})g ( w, e_{\hat{f}})-g(e_f, e_{\hat{\alpha}})g ( e_\alpha,  e_{\hat{f}})g ( w, e_{\hat{\beta}})+g(e_f, e_{\hat{\alpha}})g ( e_\alpha, w)g ( e_{\hat{f}}, e_{\hat{\beta}})\nonumber\\
	&\;\;\;\;\;\;\;\;\;\;\;\;\;\;\;\;\;\;\;\;\;\;\;\;\;\;-g(e_f, e_{\hat{f}})g ( e_\alpha, e_{\hat{\beta}})g ( w, e_{\hat{\alpha}})+g(e_f, e_{\hat{f}})g ( e_\alpha, e_{\hat{\alpha}})g ( w, e_{\hat{\beta}})-g(e_f, e_{\hat{f}})g ( e_\alpha, w)g ( e_{\hat{\alpha}}, e_{\hat{\beta}})\nonumber\\
	&\;\;\;\;\;\;\;\;\;\;\;\;\;\;\;\;\;\;\;\;\;\;\;\;\;\;+g(e_f, w)g ( e_\alpha, e_{\hat{\beta}})g ( e_{\hat{f}}, e_{\hat{\alpha}})-g(e_f, w)g ( e_\alpha, e_{\hat{\alpha}})g ( e_{\hat{f}}, e_{\hat{\beta}})+g(e_f, w)g ( e_\alpha, e_{\hat{f}})g ( e_{\hat{\alpha}}, e_{\hat{\beta}})\nonumber\\
	&\;\;\;\;\;\;\;\;\;\;\;\;\;\;\;\;\;\;\;\;\;-g(e_f, e_\alpha)g ( w, e_{\hat{\beta}})g ( e_{\hat{f}}, e_{\hat{\alpha}})+g(e_f, e_\alpha)g ( w, e_{\hat{\alpha}})g ( e_{\hat{f}}, e_{\hat{\beta}})-g(e_f, e_\alpha)g ( w, e_{\hat{f}})g ( e_{\hat{\alpha}}, e_{\hat{\beta}})\bigg]{\rm tr}[id]\nonumber\\
	&=\sum_{f\neq\alpha,\hat{f}\neq\hat{\alpha}\neq\hat{\beta}=1}^{2m}\bigg[-g(e_f,  e_{\hat{\beta}})g ( e_\alpha,  e_{\hat{\alpha}})g ( w,  e_{\hat{f}})+g(e_f,  e_{\hat{\beta}})g ( e_\alpha,  e_{\hat{f}})g ( w,  e_{\hat{\alpha}})\nonumber\\
	&\;\;\;\;\;\;\;\;\;\;\;\;\;\;\;\;\;\;\;\;\;\;+g(e_f, e_{\hat{\alpha}})g ( e_\alpha,  e_{\hat{\beta}})g ( w, e_{\hat{f}})-g(e_f, e_{\hat{\alpha}})g ( e_\alpha,  e_{\hat{f}})g ( w, e_{\hat{\beta}})\nonumber\\
	&\;\;\;\;\;\;\;\;\;\;\;\;\;\;\;\;\;\;\;\;\;\;-g(e_f, e_{\hat{f}})g ( e_\alpha, e_{\hat{\beta}})g ( w, e_{\hat{\alpha}})+g(e_f, e_{\hat{f}})g ( e_\alpha, e_{\hat{\alpha}})g ( w, e_{\hat{\beta}})\bigg]{\rm tr}[id]\nonumber\\
	&=\sum_{f\neq\alpha,\hat{f}\neq\hat{\alpha}\neq\hat{\beta}=1}^{2m}\bigg[ -\delta_{f}^{\hat{\beta}} \delta_{\hat{\alpha}}^{\alpha} w_{\hat{f}}+ \delta_{f}^{\hat{\beta}} \delta^{\hat{f}}_{\alpha} w_{\hat{\alpha}}+ \delta_{f}^{\hat{\alpha}} \delta_{\hat{\beta}}^{\alpha} w_{\hat{f}}- \delta_{f}^{\hat{\alpha}}  \delta^{\hat{f}}_{\alpha}w_{\hat{\beta}}-  \delta_{\hat{\beta}}^{\alpha} \delta^{\hat{f}}_{f}w_{\hat{\alpha}}+  \delta_{\hat{\alpha}}^{\alpha} \delta^{\hat{f}}_{f}w_{\hat{\beta}}\bigg]{\rm tr}[id],\nonumber\\
	{\mathcal{H}_6}=&\sum_{f\neq\beta,\hat{f}\neq\hat{\alpha}\neq\hat{\beta}=1}^{2m}{\rm tr}\bigg[c( e_f)c( e_\beta)c( w)c( e_{\hat{f}})c( e_{\hat{\alpha}})c( e_{\hat{\beta}})\bigg]\label{h6}\\
	&=\sum_{f\neq\beta,\hat{f}\neq\hat{\alpha}\neq\hat{\beta}=1}^{2m}\bigg[ -\delta_{f}^{\hat{\beta}} \delta_{\hat{\alpha}}^{\beta} w_{\hat{f}}+ \delta_{f}^{\hat{\beta}} \delta^{\hat{f}}_{\beta} w_{\hat{\alpha}}+ \delta_{f}^{\hat{\alpha}} \delta_{\hat{\beta}}^{\beta} w_{\hat{f}}- \delta_{f}^{\hat{\alpha}}  \delta^{\hat{f}}_{\beta}w_{\hat{\beta}}-  \delta_{\hat{\beta}}^{\beta} \delta^{\hat{f}}_{f}w_{\hat{\alpha}}+  \delta_{\hat{\alpha}}^{\beta} \delta^{\hat{f}}_{f}w_{\hat{\beta}}\bigg]{\rm tr}[id],\nonumber\\
	{\mathcal{H}_7}=&\sum_{\alpha\neq\beta,\hat{f}\neq\hat{\alpha}\neq\hat{\beta}=1}^{2m}{\rm tr}\bigg[c( e_\alpha)c( e_\beta)c( w)c( e_{\hat{f}})c( e_{\hat{\alpha}})c( e_{\hat{\beta}})\bigg]\label{h7}\\
	&=\sum_{\alpha\neq\beta,\hat{f}\neq\hat{\alpha}\neq\hat{\beta}=1}^{2m}\bigg[ -\delta_{\alpha}^{\hat{\beta}} \delta_{\hat{\alpha}}^{\beta} w_{\hat{f}}+ \delta_{\alpha}^{\hat{\beta}} \delta^{\hat{f}}_{\beta} w_{\hat{\alpha}}+ \delta_{\alpha}^{\hat{\alpha}} \delta_{\hat{\beta}}^{\beta} w_{\hat{f}}- \delta_{\alpha}^{\hat{\alpha}}  \delta^{\hat{f}}_{\beta}w_{\hat{\beta}}-  \delta_{\hat{\beta}}^{\beta} \delta^{\hat{f}}_{\alpha}w_{\hat{\alpha}}+  \delta_{\hat{\alpha}}^{\beta} \delta^{\hat{f}}_{\alpha}w_{\hat{\beta}}\bigg]{\rm tr}[id],\nonumber
\end{align}

Bringing \eqref{h1}...\eqref{h7} into \eqref{At1} yields
\begin{align}\label{At1a}
	&{\rm tr} \bigg(\sum_{f\neq\alpha\neq\beta,\hat{f}\neq\hat{\alpha}\neq\hat{\beta}=1}^{2m}c(v)c(e_{f})c(e_{\alpha})c(e_{\beta})c(w)c(e_{\hat{f}})c(e_{\hat{\alpha}})c(e_{\hat{\beta}})\bigg)\\
	&=\sum_{f\neq\alpha\neq\beta,\hat{f}\neq\hat{\alpha}\neq\hat{\beta}=1}^{2m}\bigg[v_f \delta_{\alpha}^{\hat{\beta}} \delta_{\hat{\alpha}}^{\beta} w_{\hat{f}}-v_f \delta_{\alpha}^{\hat{\beta}} \delta_{\hat{f}}^{\beta} w_{\hat{\alpha}}-v_f \delta_{\alpha}^{\hat{\alpha}} \delta_{\hat{\beta}}^{\beta} w_{\hat{f}}+v_f \delta_{\alpha}^{\hat{\alpha}} \delta_{\hat{f}}^{\beta} w_{\hat{\beta}}+v_f \delta_{\alpha}^{\hat{f}} \delta_{\hat{\beta}}^{\beta} w_{\hat{\alpha}}-v_f \delta_{\alpha}^{\hat{f}} \delta_{\hat{\alpha}}^{\beta} w_{\hat{\beta}}\nonumber\\
	&\;\;\;\;\;\;\;\;\;\;\;\;\;\;\;\;\;\;\;\;\;\;\;\;\;\;\;\;-v_{\alpha} \delta_{f}^{\hat{\beta}} \delta_{\hat{\alpha}}^{\beta} w_{\hat{f}}+v_{\alpha} \delta_{f}^{\hat{\beta}} \delta_{\hat{f}}^{\beta} w_{\hat{\alpha}}+v_{\alpha} \delta_{f}^{\hat{\alpha}} \delta_{\hat{\beta}}^{\beta} w_{\hat{f}}-v_{\alpha} \delta_{f}^{\hat{\alpha}} \delta_{\hat{f}}^{\beta} w_{\hat{\beta}}-v_{\alpha} \delta_{f}^{\hat{f}} \delta_{\hat{\beta}}^{\beta} w_{\hat{\alpha}}+v_{\alpha} \delta_{f}^{\hat{f}} \delta_{\hat{\alpha}}^{\beta} w_{\hat{\beta}}\nonumber\\
	&\;\;\;\;\;\;\;\;\;\;\;\;\;\;\;\;\;\;\;\;\;\;\;\;\;\;\;\;+v_{\beta} \delta_{f}^{\hat{\beta}} \delta_{\hat{\alpha}}^{\alpha} w_{\hat{f}}-v_{\beta} \delta_{f}^{\hat{\beta}} \delta_{\hat{f}}^{\alpha} w_{\hat{\alpha}}-v_{\beta} \delta_{f}^{\hat{\alpha}} \delta_{\hat{\beta}}^{\alpha} w_{\hat{f}}+v_{\beta} \delta_{f}^{\hat{\alpha}} \delta_{\hat{f}}^{\alpha} w_{\hat{\beta}}+v_{\beta} \delta_{f}^{\hat{f}} \delta_{\hat{\beta}}^{\alpha} w_{\hat{\alpha}}-v_{\beta} \delta_{f}^{\hat{f}} \delta_{\hat{\alpha}}^{\alpha} w_{\hat{\beta}}\nonumber\\
	&\;\;\;\;\;\;\;\;\;\;\;\;\;\;\;\;\;\;\;\;\;\;\;\;\;\;\;\;+v_{\hat{f}} \delta_{f}^{\hat{\beta}} \delta_{\hat{\alpha}}^{\alpha} w_{\beta}-v_{\hat{f}} \delta_{f}^{\hat{\beta}} \delta_{\hat{\alpha}}^{\beta} w_{\alpha}-v_{\hat{f}} \delta_{f}^{\hat{\alpha}} \delta_{\hat{\beta}}^{\alpha} w_{\beta}+v_{\hat{f}} \delta_{f}^{\hat{\alpha}} \delta_{\hat{\beta}}^{\beta} w_{\alpha}+v_{\hat{f}} \delta_{\beta}^{\hat{\alpha}} \delta_{\hat{\beta}}^{\alpha} w_{f}-v_{\hat{f}} \delta_{\beta}^{\hat{\beta}} \delta_{\hat{\alpha}}^{\alpha} w_{f}\nonumber\\
	&\;\;\;\;\;\;\;\;\;\;\;\;\;\;\;\;\;\;\;\;\;\;\;\;\;\;\;\;-v_{\hat{\alpha}} \delta_{f}^{\hat{\beta}} \delta_{\hat{f}}^{\alpha} w_{\beta}+v_{\hat{\alpha}} \delta_{f}^{\hat{\beta}} \delta_{\hat{f}}^{\beta} w_{\alpha}+v_{\hat{\alpha}} \delta_{f}^{\hat{f}} \delta_{\hat{\beta}}^{\alpha} w_{\beta}-v_{\hat{\alpha}} \delta_{f}^{\hat{f}} \delta_{\hat{\beta}}^{\beta} w_{\alpha}-v_{\hat{\alpha}} \delta_{\beta}^{\hat{f}} \delta_{\hat{\beta}}^{\alpha} w_{f}+v_{\hat{\alpha}} \delta_{\beta}^{\hat{\beta}} \delta_{\hat{f}}^{\alpha} w_{f}\nonumber\\
	&\;\;\;\;\;\;\;\;\;\;\;\;\;\;\;\;\;\;\;\;\;\;\;\;\;\;\;\;+v_{\hat{\beta
	}} \delta_{f}^{\hat{\alpha}} \delta_{\hat{f}}^{\alpha} w_{\beta}-v_{\hat{\beta}} \delta_{f}^{\hat{\alpha}} \delta_{\hat{f}}^{\beta} w_{\alpha}-v_{\hat{\beta}} \delta_{f}^{\hat{f}} \delta_{\hat{\alpha}}^{\alpha} w_{\beta}+v_{\hat{\beta}} \delta_{f}^{\hat{f}} \delta_{\hat{\alpha}}^{\beta} w_{\alpha}+v_{\hat{\beta}} \delta_{\beta}^{\hat{f}} \delta_{\hat{\alpha}}^{\alpha} w_{f}-v_{\hat{\beta}} \delta_{\beta}^{\hat{\alpha}} \delta_{\hat{f}}^{\alpha} w_{f}\nonumber\\
	&\;\;\;\;\;\;\;\;\;\;\;- \delta_{f}^{\hat{\beta}} \delta_{\hat{\alpha}}^{\alpha}\delta_{\hat{f}}^{\beta}g(v,w) +\delta_{f}^{\hat{\beta}} \delta_{\hat{f}}^{\alpha}\delta_{\hat{\alpha}}^{\beta} g(v,w)+ \delta_{f}^{\hat{\alpha}} \delta_{\hat{\beta}}^{\alpha}\delta_{\hat{f}}^{\beta} g(v,w)-\delta_{f}^{\hat{\alpha}}\delta_{\hat{f}}^{\alpha} \delta_{\hat{\beta}}^{\beta} g(v,w)-\delta_{f}^{\hat{f}}\delta_{\hat{\beta}}^{\alpha} \delta_{\hat{\alpha}}^{\beta} g(v,w)+\delta_{f}^{\hat{f}} \delta_{\hat{\alpha}}^{\alpha}\delta_{\hat{\beta}}^{\beta} g(v,w)\bigg]{\rm tr}[id].\nonumber
\end{align}
Therefore, \eqref{t1} holds.
\end{proof}

\noindent{\bf (II-1-$\mathbb{A}$)}\eqref{t2}:
\begin{align}
	&{\rm tr}\bigg(\sum_{j\neq p,t\neq s=1}^{2m}c(v)c(e_{j})c(w)c(e_p)c(e_s)c(e_t)\bigg)\nonumber\\
	&=\sum_{j\neq p,t\neq s=1}^{2m}\bigg[-v_t w_p \delta_{j}^{s}-v_t w_j \delta_{p}^{s}+v_s w_p \delta_{j}^{t}+v_s w_j \delta_{p}^{t}-v_p w_s \delta_{j}^{t}+v_p w_t \delta_{j}^{s}\nonumber\\
	&\;\;\;\;\;\;\;\;\;\;\;\;\;\;\;\;\;-v_j w_t \delta_{p}^{s}+v_j w_s \delta_{p}^{t}+\delta_{j}^{t}\delta_{p}^{s}g(v,w)-\delta_{j}^{s}\delta_{p}^{t}g(v,w)\bigg]{\rm tr}[id].\nonumber
\end{align}

\begin{proof}

According to \eqref{trace6} in lemma \ref{A}, $g(e_{i}, e_{j})=\delta_i^j$, $g(v,e_i)=v_i$ and $ g(w,e_j)=w_j$,  we get
\begin{align}\label{At2}
	&{\rm tr}\bigg(\sum_{j\neq p,t\neq s=1}^{2m}c(v)c(e_{j})c(w)c(e_p)c(e_s)c(e_t)\bigg)\\
	&=\sum_{j\neq p,t\neq s=1}^{2m}\bigg[-g(v,  e_t)g ( e_j,  e_s)g ( w,  e_p)+g(v,  e_t)g ( e_j,  e_p)g ( w,  e_s)-g(v,  e_t)g ( e_j,  w)g ( e_p,  e_s)\nonumber\\
	&\;\;\;\;\;\;\;\;\;\;\;\;\;\;\;\;\;\;\;\;+g(v, e_s)g ( e_j,  e_t)g ( w, e_p)-g(v, e_s)g ( e_j,  e_p)g ( w, e_t)+g(v, e_s)g ( e_j, w)g ( e_p, e_t)\nonumber\\
	&\;\;\;\;\;\;\;\;\;\;\;\;\;\;\;\;\;\;\;\;-g(v, e_p)g ( e_j, e_t)g ( w, e_s)+g(v, e_p)g ( e_j, e_s)g ( w, e_t)-g(v, e_p)g ( e_j, w)g ( e_s, e_t)\nonumber\\
	&\;\;\;\;\;\;\;\;\;\;\;\;\;\;\;\;\;\;\;\;+g(v, w)g ( e_j, e_t)g ( e_p, e_s)-g(v, w)g ( e_j, e_s)g ( e_p, e_t)+g(v, w)g ( e_j, e_p)g ( e_s, e_t)\nonumber\\
	&\;\;\;\;\;\;\;\;\;\;\;\;\;\;\;\;\;\;\;\;-g(v, e_j)g ( w, e_t)g ( e_p, e_s)+g(v, e_j)g ( w, e_s)g ( e_p, e_t)-g(v, e_j)g ( w, e_p)g ( e_s, e_t)\bigg]{\rm tr}[id]\nonumber\\
	&=\sum_{j\neq p,t\neq s=1}^{2m}\bigg[-v_t w_p \delta_{j}^{s}-v_t w_j \delta_{p}^{s}+v_s w_p \delta_{j}^{t}+v_s w_j \delta_{p}^{t}-v_p w_s \delta_{j}^{t}+v_p w_t \delta_{j}^{s}\nonumber\\
	&\;\;\;\;\;\;\;\;\;\;\;\;\;\;\;\;\;-v_j w_t \delta_{p}^{s}+v_j w_s \delta_{p}^{t}+\delta_{j}^{t}\delta_{p}^{s}g(v,w)-\delta_{j}^{s}\delta_{p}^{t}g(v,w)\bigg]{\rm tr}[id],\nonumber
\end{align}

Therefore, \eqref{t2} holds.
\end{proof}

\noindent{\bf (II-1-$\mathbb{A}$)}\eqref{t2s}:

\begin{align}
	&\int_{\|\xi\|=1}\operatorname{tr}\biggl\{ \frac{1}{8}\|\xi\|^{-2m}\sum_{j,p,t,s=1}^{2m}  {\operatorname{R}}_{jpts}c(v)c(e_j)c(w)c(e_p)c(e_s)c(e_t) \biggr\}(x_0)\sigma(\xi)\nonumber\\
	&=\bigg(\frac{1}{4} s g(v,w)-\frac{1}{2}{\rm Ric}(v,w)\bigg){\rm tr}[id]{\rm Vol}(S^{n-1}).\nonumber
\end{align}

\begin{proof}
	Let ${\operatorname{R}}_{jpts}:={\operatorname{R}}(e_j,e_p,e_t,e_s),s:={\operatorname{R}}(e_j,e_p,e_j,e_p)$, and according to \eqref{t2} we get
	
\begin{align}\label{at2s}
	&\int_{\|\xi\|=1}\operatorname{tr}\biggl\{ \frac{1}{8}\|\xi\|^{-2m}\sum_{j,p,t,s=1}^{2m}  {\operatorname{R}}_{jpts}c(v)c(e_j)c(w)c(e_p)c(e_s)c(e_t) \biggr\}(x_0)\sigma(\xi)\\
	&=\frac{1}{8}\sum_{j,p,t,s=1}^{2m}{\operatorname{R}}(e_j,e_p,e_t,e_s)\bigg[-v_t w_p \delta_{j}^{s}-v_t w_j \delta_{p}^{s}+v_s w_p \delta_{j}^{t}+v_s w_j \delta_{p}^{t}-v_p w_s \delta_{j}^{t}+v_p w_t \delta_{j}^{s}\nonumber\\
	&\;\;\;\;\;\;\;\;\;\;\;\;\;\;\;\;\;-v_j w_t \delta_{p}^{s}+v_j w_s \delta_{p}^{t}+\delta_{j}^{t}\delta_{p}^{s}g(v,w)-\delta_{j}^{s}\delta_{p}^{t}g(v,w)\bigg]{\rm tr}[id]{\rm Vol}(S^{n-1})\nonumber\\
	&=\frac{1}{8}\sum_{j,p=1}^{2m}\bigg[-{\operatorname{R}}(e_j,w,v,e_j)-{\operatorname{R}}(w,e_p,v,e_p)+{\operatorname{R}}(e_j,w,e_j,v)+{\operatorname{R}}(e_j,v,w,e_j)\nonumber\\
	&\;\;\;\;\;\;\;\;\;\;\;\;\;\;\;\;-{\operatorname{R}}(e_j,v,e_j,w)+{\operatorname{R}}(e_j,v,w,e_j)-{\operatorname{R}}(v,e_p,w,e_p)+{\operatorname{R}}(v,e_p,e_p,w)\nonumber\\
	&\;\;\;\;\;\;\;\;\;\;\;\;\;\;\;\;+{\operatorname{R}}(e_j,e_p,e_j,e_p)g(v,w)-{\operatorname{R}}(e_j,e_p,e_p,e_j)g(v,w)\bigg]{\rm tr}[id]{\rm Vol}(S^{n-1})\nonumber\\
	&=\frac{1}{8}\sum_{j,p=1}^{2m}\bigg[2{\operatorname{R}}(e_j,e_p,e_j,e_p)g(v,w)-4{\operatorname{R}}(e_j,w,e_j,v)\bigg]{\rm tr}[id]{\rm Vol}(S^{n-1})\nonumber\\
	&=\bigg(\frac{1}{4} s g(v,w)-\frac{1}{2}{\rm Ric}(v,w)\bigg){\rm tr}[id]{\rm Vol}(S^{n-1}).\nonumber
\end{align}
	Therefore, \eqref{t2s} holds.
\end{proof}

\noindent{\bf (II-2-$\mathbb{A}$)}\eqref{t5}:

\begin{align}
	&{\rm tr}\bigg(\sum_{f\neq{\alpha}\neq{\beta},k\neq l \neq \eta=1}^{2m}c(v)c(e_{\eta})c(w)c(e_f)c(e_{\alpha})c(e_{\beta})c(e_k)c(e_l)+c(v)c(e_f)c(e_{\alpha})c(e_{\beta})c(w)c(e_{\eta})c(e_k)c(e_l)\bigg)\nonumber\\
	&=\sum_{f\neq{\alpha}\neq{\beta},k\neq l \neq \eta=1}^{2m}\biggl\{v_{\beta}\delta_{f}^{l} \delta_{k}^{\alpha} w_{\eta}-v_{\beta}\delta_{f}^{l} \delta_{\eta}^{\alpha} w_{k}-v_{\beta}\delta_{f}^{k} \delta_{l}^{\alpha} w_{\eta}+v_{\beta}\delta_{f}^{k} \delta_{\eta}^{\alpha} w_{l}+v_{\beta}\delta_{f}^{\eta} \delta_{l}^{\alpha} w_{k}-v_{\beta}\delta_{f}^{\eta} \delta_{k}^{\alpha} w_{l}\nonumber\\
	&\;\;\;\;\;\;\;\;\;\;\;\;\;\;\;\;\;\;\;\;\;\;\;-v_{\alpha}\delta_{f}^{l} \delta_{k}^{\beta} w_{\eta}+v_{\alpha}\delta_{f}^{l} \delta_{\eta}^{\beta} w_{k}+v_{\alpha}\delta_{f}^{k} \delta_{l}^{\beta} w_{\eta}-v_{\alpha}\delta_{f}^{k} \delta_{\eta}^{\beta} w_{l}-v_{\alpha}\delta_{f}^{\eta} \delta_{l}^{\beta} w_{k}+v_{\alpha}\delta_{f}^{\eta} \delta_{k}^{\beta} w_{l}\nonumber\\
	&\;\;\;\;\;\;\;\;\;\;\;\;\;\;\;\;\;\;\;\;\;\;\;+v_{f}\delta_{\alpha}^{l} \delta_{k}^{\beta} w_{\eta}-v_{f}\delta_{\alpha}^{l} \delta_{\eta}^{\beta} w_{k}-v_{f}\delta_{\alpha}^{k} \delta_{l}^{\beta} w_{\eta}+v_{f}\delta_{\alpha}^{k} \delta_{\eta}^{\beta} w_{l}+v_{f}\delta_{\alpha}^{\eta} \delta_{l}^{\beta} w_{k}-v_{f}\delta_{\alpha}^{\eta} \delta_{k}^{\beta} w_{l}\nonumber\\
	&\;\;\;\;\;\;\;\;\;\;\;\;\;\;\;\;\;\;\;\;\;\;\;+v_{\eta}\delta_{\alpha}^{l} \delta_{k}^{\beta} w_{f}-v_{\eta}\delta_{\alpha}^{l} \delta_{f}^{k} w_{\beta}-v_{\eta}\delta_{\alpha}^{k} \delta_{l}^{\beta} w_{f}+v_{\eta}\delta_{\alpha}^{k} \delta_{f}^{l} w_{\beta}+v_{\eta} \delta_{l}^{\beta} \delta_{f}^{k}w_{\alpha}-v_{\eta} \delta_{k}^{\beta} \delta_{f}^{l}w_{\alpha}\nonumber\\
	&\;\;\;\;\;\;\;\;\;\;\;\;\;\;\;\;\;\;\;\;\;\;\;\bigg[-\delta_{\alpha}^{l} \delta_{k}^{\beta} \delta_{f}^{\eta}+\delta_{\alpha}^{l} \delta_{\eta}^{\beta} \delta_{f}^{k}+\delta_{\alpha}^{k} \delta_{l}^{\beta} \delta_{f}^{\eta}-\delta_{\alpha}^{k} \delta_{\eta}^{\beta} \delta_{f}^{l}-\delta_{\alpha}^{\eta} \delta_{l}^{\beta} \delta_{f}^{k}+\delta_{\alpha}^{\eta} \delta_{k}^{\beta} \delta_{f}^{l}\bigg]g(v,w)\biggr\}{\rm tr}[id].\nonumber
\end{align}

\begin{proof}

According to \eqref{trace8} in lemma \ref{A}, $g(e_{i}, e_{j})=\delta_i^j$, $g(v,e_i)=v_i$ and $ g(w,e_j)=w_j$,  we get
\begin{align}\label{At5}
	&{\rm tr}\bigg(\sum_{f\neq{\alpha}\neq{\beta},k\neq l \neq \eta=1}^{2m}c(v)c(e_{\eta})c(w)c(e_f)c(e_{\alpha})c(e_{\beta})c(e_k)c(e_l)+c(v)c(e_f)c(e_{\alpha})c(e_{\beta})c(w)c(e_{\eta})c(e_k)c(e_l)\bigg)\\
    &={\rm tr}\bigg(\sum_{f\neq{\alpha}\neq{\beta},k\neq l \neq \eta=1}^{2m}c(v)c(e_{\eta})c(w)c(e_f)c(e_{\alpha})c(e_{\beta})c(e_k)c(e_l)\bigg)\nonumber\\
    &\;\;\;\;+{\rm tr}\bigg(\sum_{f\neq{\alpha}\neq{\beta},k\neq l \neq \eta=1}^{2m}c(v)c(e_f)c(e_{\alpha})c(e_{\beta})c(w)c(e_{\eta})c(e_k)c(e_l)\bigg)\nonumber\\
    &=\sum_{f\neq{\alpha}\neq{\beta},k\neq l \neq \eta=1}^{2m}\biggl\{-g(v,  e_l){\rm tr}\bigg[c( {e_\eta})c( w)c( e_f)c( e_{\alpha})c( e_{\beta})c( e_k)\bigg]+g(v,  e_k){\rm tr}\bigg[c( {e_\eta})c( w)c( e_f)c( e_{\alpha})c( e_{\beta})c( e_l)\bigg]\nonumber\\
    &\;\;\;\;\;\;\;\;\;\;\;\;\;\;\;\;\;\;\;\;-g(v,  e_{\beta}){\rm tr}\bigg[c( {e_\eta})c( w)c( e_f)c( e_{\alpha})c( e_k)c( e_l)\bigg]+g(v,  e_{\alpha}){\rm tr}\bigg[c( {e_\eta})c( w)c( e_f)c( e_{\beta})c( e_k)c( e_l)\bigg]\nonumber\\
    &\;\;\;\;\;\;\;\;\;\;\;\;\;\;\;\;\;\;\;\;-g(v,  e_f){\rm tr}\bigg[c( {e_\eta})c( w)c( e_{\alpha})c( e_{\beta})c( e_k)c( e_l)\bigg]+g(v,  w){\rm tr}\bigg[c( {e_\eta})c( e_f)c( e_{\alpha})c( e_{\beta})c( e_k)c( e_l)\bigg]\nonumber\\
    &\;\;\;\;\;\;\;\;\;\;\;\;\;\;\;\;\;\;\;\;-g(v,  {e_\eta}){\rm tr}\bigg[c( w)c( e_f)c( e_{\alpha})c( e_{\beta})c( e_k)c( e_l)\bigg]\biggr\}\nonumber \\
    &+\sum_{f\neq{\alpha}\neq{\beta},k\neq l \neq \eta=1}^{2m}\biggl\{-g(v,  e_l){\rm tr}\bigg[c( e_f)c( e_{\alpha})c( e_{\beta})c( w)c( e_{\eta})c( e_k)\bigg]+g(v,  e_k){\rm tr}\bigg[c( e_f)c( e_{\alpha})c( e_{\beta})c( w)c( e_{\eta})c( e_l)\bigg]\nonumber\\
    &\;\;\;\;\;\;\;\;\;\;\;\;\;\;\;\;\;\;\;\;-g(v,  e_{\eta}){\rm tr}\bigg[c( e_f)c( e_{\alpha})c( e_{\beta})c( w)c( e_k)c( e_l)\bigg]+g(v,  w){\rm tr}\bigg[c( e_f)c( e_{\alpha})c( e_{\beta})c( e_{\eta})c( e_k)c( e_l)\bigg]\nonumber\\
    &\;\;\;\;\;\;\;\;\;\;\;\;\;\;\;\;\;\;\;\;-g(v,  e_{\beta}){\rm tr}\bigg[c( e_f)c( e_{\alpha})c( w)c( e_{\eta})c( e_k)c( e_l)\bigg]+g(v,  e_{\alpha}){\rm tr}\bigg[c( e_f)c( e_{\beta})c( w)c( e_{\eta})c( e_k)c( e_l)\bigg]\nonumber\\
    &\;\;\;\;\;\;\;\;\;\;\;\;\;\;\;\;\;\;\;\;-g(v,  e_f){\rm tr}\bigg[c( e_{\alpha})c( e_{\beta})c( w)c( e_{\eta})c( e_k)c( e_l)\bigg]\biggr\}\nonumber\\
    &:=\sum_{f\neq{\alpha}\neq{\beta},k\neq l \neq \eta=1}^{2m}\biggl\{-g(v,  e_l){\mathcal{K}_1}+g(v,  e_k){\rm tr}{\mathcal{K}_2}-g(v,  e_{\beta}){\mathcal{K}_3}+g(v,  e_{\alpha}){\mathcal{K}_4}-g(v,  e_f){\mathcal{K}_5}+g(v,  w){\mathcal{K}_6}-g(v,  {e_\eta}){\mathcal{K}_7}\biggr\}\nonumber \\
    &+\sum_{f\neq{\alpha}\neq{\beta},k\neq l \neq \eta=1}^{2m}\biggl\{-g(v,  e_l){\mathcal{K}_8}+g(v,  e_k){\mathcal{K}_9}-g(v,  e_{\eta}){\mathcal{K}_{10}}+g(v,  w){\mathcal{K}_{11}}-g(v,  e_{\beta}){\mathcal{K}_{12}}+g(v,  e_{\alpha}){\mathcal{K}_{13}}-g(v,  e_f){\mathcal{K}_{14}}\biggr\}.\nonumber
\end{align}

According to ${\rm tr}\mathcal{XY}={\rm tr}\mathcal{YX}$, $c(e_{i})c(e_{j})+c(e_{j})c(e_{i})=-2g(e_{i}, e_{j})=-2\delta_i^j$ and \eqref{trace4} in lemma \ref{A}, we get
\begin{align}
	{\mathcal{K}_{1}}&=\sum_{f\neq{\alpha}\neq{\beta},k\neq \eta=1}^{2m}{\rm tr}\bigg[c( {e_\eta})c( w)c( e_f)c( e_{\alpha})c( e_{\beta})c( e_k)\bigg]\label{k1}\\	
	&=\sum_{f\neq{\alpha}\neq{\beta},k\neq \eta=1}^{2m}{\rm tr}\bigg[c( e_f)c( e_{\alpha})c( e_{\beta})c( e_k)c( {e_\eta})c( w)\bigg]\nonumber\\
	&=-\sum_{f\neq{\alpha}\neq{\beta},k\neq  \eta=1}^{2m}{\rm tr}\bigg[c( e_f)c( e_{\alpha})c( e_{\beta})c( e_k)c( w)c( {e_\eta})\bigg]-2g(w,e_\eta)\sum_{f\neq{\alpha}\neq{\beta},k=1}^{2m}{\rm tr}\bigg[c( e_f)c( e_{\alpha})c( e_{\beta})c( e_k)\bigg]\nonumber\\
	&=-\sum_{f\neq{\alpha}\neq{\beta},k\neq  \eta=1}^{2m}{\rm tr}\bigg[c( e_f)c( e_{\alpha})c( e_{\beta})c( e_k)c( w)c( {e_\eta})\bigg]\nonumber\\
	&=\sum_{f\neq{\alpha}\neq{\beta},k\neq  \eta=1}^{2m}{\rm tr}\bigg[c( e_f)c( e_{\alpha})c( e_{\beta})c( w)c( e_k)c( {e_\eta})\bigg]+2g(w,e_k)\sum_{f\neq{\alpha}\neq{\beta},\eta=1}^{2m}{\rm tr}\bigg[c( e_f)c( e_{\alpha})c( e_{\beta})c( e_\eta)\bigg]\nonumber\\
	&=\sum_{f\neq{\alpha}\neq{\beta},k\neq  \eta=1}^{2m}{\rm tr}\bigg[c( e_f)c( e_{\alpha})c( e_{\beta})c( w)c( e_k)c( {e_\eta})\bigg]\nonumber\\
	&=-\sum_{f\neq{\alpha}\neq{\beta},k\neq  \eta=1}^{2m}{\rm tr}\bigg[c( e_f)c( e_{\alpha})c( e_{\beta})c( w)c( {e_\eta})c( e_k)\bigg]=-{\mathcal{K}_{8}}.\nonumber
\end{align}
Hence obtaining ${\mathcal{K}_{1}}=-{\mathcal{K}_{8}}$, similar calculations yield
\begin{align} &{\mathcal{K}_{2}}=-{\mathcal{K}_{9}},\nonumber\\ &{\mathcal{K}_{3}}={\mathcal{K}_{10}},\nonumber\\  &{\mathcal{K}_{4}}={\mathcal{K}_{13}},\nonumber\\  &{\mathcal{K}_{5}}={\mathcal{K}_{14}},\nonumber\\  &{\mathcal{K}_{6}}={\mathcal{K}_{11}},\nonumber\\  &{\mathcal{K}_{7}}={\mathcal{K}_{10}}.\nonumber
\end{align}

According to \eqref{trace6} in lemma \ref{A}, $g(e_{i}, e_{j})=\delta_i^j$, $g(v,e_i)=v_i$ and $ g(w,e_j)=w_j$,  we get
\begin{align}\label{k3} 
	{\mathcal{K}_{3}}=&\sum_{f\neq{\alpha},k\neq l \neq  \eta=1}^{2m}{\rm tr}\bigg[c( {e_\eta})c( w)c( e_f)c( e_{\alpha})c( e_k)c( e_l)\bigg]\\
	&=\sum_{f\neq{\alpha},k\neq l \neq  \eta=1}^{2m}\bigg[-g(e_{\eta},  e_l)g ( w,  e_k)g ( e_f,  e_{\alpha})+g(e_{\eta},  e_l)g ( w,  e_{\alpha})g ( e_f,  e_k)-g(e_{\eta},  e_l)g ( w,  e_f)g ( e_{\alpha},  e_k)\nonumber\\
	&\;\;\;\;\;\;\;\;\;\;\;\;\;\;\;\;\;\;\;\;+g(e_{\eta}, e_k)g ( w,  e_l)g ( e_f, e_{\alpha})-g(e_{\eta}, e_k)g ( w,  e_{\alpha})g ( e_f, e_l)+g(e_{\eta}, e_k)g ( w, e_f)g ( e_{\alpha}, e_l)\nonumber\\
	&\;\;\;\;\;\;\;\;\;\;\;\;\;\;\;\;\;\;\;\;-g(e_{\eta}, e_{\alpha})g ( w, e_l)g ( e_f, e_k)+g(e_{\eta}, e_{\alpha})g ( w, e_k)g ( e_f, e_l)-g(e_{\eta}, e_{\alpha})g ( w, e_f)g ( e_k, e_l)\nonumber\\
	&\;\;\;\;\;\;\;\;\;\;\;\;\;\;\;\;\;\;\;\;+g(e_{\eta}, e_f)g ( w, e_l)g ( e_{\alpha}, e_k)-g(e_{\eta}, e_f)g ( w, e_k)g ( e_{\alpha}, e_l)+g(e_{\eta}, e_f)g ( w, e_{\alpha})g ( e_k, e_l)\nonumber\\
	&\;\;\;\;\;\;\;\;\;\;\;\;\;\;\;\;\;\;\;\;-g(e_{\eta}, w)g ( e_f, e_l)g ( e_{\alpha}, e_k)+g(e_{\eta}, w)g ( e_f, e_k)g ( e_{\alpha}, e_l)-g(e_{\eta}, w)g ( e_f, e_{\alpha})g ( e_k, e_l)\bigg]{\rm tr}[id]\nonumber\\
	&=\sum_{f\neq{\alpha},k\neq l \neq  \eta=1}^{2m}\bigg[-\delta_{f}^{l} \delta_{k}^{\alpha} w_{\eta}+\delta_{f}^{l} \delta_{\eta}^{\alpha} w_{k}+\delta_{f}^{k} \delta_{l}^{\alpha} w_{\eta}-\delta_{f}^{k} \delta_{\eta}^{\alpha} w_{l}-\delta_{f}^{\eta} \delta_{l}^{\alpha} w_{k}+\delta_{f}^{\eta} \delta_{k}^{\alpha} w_{l}\bigg]{\rm tr}[id].\nonumber
\end{align}

The same reasoning can be used to obtain
\begin{align} \label{k4}
	{\mathcal{K}_{4}}=&\sum_{f\neq{\beta},k\neq l \neq  \eta=1}^{2m}{\rm tr}\bigg[c( {e_\eta})c( w)c( e_f)c( e_{\beta})c( e_k)c( e_l)\bigg]\\
	&=\sum_{f\neq{\beta},k\neq l \neq  \eta=1}^{2m}\bigg[-\delta_{f}^{l} \delta_{k}^{\beta} w_{\eta}+\delta_{f}^{l} \delta_{\eta}^{\beta} w_{k}+\delta_{f}^{k} \delta_{l}^{\beta} w_{\eta}-\delta_{f}^{k} \delta_{\eta}^{\beta} w_{l}-\delta_{f}^{\eta} \delta_{l}^{\beta} w_{k}+\delta_{f}^{\eta} \delta_{k}^{\beta} w_{l}\bigg]{\rm tr}[id],\nonumber
\end{align}
\begin{align} \label{k5}
	{\mathcal{K}_{5}}=&\sum_{\alpha\neq{\beta},k\neq l \neq  \eta=1}^{2m}{\rm tr}\bigg[c( {e_\eta})c( w)c( e_\alpha)c( e_{\beta})c( e_k)c( e_l)\bigg]\\
	&=\sum_{\alpha\neq{\beta},k\neq l \neq  \eta=1}^{2m}\bigg[-\delta_{\alpha}^{l} \delta_{k}^{\beta} w_{\eta}+\delta_{\alpha}^{l} \delta_{\eta}^{\beta} w_{k}+\delta_{\alpha}^{k} \delta_{l}^{\beta} w_{\eta}-\delta_{\alpha}^{k} \delta_{\eta}^{\beta} w_{l}-\delta_{\alpha}^{\eta} \delta_{l}^{\beta} w_{k}+\delta_{\alpha}^{\eta} \delta_{k}^{\beta} w_{l}\bigg]{\rm tr}[id],\nonumber
\end{align}
\begin{align} \label{k6}
	{\mathcal{K}_{6}}=&\sum_{f \neq \alpha\neq{\beta},k\neq l \neq  \eta=1}^{2m}{\rm tr}\bigg[c( {e_\eta})c( e_f)c( e_\alpha)c( e_{\beta})c( e_k)c( e_l)\bigg]\\
	&=\sum_{\alpha\neq{\beta},k\neq l \neq  \eta=1}^{2m}\bigg[-\delta_{\alpha}^{l} \delta_{k}^{\beta} \delta_{f}^{\eta}+\delta_{\alpha}^{l} \delta_{\eta}^{\beta} \delta_{f}^{k}+\delta_{\alpha}^{k} \delta_{l}^{\beta} \delta_{f}^{\eta}-\delta_{\alpha}^{k} \delta_{\eta}^{\beta} \delta_{f}^{l}-\delta_{\alpha}^{\eta} \delta_{l}^{\beta} \delta_{f}^{k}+\delta_{\alpha}^{\eta} \delta_{k}^{\beta} \delta_{f}^{l}\bigg]{\rm tr}[id],\nonumber
\end{align}
\begin{align} \label{k7}
	{\mathcal{K}_{7}}=&\sum_{f \neq \alpha\neq{\beta},k\neq l=1}^{2m}{\rm tr}\bigg[c( w)c( e_f)c( e_\alpha)c( e_{\beta})c( e_k)c( e_l)\bigg]\\
	&=\sum_{\alpha\neq{\beta},k\neq l \neq  \eta=1}^{2m}\bigg[-\delta_{\alpha}^{l} \delta_{k}^{\beta} w_{f}+\delta_{\alpha}^{l} \delta_{f}^{k} w_{\beta}+\delta_{\alpha}^{k} \delta_{l}^{\beta} w_{f}-\delta_{\alpha}^{k} \delta_{f}^{l} w_{\beta}- \delta_{l}^{\beta} \delta_{f}^{k}w_{\alpha}+ \delta_{k}^{\beta} \delta_{f}^{l}w_{\alpha}\bigg]{\rm tr}[id],\nonumber
\end{align}
Bringing \eqref{k3}...\eqref{k7} into \eqref{At5} yields
\begin{align}\label{At5a}
	&{\rm tr}\bigg(\sum_{f\neq{\alpha}\neq{\beta},k\neq l \neq \eta=1}^{2m}c(v)c(e_{\eta})c(w)c(e_f)c(e_{\alpha})c(e_{\beta})c(e_k)c(e_l)+c(v)c(e_f)c(e_{\alpha})c(e_{\beta})c(w)c(e_{\eta})c(e_k)c(e_l)\bigg)\\
	&=\sum_{f\neq{\alpha}\neq{\beta},k\neq l \neq \eta=1}^{2m}\biggl\{v_{\beta}\delta_{f}^{l} \delta_{k}^{\alpha} w_{\eta}-v_{\beta}\delta_{f}^{l} \delta_{\eta}^{\alpha} w_{k}-v_{\beta}\delta_{f}^{k} \delta_{l}^{\alpha} w_{\eta}+v_{\beta}\delta_{f}^{k} \delta_{\eta}^{\alpha} w_{l}+v_{\beta}\delta_{f}^{\eta} \delta_{l}^{\alpha} w_{k}-v_{\beta}\delta_{f}^{\eta} \delta_{k}^{\alpha} w_{l}\nonumber\\
	&\;\;\;\;\;\;\;\;\;\;\;\;\;\;\;\;\;\;\;\;\;\;\;-v_{\alpha}\delta_{f}^{l} \delta_{k}^{\beta} w_{\eta}+v_{\alpha}\delta_{f}^{l} \delta_{\eta}^{\beta} w_{k}+v_{\alpha}\delta_{f}^{k} \delta_{l}^{\beta} w_{\eta}-v_{\alpha}\delta_{f}^{k} \delta_{\eta}^{\beta} w_{l}-v_{\alpha}\delta_{f}^{\eta} \delta_{l}^{\beta} w_{k}+v_{\alpha}\delta_{f}^{\eta} \delta_{k}^{\beta} w_{l}\nonumber\\
	&\;\;\;\;\;\;\;\;\;\;\;\;\;\;\;\;\;\;\;\;\;\;\;+v_{f}\delta_{\alpha}^{l} \delta_{k}^{\beta} w_{\eta}-v_{f}\delta_{\alpha}^{l} \delta_{\eta}^{\beta} w_{k}-v_{f}\delta_{\alpha}^{k} \delta_{l}^{\beta} w_{\eta}+v_{f}\delta_{\alpha}^{k} \delta_{\eta}^{\beta} w_{l}+v_{f}\delta_{\alpha}^{\eta} \delta_{l}^{\beta} w_{k}-v_{f}\delta_{\alpha}^{\eta} \delta_{k}^{\beta} w_{l}\nonumber\\
	&\;\;\;\;\;\;\;\;\;\;\;\;\;\;\;\;\;\;\;\;\;\;\;+v_{\eta}\delta_{\alpha}^{l} \delta_{k}^{\beta} w_{f}-v_{\eta}\delta_{\alpha}^{l} \delta_{f}^{k} w_{\beta}-v_{\eta}\delta_{\alpha}^{k} \delta_{l}^{\beta} w_{f}+v_{\eta}\delta_{\alpha}^{k} \delta_{f}^{l} w_{\beta}+v_{\eta} \delta_{l}^{\beta} \delta_{f}^{k}w_{\alpha}-v_{\eta} \delta_{k}^{\beta} \delta_{f}^{l}w_{\alpha}\nonumber\\
	&\;\;\;\;\;\;\;\;\;\;\;\;\;\;\;\;\;\;\;\;\;\;\;\bigg[-\delta_{\alpha}^{l} \delta_{k}^{\beta} \delta_{f}^{\eta}+\delta_{\alpha}^{l} \delta_{\eta}^{\beta} \delta_{f}^{k}+\delta_{\alpha}^{k} \delta_{l}^{\beta} \delta_{f}^{\eta}-\delta_{\alpha}^{k} \delta_{\eta}^{\beta} \delta_{f}^{l}-\delta_{\alpha}^{\eta} \delta_{l}^{\beta} \delta_{f}^{k}+\delta_{\alpha}^{\eta} \delta_{k}^{\beta} \delta_{f}^{l}\bigg]g(v,w)\biggr\}{\rm tr}[id].\nonumber
\end{align}
Therefore, \eqref{t5} holds.
\end{proof}

\noindent{\bf (II-3-$\mathbb{B}$)}\eqref{t70}:

\begin{align}
	&{\rm tr}\bigg(\sum_{f,j\neq l,\hat{j}\neq\hat{l}=1}^{2m}c(v)c(e_f)c(w)c(e_f)c(e_j)c(e_l)c(e_{\hat{j}})c(e_{\hat{l}})\bigg)\nonumber\\
	&=2\sum_{f,j\neq l,\hat{j}\neq\hat{l}=1}^{2m}\biggl\{m\bigg[v_{\hat{l}}w_{l}\delta^{\hat{j}}_{j}-v_{\hat{l}}w_{j}\delta^{\hat{j}}_{l}-v_{\hat{j}}w_{l}\delta^{\hat{l}}_{j}+v_{\hat{j}}w_{j}\delta^{\hat{l}}_{j}-v_{l}w_{\hat{l}}\delta^{\hat{j}}_{j}+v_{l}w_{\hat{j}}\delta^{\hat{l}}_{j}+v_{j}w_{\hat{l}}\delta^{\hat{j}}_{l}-v_{j}w_{\hat{j}}\delta^{\hat{l}}_{l}\nonumber\\
	&\;\;\;\;\;\;\;\;\;\;\;\;\;\;\;\;\;\;-\delta^{\hat{l}}_{j}\delta^{\hat{j}}_{l}g(v,w)+\delta^{\hat{j}}_{j}\delta^{\hat{l}}_{l}g(v,w)\bigg]-w_{f}v_{\hat{l}}\delta^{f}_{l}\delta^{\hat{j}}_{j}+w_{f}v_{\hat{l}}\delta^{f}_{j}\delta^{\hat{j}}_{l}+w_{f}v_{\hat{j}}\delta^{f}_{l}\delta^{\hat{l}}_{j}-w_{f}v_{\hat{j}}\delta^{f}_{j}\delta^{\hat{l}}_{j}\nonumber\\
	&\;\;\;\;\;\;\;\;\;\;\;\;\;\;\;\;\;\;+w_{f}v_{l}\delta^{f}_{\hat{l}}\delta^{\hat{j}}_{j}-w_{f}v_{l}\delta^{f}_{\hat{j}}\delta^{\hat{l}}_{j}-w_{f}v_{j}\delta^{f}_{\hat{l}}\delta^{\hat{j}}_{l}+w_{f}v_{j}\delta^{f}_{\hat{j}}\delta^{\hat{l}}_{l}+w_{f}v_{f}\delta^{\hat{l}}_{j}\delta^{\hat{j}}_{l}-w_{f}v_{f}\delta^{\hat{j}}_{j}\delta^{\hat{l}}_{l}\biggr\}{\rm tr}[id].\nonumber
\end{align}

\begin{proof}

Similarly, use $c(e_{i})c(e_{j})+c(e_{j})c(e_{i})=-2g(e_{i}, e_{j})=-2\delta_i^j$, exchange of $c(e_f)$ and $c(w)$, we get
\begin{align}\label{At70}
	&{\rm tr}\bigg(\sum_{f,j\neq l,\hat{j}\neq\hat{l}=1}^{2m}c(v)c(e_f)c(w)c(e_f)c(e_j)c(e_l)c(e_{\hat{j}})c(e_{\hat{l}})\bigg)\\
	&=\sum_{f,j\neq l,\hat{j}\neq\hat{l}=1}^{2m}\bigg[{\rm tr}\bigg(c(v)c(w)c(e_f)c(e_f)c(e_j)c(e_l)c(e_{\hat{j}})c(e_{\hat{l}})\bigg) -2g(w,e_f){\rm tr}\bigg(c(v)c(e_f)c(e_j)c(e_l)c(e_{\hat{j}})c(e_{\hat{l}})\bigg)\bigg]\nonumber\\
	&=\sum_{f,j\neq l,\hat{j}\neq\hat{l}=1}^{2m}\bigg[\delta_{f}^{f}{\rm tr}\bigg(c(v)c(w)c(e_j)c(e_l)c(e_{\hat{j}})c(e_{\hat{l}})\bigg) -2w_f{\rm tr}\bigg(c(v)c(e_f)c(e_j)c(e_l)c(e_{\hat{j}})c(e_{\hat{l}})\bigg)\bigg]\nonumber\\
	&=\sum_{f,j\neq l,\hat{j}\neq\hat{l}=1}^{2m}\bigg[2m{\rm tr}\bigg(c(v)c(w)c(e_j)c(e_l)c(e_{\hat{j}})c(e_{\hat{l}})\bigg) -2w_f{\rm tr}\bigg(c(v)c(e_f)c(e_j)c(e_l)c(e_{\hat{j}})c(e_{\hat{l}})\bigg)\bigg].\nonumber
\end{align}

According to \eqref{trace6} in lemma \ref{A}, $g(e_{i}, e_{j})=\delta_i^j$, $g(v,e_i)=v_i$ and $ g(w,e_j)=w_j$,  we get
\begin{align}\label{At70a1}
	&\sum_{j\neq l,\hat{j}\neq\hat{l}=1}^{2m}{\rm tr}\bigg(c(v)c(w)c(e_j)c(e_l)c(e_{\hat{j}})c(e_{\hat{l}})\bigg)\\
	&=\sum_{j\neq l,\hat{j}\neq\hat{l}=1}^{2m}\bigg[-g(v,  e_{\hat{l}})g ( w,  e_{\hat{j}})g ( e_j,  e_l)+g(v,  e_{\hat{l}})g ( w,  e_l)g ( e_j,  e_{\hat{j}})-g(v,  e_{\hat{l}})g ( w,  e_j)g ( e_l,  e_{\hat{j}})\nonumber\\
	&\;\;\;\;\;\;\;\;\;\;\;\;\;\;\;\;\;\;+g(v, e_{\hat{j}})g ( w,  e_{\hat{l}})g ( e_j, e_l)-g(v, e_{\hat{j}})g ( w,  e_l)g ( e_j, e_{\hat{l}})+g(v, e_{\hat{j}})g ( w, e_j)g ( e_l, e_{\hat{l}})\nonumber\\
	&\;\;\;\;\;\;\;\;\;\;\;\;\;\;\;\;\;\;-g(v, e_l)g ( w, e_{\hat{l}})g ( e_j, e_{\hat{j}})+g(v, e_l)g ( w, e_{\hat{j}})g ( e_j, e_{\hat{l}})-g(v, e_l)g ( w, e_j)g ( e_{\hat{j}}, e_{\hat{l}})\nonumber\\
	&\;\;\;\;\;\;\;\;\;\;\;\;\;\;\;\;\;\;+g(v, e_j)g ( w, e_{\hat{l}})g ( e_l, e_{\hat{j}})-g(v, e_j)g ( w, e_{\hat{j}})g ( e_l, e_{\hat{l}})+g(v, e_j)g ( w, e_l)g ( e_{\hat{j}}, e_{\hat{l}})\nonumber\\
	&\;\;\;\;\;\;\;\;\;\;\;\;\;\;\;\;\;\;-g(v, w)g ( e_j, e_{\hat{l}})g ( e_l, e_{\hat{j}})+g(v, w)g ( e_j, e_{\hat{j}})g ( e_l, e_{\hat{l}})-g(v, w)g ( e_j, e_l)g ( e_{\hat{j}}, e_{\hat{l}})\bigg]{\rm tr}[id]\nonumber\\
    &=\sum_{j\neq l,\hat{j}\neq\hat{l}=1}^{2m}\bigg[v_{\hat{l}}w_{l}\delta^{\hat{j}}_{j}-v_{\hat{l}}w_{j}\delta^{\hat{j}}_{l}-v_{\hat{j}}w_{l}\delta^{\hat{l}}_{j}+v_{\hat{j}}w_{j}\delta^{\hat{l}}_{j}-v_{l}w_{\hat{l}}\delta^{\hat{j}}_{j}+v_{l}w_{\hat{j}}\delta^{\hat{l}}_{j}+v_{j}w_{\hat{l}}\delta^{\hat{j}}_{l}-v_{j}w_{\hat{j}}\delta^{\hat{l}}_{l}\nonumber\\
    &\;\;\;\;\;\;\;\;\;\;\;\;\;\;\;\;\;\;-\delta^{\hat{l}}_{j}\delta^{\hat{j}}_{l}g(v,w)+\delta^{\hat{j}}_{j}\delta^{\hat{l}}_{l}g(v,w)\bigg]{\rm tr}[id],\nonumber
\end{align}
and
\begin{align}\label{At70a2}
	&\sum_{j\neq l,\hat{j}\neq\hat{l}=1}^{2m}{\rm tr}\bigg(c(v)c(e_f)c(e_j)c(e_l)c(e_{\hat{j}})c(e_{\hat{l}})\bigg)\\
	&=\sum_{j\neq l,\hat{j}\neq\hat{l}=1}^{2m}\bigg[-g(v,  e_{\hat{l}})g ( e_f,  e_{\hat{j}})g ( e_j,  e_l)+g(v,  e_{\hat{l}})g ( e_f,  e_l)g ( e_j,  e_{\hat{j}})-g(v,  e_{\hat{l}})g ( e_f,  e_j)g ( e_l,  e_{\hat{j}})\nonumber\\
	&\;\;\;\;\;\;\;\;\;\;\;\;\;\;\;\;\;\;+g(v, e_{\hat{j}})g ( e_f,  e_{\hat{l}})g ( e_j, e_l)-g(v, e_{\hat{j}})g ( e_f,  e_l)g ( e_j, e_{\hat{l}})+g(v, e_{\hat{j}})g ( e_f, e_j)g ( e_l, e_{\hat{l}})\nonumber\\
	&\;\;\;\;\;\;\;\;\;\;\;\;\;\;\;\;\;\;-g(v, e_l)g ( e_f, e_{\hat{l}})g ( e_j, e_{\hat{j}})+g(v, e_l)g ( e_f, e_{\hat{j}})g ( e_j, e_{\hat{l}})-g(v, e_l)g ( e_f, e_j)g ( e_{\hat{j}}, e_{\hat{l}})\nonumber\\
	&\;\;\;\;\;\;\;\;\;\;\;\;\;\;\;\;\;\;+g(v, e_j)g ( e_f, e_{\hat{l}})g ( e_l, e_{\hat{j}})-g(v, e_j)g ( e_f, e_{\hat{j}})g ( e_l, e_{\hat{l}})+g(v, e_j)g ( e_f, e_l)g ( e_{\hat{j}}, e_{\hat{l}})\nonumber\\
	&\;\;\;\;\;\;\;\;\;\;\;\;\;\;\;\;\;\;-g(v, e_f)g ( e_j, e_{\hat{l}})g ( e_l, e_{\hat{j}})+g(v, e_f)g ( e_j, e_{\hat{j}})g ( e_l, e_{\hat{l}})-g(v, e_f)g ( e_j, e_l)g ( e_{\hat{j}}, e_{\hat{l}})\bigg]{\rm tr}[id]\nonumber\\
	&=\sum_{j\neq l,\hat{j}\neq\hat{l}=1}^{2m}\bigg[v_{\hat{l}}\delta^{f}_{l}\delta^{\hat{j}}_{j}-v_{\hat{l}}\delta^{f}_{j}\delta^{\hat{j}}_{l}-v_{\hat{j}}\delta^{f}_{l}\delta^{\hat{l}}_{j}+v_{\hat{j}}\delta^{f}_{j}\delta^{\hat{l}}_{j}-v_{l}\delta^{f}_{\hat{l}}\delta^{\hat{j}}_{j}+v_{l}\delta^{f}_{\hat{j}}\delta^{\hat{l}}_{j}+v_{j}\delta^{f}_{\hat{l}}\delta^{\hat{j}}_{l}-v_{j}\delta^{f}_{\hat{j}}\delta^{\hat{l}}_{l}\nonumber\\
	&\;\;\;\;\;\;\;\;\;\;\;\;\;\;\;\;\;\;-v_{f}\delta^{\hat{l}}_{j}\delta^{\hat{j}}_{l}+v_{f}\delta^{\hat{j}}_{j}\delta^{\hat{l}}_{l}\bigg]{\rm tr}[id].\nonumber
\end{align}

Bringing \eqref{At70a1} and \eqref{At70a2} into \eqref{At70} gives
\begin{align}\label{At70b}
	&{\rm tr}\bigg(\sum_{f,j\neq l,\hat{j}\neq\hat{l}=1}^{2m}c(v)c(e_f)c(w)c(e_f)c(e_j)c(e_l)c(e_{\hat{j}})c(e_{\hat{l}})\bigg)\\
	&=2\sum_{f,j\neq l,\hat{j}\neq\hat{l}=1}^{2m}\biggl\{m\bigg[v_{\hat{l}}w_{l}\delta^{\hat{j}}_{j}-v_{\hat{l}}w_{j}\delta^{\hat{j}}_{l}-v_{\hat{j}}w_{l}\delta^{\hat{l}}_{j}+v_{\hat{j}}w_{j}\delta^{\hat{l}}_{j}-v_{l}w_{\hat{l}}\delta^{\hat{j}}_{j}+v_{l}w_{\hat{j}}\delta^{\hat{l}}_{j}+v_{j}w_{\hat{l}}\delta^{\hat{j}}_{l}-v_{j}w_{\hat{j}}\delta^{\hat{l}}_{l}\nonumber\\
	&\;\;\;\;\;\;\;\;\;\;\;\;\;\;\;\;\;\;-\delta^{\hat{l}}_{j}\delta^{\hat{j}}_{l}g(v,w)+\delta^{\hat{j}}_{j}\delta^{\hat{l}}_{l}g(v,w)\bigg]-w_{f}v_{\hat{l}}\delta^{f}_{l}\delta^{\hat{j}}_{j}+w_{f}v_{\hat{l}}\delta^{f}_{j}\delta^{\hat{j}}_{l}+w_{f}v_{\hat{j}}\delta^{f}_{l}\delta^{\hat{l}}_{j}-w_{f}v_{\hat{j}}\delta^{f}_{j}\delta^{\hat{l}}_{j}\nonumber\\
	&\;\;\;\;\;\;\;\;\;\;\;\;\;\;\;\;\;\;+w_{f}v_{l}\delta^{f}_{\hat{l}}\delta^{\hat{j}}_{j}-w_{f}v_{l}\delta^{f}_{\hat{j}}\delta^{\hat{l}}_{j}-w_{f}v_{j}\delta^{f}_{\hat{l}}\delta^{\hat{j}}_{l}+w_{f}v_{j}\delta^{f}_{\hat{j}}\delta^{\hat{l}}_{l}+w_{f}v_{f}\delta^{\hat{l}}_{j}\delta^{\hat{j}}_{l}-w_{f}v_{f}\delta^{\hat{j}}_{j}\delta^{\hat{l}}_{l}\biggr\}{\rm tr}[id].\nonumber
\end{align}
Therefore, \eqref{t70} holds.
\end{proof}

\noindent{\bf (II-3-$\mathbb{B}$)}\eqref{t8}:
\begin{align}
	&{\rm tr}\sum_{a,b,j\neq l,\hat{j}\neq\hat{l}=1}^{2m}\bigg(c(v)c(e_a)c(w)c(e_b)c(e_j)c(e_l)c(e_{\hat{j}})c(e_{\hat{l}})+c(v)c(e_b)c(w)c(e_a)c(e_j)c(e_l)c(e_{\hat{j}})c(e_{\hat{l}})\bigg)\nonumber\\
	&=\sum_{a,b,j\neq l,\hat{j}\neq\hat{l}=1}^{2m}\bigg[-2g(e_a,w){\rm tr}\bigg(c(v)c(e_b)c(e_j)c(e_l)c(e_{\hat{j}})c(e_{\hat{l}})\bigg)+2g(e_a,e_b){\rm tr}\bigg(c(v)c(w)c(e_j)c(e_l)c(e_{\hat{j}})c(e_{\hat{l}})\bigg)\bigg]\nonumber\\
	& \;\;\;\;-2g(w,e_b){\rm tr}\bigg(c(v)c(e_a)c(e_j)c(e_l)c(e_{\hat{j}})c(e_{\hat{l}})\bigg)\bigg].\nonumber
\end{align}

\begin{proof}

Use $c(e_{i})c(e_{j})+c(e_{j})c(e_{i})=-2g(e_{i}, e_{j})=-2\delta_i^j$, we get
\begin{align}\label{t8a}
	&{\rm tr}\bigg(\sum_{a,b,j\neq l,\hat{j}\neq\hat{l}=1}^{2m}c(v)c(e_a)c(w)c(e_b)c(e_j)c(e_l)c(e_{\hat{j}})c(e_{\hat{l}})\bigg)\\
	&=\sum_{a,b,j\neq l,\hat{j}\neq\hat{l}=1}^{2m}\bigg[-{\rm tr}\bigg(c(v)c(w)c(e_a)c(e_b)c(e_j)c(e_l)c(e_{\hat{j}})c(e_{\hat{l}})\bigg)-2g(e_a,w){\rm tr}\bigg(c(v)c(e_b)c(e_j)c(e_l)c(e_{\hat{j}})c(e_{\hat{l}})\bigg)\bigg]\nonumber\\
	&=\sum_{a,b,j\neq l,\hat{j}\neq\hat{l}=1}^{2m}\bigg[{\rm tr}\bigg(c(v)c(w)c(e_b)c(e_a)c(e_j)c(e_l)c(e_{\hat{j}})c(e_{\hat{l}})\bigg)+2g(e_a,e_b){\rm tr}\bigg(c(v)c(w)c(e_j)c(e_l)c(e_{\hat{j}})c(e_{\hat{l}})\bigg)\nonumber\\
	&\;\;\;\;-2g(e_a,w){\rm tr}\bigg(c(v)c(e_b)c(e_j)c(e_l)c(e_{\hat{j}})c(e_{\hat{l}})\bigg)\bigg]\nonumber\\
	&=\sum_{a,b,j\neq l,\hat{j}\neq\hat{l}=1}^{2m}\bigg[-{\rm tr}\bigg(c(v)c(e_b)c(w)c(e_a)c(e_j)c(e_l)c(e_{\hat{j}})c(e_{\hat{l}})\bigg)-2g(w,e_b){\rm tr}\bigg(c(v)c(e_a)c(e_j)c(e_l)c(e_{\hat{j}})c(e_{\hat{l}})\bigg)\nonumber\\
	&\;\;\;\; \;\;\;\; \;\;\;\; \;\;\;\;\;\;\;+2g(e_a,e_b){\rm tr}\bigg(c(v)c(w)c(e_j)c(e_l)c(e_{\hat{j}})c(e_{\hat{l}})\bigg)-2g(e_a,w){\rm tr}\bigg(c(v)c(e_b)c(e_j)c(e_l)c(e_{\hat{j}})c(e_{\hat{l}})\bigg)\bigg],\nonumber
\end{align}
Based on \eqref{At70a1} \eqref{At70a2} and \eqref{t8a}, we get
\begin{align}
	&{\rm tr}\sum_{a,b,j\neq l,\hat{j}\neq\hat{l}=1}^{2m}\bigg(c(v)c(e_a)c(w)c(e_b)c(e_j)c(e_l)c(e_{\hat{j}})c(e_{\hat{l}})+c(v)c(e_b)c(w)c(e_a)c(e_j)c(e_l)c(e_{\hat{j}})c(e_{\hat{l}})\bigg)\\
	&=\sum_{a,b,j\neq l,\hat{j}\neq\hat{l}=1}^{2m}\bigg[-2g(e_a,w){\rm tr}\bigg(c(v)c(e_b)c(e_j)c(e_l)c(e_{\hat{j}})c(e_{\hat{l}})\bigg)+2g(e_a,e_b){\rm tr}\bigg(c(v)c(w)c(e_j)c(e_l)c(e_{\hat{j}})c(e_{\hat{l}})\bigg)\bigg]\nonumber\\
	& \;\;\;\; \;\;\;\; \;\;\;\; \;\;\;\;\;\;\;-2g(w,e_b){\rm tr}\bigg(c(v)c(e_a)c(e_j)c(e_l)c(e_{\hat{j}})c(e_{\hat{l}})\bigg)\bigg]\nonumber\\
	&=2\sum_{a,b,j\neq l,\hat{j}\neq\hat{l}=1}^{2m}\bigg[-w_{a}v_{\hat{l}}\delta^{b}_{l}\delta^{\hat{j}}_{j}+w_{a}v_{\hat{l}}\delta^{b}_{j}\delta^{\hat{j}}_{l}+w_{a}v_{\hat{j}}\delta^{b}_{l}\delta^{\hat{l}}_{j}-w_{a}v_{\hat{j}}\delta^{b}_{j}\delta^{\hat{l}}_{j}+w_{a}v_{l}\delta^{b}_{\hat{l}}\delta^{\hat{j}}_{j}-w_{a}v_{l}\delta^{b}_{\hat{j}}\delta^{\hat{l}}_{j}-w_{a}v_{j}\delta^{b}_{\hat{l}}\delta^{\hat{j}}_{l}\nonumber\\
	&\;\;\;\;\;\;\;\;\;\;\;\;\;\;\;\;\;\;\;+w_{a}v_{j}\delta^{b}_{\hat{j}}\delta^{\hat{l}}_{l}+w_{a}v_{b}\delta^{\hat{l}}_{j}\delta^{\hat{j}}_{l}-w_{a}v_{b}\delta^{\hat{j}}_{j}\delta^{\hat{l}}_{l}+v_{\hat{l}}w_{l}\delta^{\hat{j}}_{j}\delta^{a}_{b}-v_{\hat{l}}w_{j}\delta^{\hat{j}}_{l}\delta^{a}_{b}-v_{\hat{j}}w_{l}\delta^{\hat{l}}_{j}\delta^{a}_{b}+v_{\hat{j}}w_{j}\delta^{\hat{l}}_{j}\delta^{a}_{b}\nonumber\\
	&\;\;\;\;\;\;\;\;\;\;\;\;\;\;\;\;\;\;\;-v_{l}w_{\hat{l}}\delta^{\hat{j}}_{j}\delta^{a}_{b}+v_{l}w_{\hat{j}}\delta^{\hat{l}}_{j}\delta^{a}_{b}+v_{j}w_{\hat{l}}\delta^{\hat{j}}_{l}-v_{j}w_{\hat{j}}\delta^{\hat{l}}_{l}\delta^{a}_{b}-\delta^{\hat{l}}_{j}\delta^{\hat{j}}_{l}\delta^{a}_{b}g(v,w)+\delta^{\hat{j}}_{j}\delta^{\hat{l}}_{l}\delta^{a}_{b}g(v,w)\nonumber\\
	&\;\;\;\;\;\;\;\;\;\;\;\;\;\;\;\;\;\;\;-w_{b}v_{\hat{l}}\delta^{a}_{l}\delta^{\hat{j}}_{j}+w_{b}v_{\hat{l}}\delta^{a}_{j}\delta^{\hat{j}}_{l}+w_{b}v_{\hat{j}}\delta^{a}_{l}\delta^{\hat{l}}_{j}-w_{b}v_{\hat{j}}\delta^{a}_{j}\delta^{\hat{l}}_{j}+w_{b}v_{l}\delta^{a}_{\hat{l}}\delta^{\hat{j}}_{j}-w_{b}v_{l}\delta^{a}_{\hat{j}}\delta^{\hat{l}}_{j}-w_{b}v_{j}\delta^{a}_{\hat{l}}\delta^{\hat{j}}_{l}\nonumber\\
	&\;\;\;\;\;\;\;\;\;\;\;\;\;\;\;\;\;\;+w_{b}v_{j}\delta^{a}_{\hat{j}}\delta^{\hat{l}}_{l}+w_{b}v_{a}\delta^{\hat{l}}_{j}\delta^{\hat{j}}_{l}-w_{b}v_{a}\delta^{\hat{j}}_{j}\delta^{\hat{l}}_{l}\bigg]{\rm tr}[id].\nonumber
\end{align}
Therefore, \eqref{t8} holds.
\end{proof}

\section*{ Acknowledgements}
This work is sponsored by Natural Science Foundation of Xinjiang Uygur Autonomous Region 2024D01C341 and supported by the National Natural Science Foundation of China 11771070, 12061078 .
 The authors thank the referee for his (or her) careful reading and helpful comments.

\section*{Declarations}

\begin{itemize}
	\item  {\bf Funding} The Natural Science Foundation of Xinjiang Uygur Autonomous Region 2024D01C341 and the National Natural Science Foundation of China 11771070, 12061078.
	\item {\bf Conflict of interest} The authors declare no Confict of interest.
	\item {\bf Ethics approval and consent to participate} Ethics approval and consent to participate.
	\item {\bf Consent for publication} All the authors agreed to publish this research.
	\item {\bf Author contribution} All authors contributed to the study conception and design. Material preparation, data collection and analysis were performed by JH and YW. The frst draft of the manuscript was written by JH and all authors commented on previous versions of the manuscript. All authors read and approved the fnal manuscript.
\end{itemize}

\end{document}